\documentclass[a4paper,12pt,intlimits,oneside,reqno]{amsart}
\usepackage{amsmath}
\usepackage{amsthm}
\usepackage{latexsym}
\usepackage{amssymb}
\usepackage{mathrsfs}
\usepackage{xcolor}
\usepackage{enumitem}
\usepackage[colorlinks=true,
]{hyperref}
\numberwithin{equation}{section}
\numberwithin{figure}{section}
\def\R{{\mathbb R}}
\def\C{{\mathbb C}}
\def\T{{\mathbb T}}
\def\Z{{\mathbb Z}}
\newtheorem{theorem}{Theorem}[section]
\newtheorem{corollary}{Corollary}[section]
\newtheorem{proposition}{Proposition}[section]
\newtheorem{lemma}{Lemma}[section]
\newtheorem{remark}{Remark}[section]

\newtheorem{assumption}{Assumption}[section]
\begin{document}
\title[Energy growth and modified scattering]{Unbounded Sobolev trajectories and modified scattering theory for a wave guide nonlinear Schr\"odinger equation}
\author{Haiyan XU}
\address{Universit\'e Paris-Sud XI, Laboratoire de Math\'ematiques
d'Orsay, CNRS, UMR 8628} \email{{\tt Haiyan.xu@math.u-psud.fr}}
\thanks{This work was supported by grants from R\'egion Ile-de-France (RDMath - IdF). }

\begin{abstract}
We consider the following wave guide nonlinear Schr\"odinger equation,
\begin{equation}
(i\partial _t+\partial _{xx}-\vert D_y\vert )U=\vert U\vert ^2U\   \tag{WS}
\end{equation}
on the spatial cylinder $\R _x\times \T _y$. We establish a modified scattering theory between small solutions to this equation and small solutions to the cubic Szeg\H{o} equation. The proof is an adaptation of the method of Hani--Pausader--Tzvetkov--Visciglia \cite{HPTV}. Combining this scattering  theory with a recent result by G\'erard--Grellier \cite{GG2015arxiv},  we infer existence of global solutions to (WS) which are unbounded in the space $L^2_xH^s_y(\R \times \T )$ for every $s>\frac 12$. 
\end{abstract}

\maketitle

\section{Introduction}
The purpose of this work is to study the large time behavior of solutions to the following Hamiltonian equation. On the cylinder $\R_x\times \T _y$, consider the Hilbert space $\mathcal H=L^2(\R \times \T )$ with the symplectic form
$$\omega (u,v)={\mathrm{Im}}(u|v)\ $$
and the Hamiltonian function on $\mathcal H$,
$$H(U)=\frac 12 \int _{\R \times \T } (\vert\partial _xU(x,y)\vert ^2+\vert D_y\vert U(x,y)\overline U(x,y))\, dx\, dy +\frac 14 \int _{\R \times \T }\vert U(x,y)\vert ^4\, dx\, dy \ ,$$
where $\vert D_y\vert :=\sqrt{-\partial_{yy}}$. The corresponding Hamiltonian system turns out to be a wave guide nonlinear Schr\"odinger equation,
\begin{equation}\label{HWS}
\left(i\partial_t +\mathcal{A}\right) U = \vert U\vert^2U,\ (x,y)\in \R\times\T\ ,
\end{equation}
where we set
$$\mathcal{A}:=\partial_{xx}-\vert D_y\vert\ .$$
Notice that, besides the energy $H(U)$, this equation formally enjoys the mass conservation law
$$\int _{\R \times \T}\vert U(t,x,y)\vert ^2\, dx\, dy =\int _{\R \times \T}\vert U(0,x,y)\vert ^2\, dx\, dy \ .$$
In particular, the trajectories are  bounded in $H^1_xL^2_y \cap
 L^2_xH^{\frac 12}_y$. These conservation laws correspond to a critical regularity for equation (\ref{HWS}), so that global wellposedness of the Cauchy problem is not easy. In this paper, we shall prove that global solutions do exist for every Cauchy datum satisfying a smallness assumption in an appropriate high regularity norm. However, our main objective in this paper is to study the possible large time unboundedness  of the solution,  in a slightly more regular norm than the energy norm, typically $L^2_xH^s_y(\R \times \T )$ for $s>\frac 12$. This general question of existence of unbounded Sobolev trajectories comes back to \cite{BourgainGAFA}, and was addressed by several authors for various Hamiltonian PDEs, see e.g. \cite{CKSTT2010,GGHW,Guardia,GK2015,Hani,HPTV,HaniThomann,HausProcesi,Pocovnicu2011DCDS,XU2014APDE}.
The choice of the equation (\ref{HWS}) is naturally based on the state of the art for this question concerning the nonlinear Schr\"odinger equation and the cubic half wave equation, which we recall in the next paragraphs.
\subsection{Motivation}
In this paragraph, we briefly recall the state of the art about the existence of unbounded Sobolev trajectories for the nonlinear Schr\"odinger equation and the cubic half wave equation. 
\subsubsection{The nonlinear Schr\"odinger equation}
Firstly, consider the  following Schr\"odinger equation with smooth initial data
\begin{equation}\label{NLScubique}
i\partial_t u +\Delta u=|u|^2u\ .
\end{equation}
If we consider the case with spatial domain $\R$ or $\T$, the $1D$ Schr\"odinger turns out to be globally well-posed and completely integrable \cite{ZS1971}, and the higher conservation laws
in that case imply 
$$\Vert u(t)\Vert_{H^s}\leq C_s(\Vert u(0)\Vert_{H^s})\ ,\ s \geq 1\ , \text{ for all }t \in \R\ .$$
Hani--Pausader--Tzvetkov--Visciglia studied the nonlinear Schr\"odinger on the cylinder $\R_x\times\T_y^d$ \cite{HPTV}, they found infinite cascade solutions for $d\geq2$, which means there exists solutions with small Sobolev norms at the initial time, while admit infinite Sobolev norms when time goes to infinity.
\begin{theorem}\cite[Corollary 1.4]{HPTV}
Let $d\geq 2$ and $s\in \mathbb{N}$, $s\ge 30$. Then for every $\varepsilon>0$ there exists a global solution $U(t)$ of the cubic Schr\"odinger equation \eqref{NLScubique} on $\R\times\T^d$, such that
\begin{equation}
\Vert U(0)\Vert_{H^s(\mathbb{R}\times\mathbb{T}^d)}\le\varepsilon,\qquad\limsup_{t\to+\infty}\Vert U(t)\Vert_{H^s(\mathbb{R}\times\mathbb{T}^d)}=+\infty.
\end{equation}
\end{theorem}

Unfortunately, these infinite cascades do not occur for $d=1$, actually the dynamics of small solutions is fairly similar on $\R\times\T$ and $\R$. But we may apply their general strategy to the wave guide Schr\"odinger equation, to understand the asymptotic behavior and in particular how this asymptotic behavior is related to {\it resonant dynamics}.

\subsubsection{The half wave equation}
Another motivation is from the study of the so-called {\em half wave} equation \cite{GGHW}. Actually, if we start with a solution $u$ which does not depend on $x$, then it satisfies the following half wave equation
\begin{equation}\label{HW}
i\partial_t u - \vert D_y \vert u = \vert u\vert^2u,\ y\in\T\ .
\end{equation}

The following theorem was proved by G\'erard and Grellier, which tells us the global well-posedness and partially about its large time behavior. The orthogonal projector from $L^2(\T )$ onto $$L^2_+(\T ):=\Big\{u(y)=\sum_{p\ge 0} u_p {\mathrm{e}}^{ipy},\; (u_{p})_{p\ge 0}\in\ell^2 \Big\}\ ,$$
is called the Szeg\H{o} projector and is denoted by $\Pi_+$.
\begin{theorem}\cite{GGHW}
Given $u_0\in H^{\frac 12}(\T )$, there exists a unique solution $u\in C(\R ,H^{\frac 12}(\T ))$ satisfying \eqref{HW}. And if $u_0\in H^s(\T )$ for some $s>\frac 12$, then $u\in C(\R , H^s(\T ))$. 
Moreover, let $s>1$ and $u_0=\Pi_+(u_0)\in L^2_+(\T)\cap H^s(\T)$ with $\Vert u_0\Vert_{H^s}=\varepsilon$, $\varepsilon>0$ small enough.
Denote by  $v$  the solution of the cubic Szeg\H{o} equation\cite{GGASENS,GGINV}
\begin{equation}\label{CauchySzego} i\partial_tv-Dv=\Pi_+(|v|^2v) \ ,\ 
v(0,\cdot)=u_0\ .
\end{equation}
Then, for any $\alpha >0$,  there exists  a constant $c=c_\alpha <1$ so that
\begin{equation}\label{approximation}\Vert u(t)-v(t)\Vert_{H^s}=\mathcal O(\varepsilon^{3-\alpha} )\text{ for }t\le \frac {c_{\alpha }}{\varepsilon^2}\log\frac 1\varepsilon\ .\end{equation}
\end{theorem}
A similar result is available for the case on the real line $\mathbb{R}$, see O. Pocovnicu \cite{OP}.

The following large time behavior result of the half wave equation comes from the fact that the cubic Szeg\H{o} dynamics which appears as the effective dynamics, admits large time Sobolev norm growth.
\begin{corollary}\label{coro}\cite{GGHW}
Let $s>1$. There exists a sequence of data $u_0^n$ and a sequence of times $t^{n}$
such that, for any $r$,
$$\Vert u_0^n\Vert _{H^r}\rightarrow 0$$
while the corresponding solution of \eqref{HW} satisfies
$$\Vert u^n(t^{n})\Vert _{H^s}\simeq \Vert u_0^n\Vert _{H^s}\, \left (\log \frac 1{\Vert u_0^n\Vert _{H^s}}\right )^{2s-1}\ .$$
\end{corollary}

\begin{remark}In the statement above, one may observe that there exists norm growth, but $\Vert u^n(t^{n})\Vert _{H^s}$ stays still small. In fact, it is possible to show that for $s>1$, there exists a sequence of $u^n$ solutions to the half wave equation \eqref{HW} such that \cite{OPcomm}
\begin{equation}
\Vert u_0^n\Vert_{H^s}\rightarrow0,\ \Vert u^n(t^{n})\Vert _{H^s}\rightarrow\infty\ .
\end{equation}
Indeed, one may just take some large integers $N_n=\Big[\Big(\Vert \widetilde{u_0^n}\Vert_{H^s}\Vert \widetilde{u^n}(t^n)\Vert_{H^s}\Big)^{-\frac{1}{1+2s}}\Big]$, set $u_0^n=N_n^{\frac12}\widetilde{u_0^n}(N_n^{\frac12}y)$ with $\widetilde{u_0^n}$ given as in Corollary~\ref{coro}, then we may write the related solution as  $u_0^n=N_n^{\frac12}\widetilde{u^n}(N_nt,N_n^{\frac12}y)$. 
\end{remark}

The existence of a solution to the half wave equation \eqref{HW} satisfying
\begin{equation}
\Vert u_0\Vert_{H^s}\leq\varepsilon,\quad\lim\sup_{t\to\infty}\Vert u(t)\Vert_{H^s}=\infty\ ,
\end{equation}
is still an open problem. Though this problem is still open for the half wave equation, we are going to solve it for the wave guide Schr\"odinger equation \eqref{HWS}.

\medskip
\subsection{Main results}
The aim of this paper is to describe the large time behavior of the wave guide Schr\"odinger equation \eqref{HWS} for small smooth data. Thoughout this paper, we always assume the initial data satisfy
\begin{equation}\label{Initial}
U_0(x,y+\pi)=-U_0(x,y)\ .
\end{equation}
A direct consequence is that $U_0$ only admits odd Fourier modes on the direction $y$, which is of helpful importance in the study of the resonant system, as we will see later in section 4. We then show that the asymptotic dynamics of small solutions to \eqref{HWS} is related to that of solutions of the resonant system
\begin{equation}\label{RSS}
\begin{split}
&i\partial_t G_{\pm}(t)=\mathcal{R}[G_{\pm}(t), G_{\pm}(t), G_{\pm}(t)]\ ,\\
&\mathcal{F}_{\R}\mathcal{R}[G_{\pm}, G_{\pm}, G_{\pm}](\xi,y)= \Pi_{\pm}(\vert \widehat{G}_\pm\vert^2\widehat{G}_\pm)(\xi,y)\ .
\end{split}
\end{equation}
Here $\widehat G(\xi, \cdot)=\mathcal F_{\R} G (\xi, \cdot)$, $\Pi_+$ is the Szeg\H{o} projector onto the non-negative Fourier modes, $\Pi_-:=\mathrm{Id}-\Pi_+$, and $G_\pm:=\Pi_\pm(G)$. Noting that the dependence on $\xi$ is merely parametric, the above system is none other than the resonant system for the cubic half wave equation on $\T$, which is the cubic Szeg\H{o} equation. 

Throughout this article, we assume $N\geq13$ is an arbitrary integer, and $\delta<10^{-3}$. Our main results on the modified scattering and the existence of a wave operator are as below, where the norms of Banach spaces $S$ and $S^+$ are defined as
\begin{equation}
\begin{split}
\Vert F\Vert_{S}:=&\Vert F\Vert_{H^N_{x,y}}+\Vert xF\Vert_{L^2_{x,y}},\quad
\Vert F\Vert_{S^+}:=\Vert F\Vert_S+\Vert (1-\partial_{xx})^4F\Vert_{S}+\Vert xF\Vert_{S}.
\end{split}
\end{equation}

\begin{theorem}\label{MS}
There exists $\varepsilon=\varepsilon(N)>0$ such that if $U_0\in S^+$ satisfies
\begin{equation*}
\Vert U_0\Vert_{S^+}\le\varepsilon,
\end{equation*}
and if $U(t)$ solves \eqref{HWS} with initial data $U_0$, then $U\in C([0,+\infty): S)$ exists globally and exhibits modified scattering to its resonant dynamics \eqref{RSS} in the following sense: there exists $G_0\in S$ such that if $G(t)$ is the solution of \eqref{RSS} with initial data $G(0)=G_0$, then
\begin{equation*}
\Vert U(t)-{\mathrm{e}}^{it\mathcal{A}}G(\pi\ln t)\Vert_{S}\to 0\ \text{ as }t\to+\infty.
\end{equation*}
\end{theorem}
\begin{remark}
The Cauchy problem of our wave guide Schr\"odinger system \eqref{HWS} in the classical Sobolev space is not easy, neither by energy estimates nor by Strichartz estimates, since its Hamiltonian energy lies on the Sobolev space $H_x^1L_y^2\cap L_x^2H_y^{1/2}$. However, by the Theorem~\ref{MS} above, we can deduce directly the global well-posedness with small initial data in $S^+$.
\end{remark}
\begin{theorem}\label{MWO}
There exists $\varepsilon=\varepsilon(N)>0$ such that if $G_0\in S^+$ satisfies
\begin{equation*}
\Vert G_0\Vert_{S^+}\le\varepsilon,
\end{equation*}
$G(t)$ solves \eqref{RSS} with initial data $G_0$, then there exists $U\in C([0,\infty):S)$ a solution of \eqref{HWS} such that
\begin{equation*}
\begin{split}
\Vert U(t)-{\mathrm{e}}^{it\mathcal{A}}G(\pi\ln t)\Vert_{S}\to0\ \text{ as }\,t\to+\infty.
\end{split}
\end{equation*}
\end{theorem}
Theorem~\ref{MWO} combined with the large time behavior of the cubic Szeg\H{o} equation, leads to the infinite cascades result.
\begin{theorem}\label{BlowUp}
Given $N\geq13$, then for any $\varepsilon>0$, there exists $U_0\in S^+$ with $\Vert U_0\Vert_{S^+}\leq\varepsilon$, such that the corresponding solution to \eqref{HWS} satisfies
\begin{equation}
\limsup_{t\to\infty}\Vert U(t)\Vert_{L_x^2H_y^s}=\infty\ ,\quad \forall s> 1/2\ . 
\end{equation} 
\end{theorem}
\begin{remark}\mbox{}\\
\noindent {\bf 1.} It is likely there exists a dense $G_\delta$ set in an appropriate space containing initial data which lead to infinite cascade as above. A proof of this would involve more technicalities and we will not discuss it in this paper. 

\noindent {\bf 2.} Compared to the results in \cite{HPTV}, the unbounded Sobolev norms in our theorem are just above the energy norm.
\end{remark}
\subsection{Organization of the paper}
In section 2, we introduce the notation used in this paper. In section 3, we study the structure of the non-linearity, and establish the decomposition proposition, which is of crucial importance. We decompose the non-linearity $\mathcal{N}^t$ into a combination of the resonant part and a remainder,
\begin{align*}
\mathcal{N}^t[F,G,H]=\frac{\pi}{t}\mathcal{R}[F,G,H]+\mathcal{E}^t[F,G,H]\ .
\end{align*}
 In section 4, we study the resonant system and its large time cascade, which is similar to the cubic Szeg\H{o} equation as above. In section 5, we construct the modified wave operator and prove Theorem~\ref{MWO} and Theorem~\ref{MS}. Later in this section, we prove the large time blow up result, Theorem~\ref{BlowUp}. Finally in section 6, we present a lemma that will allow us to transfer $L^2$ estimates on operators into estimates in $S$ and $S^+$ norms.


\section{Preliminary}
\subsection{Notation}
We will follow the notation of \cite{HPTV}, $\T:=\R/(2\pi\Z)$, the inner product $(U,V):=\int_{\R\times\T} U\overline{V} dx dy$ for any $U,V\in L^2(\R\times\T)$. We will use the lower-case letter to denote functions $f:\R\to\C$ and the capital letters to denote functions $F:\R\times\T\to\C$, and calligraphic letters denote operators, except for the Littlewood-Paley operators and dyadic numbers which are capitalized most of the time.

We use a different notation to denote Fourier transform on different space variables. The Fourier transform on $\R$ is defined by
\begin{equation*}
\widehat{g}(\xi):=\mathcal{F}_x(g)(\xi)=\int_{\R}{\mathrm{e}}^{-ix\xi}g(x)dx\ .
\end{equation*}
Similarly, if $U(x,y)$ depends on $(x,y)\in\R\times\T$, $\widehat{U}(\xi,y)$ denotes the partial Fourier transform in $x$. The Fourier transform of $h:\T\to\C$ is,
\begin{equation*}
h_p:=\mathcal{F}_y(h)(p)=\frac{1}{2\pi}\int_{\T}h(y){\mathrm{e}}^{-ipy}dy,\qquad p\in\Z\ ,
\end{equation*}
and this also extends to $U(x,y)$. Finally, we define the full Fourier transform on the cylinder $\R\times\T$
\begin{equation*}
\left(\mathcal{F} U\right)(\xi,p)=\frac{1}{2\pi}\int_{\T}\widehat{U}(\xi,y){\mathrm{e}}^{-ipy}dy=\widehat{U}_p(\xi)\ .
\end{equation*}

We will often use Littlewood-Paley projections. For the full frequency space, these are defined as follows with $N$ as a dyadic integer.
\begin{equation*}
\begin{split}
\left(\mathcal{F} P_{\le N} U\right)(\xi,p)=\varphi(\frac{\xi}{N})\varphi(\frac{p}{N})\left(\mathcal{F} U\right)(\xi,p)\ ,
\end{split}
\end{equation*}
where $\varphi\in C^\infty_c(\R)$, $\varphi(x)=1$ when $\vert x\vert\le 1$ and $\varphi(x)=0$ when $\vert x\vert\ge 2$. We then also define
\begin{equation}
\phi(x)=\varphi(x)-\varphi(2x)
\end{equation}
and
\begin{equation}
P_N=P_{\le N}-P_{\le N/2},\quad P_{\geq N}=1-P_{\leq N}\ .
\end{equation}

Sometimes we concentrate on the frequency in $x$ only, and we therefore define
\begin{equation*}
\begin{split}
\left(\mathcal{F} Q_{\le N}U\right)(\xi,p)=\varphi(\frac{ \xi}{N})\left(\mathcal{F} U\right)(\xi,p)\ ,
\end{split}
\end{equation*}
and define $Q_N$ similarly. By a slight abuse of notation, we will consider $Q_N$ indifferently as an operator on functions defined on $\R\times\T$ and on $\R$.  While we consider the frequency in 
$y$ we will use notation $\Delta_N$ which means
\begin{equation}\label{deltaj}
\left(\mathcal{F}_y \Delta_Nh\right)(p)=\phi(\frac{p}{N})h_p\ .
\end{equation}
We shall use the following commutator estimate which is a direct consequence of the definition,
\begin{equation}\label{comm}
\|[Q_N,x]\|_{L^2_x\rightarrow L^2_x}\lesssim N^{-1}\ .
\end{equation}

We will use the following sets corresponding to momentum and resonance level sets:
$$\mathcal{M}:=\{(p_0,p_1,p_2,p_3)\in\Z^4:\ p_0-p_1+p_2-p_3=0\}\ ,$$
$$\Gamma_{\omega}:=\{(p_0,p_1,p_2,p_3)\in\Z^4:\ \vert p_0\vert-\vert p_1\vert+\vert p_2\vert-\vert p_3\vert=\omega\}\ .$$

\subsection{The non-linearity}
Let us write a solution of \eqref{HWS} as
$$U(x,y,t)=\sum\limits_{p\in\Z}{\mathrm{e}}^{ipy}e^{-it|p|}\big({\mathrm{e}}^{it\partial_{xx}}F_p(t)\big)(x):={\mathrm{e}}^{it\mathcal{A}}(F(t))\ ,$$
with $\mathcal{A}=\partial_{xx}-\vert D_y\vert$.
We then see that $U$ solves \eqref{HWS} if and only if $F$ solves
\begin{equation}\label{edno}
i\partial_t F(t) = {\mathrm{e}}^{-it\mathcal{A}}\Big({\mathrm{e}}^{it\mathcal{A}}F(t)\cdot {\mathrm{e}}^{-it\mathcal{A}}\overline{F(t)}\cdot{\mathrm{e}}^{it\mathcal{A}}F(t)\Big)\ .
\end{equation}
We denote the non-linearity in \eqref{edno} by $\mathcal{N}^t[F(t),F(t),F(t)]$, where the trilinear form $\mathcal{N}^t$ is defined by
$$\mathcal{N}^t[F,G,H]:={\mathrm{e}}^{-it\mathcal{A}}\Big({\mathrm{e}}^{it\mathcal{A}}F\cdot {\mathrm{e}}^{-it\mathcal{A}}\overline{G}\cdot{\mathrm{e}}^{it\mathcal{A}}H\Big)\ .$$
Now, we can compute the Fourier transform of the last expression
\begin{equation}\label{fourierN}
\mathcal{F}\mathcal{N}^t[F,G,H](\xi,p)=\sum_{(p,q,r,s)\in\mathcal{M}}{\mathrm{e}}^{it\left(\vert p\vert-\vert q\vert+\vert r\vert-\vert s\vert\right)}
\mathcal{F}_x\big(\mathcal{I}^t[F_q,G_r,H_s]\big)(\xi)\ ,
\end{equation}
where
\begin{equation}\label{DefOfI}
\mathcal{I}^t[f,g,h]:=\mathcal{U}(-t)\Big(\mathcal{U}(t)f\, \overline{\mathcal{U}(t)g}\,\mathcal{U}(t)h \Big),\ \mathcal{U}(t)={\mathrm{e}}^{it\partial_{xx}}\ .
\end{equation}
One verifies that
$$
\mathcal{F}_x\big(\mathcal{I}^t[f,g,h]\big)(\xi)=\int_{\mathbb{R}^2}{\mathrm{e}}^{it2\eta\kappa}\widehat{f}(\xi-\eta)\overline{\widehat{g}}(\xi-\eta-\kappa)\widehat{h}(\xi-\kappa)d\kappa d\eta\ .
$$

\subsection{Norms}
The following Sobolev norms will be used in the whole paper. For sequences $a:=\{a_p:\ p\in\Z\}$, we define the following norm,
\begin{equation*}
\Vert a \Vert_{h^s_p}^2:=\sum_{p\in\Z}\left[1+\vert p\vert^2\right]^s\vert a_p\vert^2\ .
\end{equation*}

The Besov space $B^1=B^1_{1,1}(\T )$ is defined as the set of functions $f$ such that $\Vert f\Vert_{B^1_{1,1}}$ is finite, where $$
\Vert f\Vert_{B^1_{1,1}}=\Vert S_0(f)\Vert_{L^1}+\sum_{N \text{ dyadic}} N\Vert \Delta_N f\Vert_{L^1},$$
here $f=S_0(f)+\sum\limits_{N \text{  dyadic}}\Delta_N f$ stands for the Littlewood-Paley decomposition of $f$ with $\Delta_N$ defined as \eqref{deltaj} above and $\mathcal{F}_y (S_0f)(p):=\varphi(p)f_p $. The space $B^1$ will be crucial in the analysis of the resonant system in section 4.

 For functions $F$ defined on $\R\times\T$, we will indicate the domain of integration by a subscript $x$ (for $\R$), $x,y$ (for $\R\times\T$) or $p$ (for $\Z$). We will use mainly four different norms:\\
two weak norms
\begin{align}
\Vert F\Vert_{Y^s}^2&:=\sup _{\xi\in\mathbb{R}}\left[1+\vert \xi\vert^2\right]^2\sum_p(1+|p|^2)^s\vert\widehat{F}_p(\xi)\vert^2\ , \label{DefYs}\\
\Vert F\Vert_{Z}&:=\sup_{\xi\in\R}\left[1+\vert \xi\vert^2\right]\|\widehat{F}(\xi,\cdot)\|_{B^1}\ ,
\end{align}
and two strong norms
\begin{equation}\label{DefS}
\Vert F\Vert_{S}:=\Vert F\Vert_{H^N_{x,y}}+\Vert xF\Vert_{L^2_{x,y}}\ ,\\
\Vert F\Vert_{S^+}:=\Vert F\Vert_S+\Vert (1-\partial_{xx})^4F\Vert_{S}+\Vert xF\Vert_{S}\ ,
\end{equation}
with $N$ to be fixed later.

The space-time norms we will use are
\begin{equation}\label{DefX}
\begin{split}
\Vert F\Vert_{X_T}:=&\sup_{0\le t\le T}\big\{\Vert F(t)\Vert_Z+(1+\vert t\vert)^{-\delta}\Vert F(t)\Vert_S +(1+\vert t\vert)^{1-3\delta}\Vert\partial_t F(t)\Vert_S\big\}\ ,\\
\Vert F\Vert_{X_T^+}:=&\Vert F\Vert_{X_T}+\sup_{0\le t\le T}\big\{(1+\vert t\vert)^{-5\delta}\Vert F(t)\Vert_{S^+}+(1+\vert t\vert)^{1-7\delta}\Vert\partial_t F(t)\Vert_{S^+}\big\}\ ,
\end{split}
\end{equation}
with a small parameter $\delta<10^{-3}$.

In the following sections, we will see that the $Z$ norm is a conserved quantity for the resonant system, which is of crucial importance, and for data in $S^+$, the solution is expected to grow slowly in $S^+$, while the difference between the true solution to \eqref{HWS} and the solution to the resonant system may decay in $S$. 

Now, at this stage, we present some elementary lemmas which will be useful in the later studies.
\begin{lemma}
Provided $N\geq13$, we have the following hierarchy
\begin{align}
&\Vert F\Vert_{Y^{1/2}}\lesssim \Vert F\Vert_{Z}\lesssim\Vert F\Vert_{Y^s},\ s>1\ ,\label{ineq1}\\
&\Vert F\Vert_{H^{1/2}_{x,y}}\lesssim \Vert F\Vert_Z\lesssim\Vert F\Vert_S\lesssim \Vert F\Vert_{S^+}\ .\label{ineq2}
\end{align}
\end{lemma}

\begin{proof}We begin with the proof of the first inequality \eqref{ineq1}, it is  sufficient to prove
$$\Vert f\Vert_{H_y^{1/2}}\lesssim \Vert f\Vert_{B^1}\lesssim\Vert f\Vert_{H_y^s},\ s>1\ .$$
\noindent {\bf 1.} $\Vert f\Vert_{H_y^{1/2}}\lesssim \Vert f\Vert_{B^1}$.\\
\begin{equation*}
\Vert f\Vert_{H_y^{1/2}}=\Vert S_0f\Vert_{L^2}+\big(\sum_{N\text{ dyadic}}N\Vert \Delta_N f\Vert_{L^2}^2\big)^{1/2}\ .
\end{equation*}
We notice that the Fourier transform of $S_0f$ is compactly supported on some interval $I$ with $|I|<2$, thus
\begin{equation*}
\Vert S_0f\Vert_{L^2}\leq \Vert \mathcal{F}_y(S_0f)(p)\Vert_{\ell_p^2}\leq 
\left(\sum_{p\in I} \vert \int{\mathrm{e}}^{-ixp}(S_0f)(x)dx\vert^2\right)^{1/2}\lesssim \Vert S_0f\Vert_{L^1}\ .
\end{equation*}
While the Fourier transform of $\Delta_N f$ is compactly supported on some interval $I$ with $|J_N|\sim N$, thus similarly
\begin{equation*}
N^{1/2}\Vert \Delta_N f\Vert_{L^2}\leq N^{1/2}\Vert \mathcal{F}_y(\Delta_N f)(p)\Vert_{\ell_p^2(J_N)}\lesssim N\Vert \Delta_N f\Vert_{L^1}\ ,
\end{equation*}
we then use the fact that $\ell^1$ is continuously embedded in $\ell^2$ and get 
\begin{equation*}
\Vert N^{1/2}\Vert \Delta_N f\Vert_{L^2}\Vert_{\ell_N^2}\leq \Vert N\Vert \Delta_N f\Vert_{L^1}\Vert_{\ell_N^2}\lesssim \sum_{N \text{ dyadic}} N\Vert \Delta_N f\Vert_{L^1}\ ,
\end{equation*}
thus $\Vert f\Vert_{H_y^{1/2}}\lesssim \Vert f\Vert_{B^1}$.

\smallskip
\noindent {\bf 2.} $\Vert f\Vert_{B^1}\lesssim\Vert f\Vert_{H_y^s},\ s>1$.\\
Since $\mathbb{T}$ is of finite measure, 
$$\Vert f\Vert_{L^1(\T)}\leq \Vert f\Vert_{L^2(\T)}\ .$$
 This inequality is deduced by Cauchy-Schwarz inequality, indeed,
\begin{equation*}
\sum N\Vert \Delta_N f\Vert_{L^1}\leq \sum N\Vert \Delta_N f\Vert_{L^2}\leq\big(\sum N^{2s}\Vert \Delta_N f\Vert_{L^2}^2\big)^{1/2}\big(\sum_{N \text{ dyadic}} N^{-2(s-1)}\big)^{1/2}\ ,
\end{equation*}
the second factor on the right hand side converges since $s>1$, and we obtain our result.

\smallskip
\noindent {\bf 3.} It is easy to show the first and last inequality in \eqref{ineq2}, and the middle inequality comes from the following Gagliardo-Nirenberg type inequality
\begin{equation}\label{GN}
 \Vert F\Vert_{Y^s}\lesssim\Vert F\Vert_{L^2_{x,y}}^{1/2-\sigma}\Vert F\Vert_S^{1/2+\sigma}\ ,\ s>1\ ,
\end{equation}
with $0<\sigma<1/2$ and the index in the definition of $S$ norm satisfies $\sigma N>3$.

To verify this inequality,  we need the elementary inequality
\begin{equation}\label{L11}
\|\hat{f}\|_{L^\infty_\xi(\mathbb{R})}\lesssim\|f\|_{L^1_x(\mathbb{R})}\lesssim \|f\|_{L^2_x(\R)}^{\frac{1}{2}}\| x  f\|_{L^2_x(\R)}^{\frac{1}{2}},
\end{equation}
one might observe that 
\begin{equation*}
\begin{split}
&\left[1+\vert \xi\vert^2\right]\vert \widehat{F}_p(\xi)\vert
\lesssim\sum_{N\text{ dyadic}} N^2\vert \widehat{Q_NF_p}(\xi)\vert\\
\qquad&\lesssim \sum_NN^2\Vert Q_NF_p(\cdot)\Vert_{L^2_x}^\frac{1}{2}\Vert  x  Q_NF_p(\cdot)\Vert_{L^2_x}^\frac{1}{2}\\
\qquad&\lesssim \big(\sum_NN^{-\frac{\theta-4}{2}}\big)\Vert (1-\partial_{xx})^\frac{\theta}{2}F_p(\cdot)\Vert_{L^2_x}^\frac{1}{2}\Vert \langle x\rangle F_p(\cdot)\Vert_{L^2_x}^\frac{1}{2}\\
\qquad&\lesssim \Vert F_p(\cdot)\Vert_{H^{\theta}_x}^\frac{1}{2}\Vert \langle x\rangle F_p(\cdot)\Vert_{L^2_x}^\frac{1}{2}\ ,
\end{split}
\end{equation*}
where we applied \eqref{comm} to gain the third inequality, and $\theta>4$. 
Squaring and multiplying by $\langle p\rangle^{2s}$, and combining with \eqref{ineq1}, we have for $s>1$,
\begin{equation*}
 \Vert F\Vert_{Y^s}^2=\sup_\xi [1+|\xi|^2]^2\Vert \widehat{F}(\xi,\cdot)\Vert_{H_y^s}^2 \lesssim \sum_{p\in\mathbb{Z}}\langle p\rangle^{2s}\Vert F_p(\cdot)\Vert_{H^{\theta}_x}\Vert \langle x\rangle F_p(\cdot)\Vert_{L^2_x}\lesssim \Vert F\Vert_{H_x^{\theta}H_y^{2s}}\Vert xF\Vert_{L_{x,y}^2}\ ,
\end{equation*}
the last inequality comes from the Cauchy-Schwarz inequality.
Then \eqref{GN} comes from an application of the Gagliardo-Nirenberg inequality on $\Vert F\Vert_{H_{x,y}^{\theta+2s}}$ with $\theta+2s>6$, 
$$\Vert F\Vert_{H_{x,y}^{\theta+2s}}\leq \Vert F\Vert_{L_{x,y}^2}^{1-2\sigma}\Vert F\Vert_{H_{x,y}^{N}}^{2\sigma}\ ,$$
and $ 2\sigma N=\theta+2s>6$.
 By choosing $\sigma=1/4$ and $N>12$, thus for $s>1$,
\begin{equation}\label{ZSNorm}
 \Vert F\Vert_Z\lesssim \Vert F\Vert_{Y^s}\lesssim\Vert F\Vert_{L^2_{x,y}}^\frac{1}{4}\Vert F\Vert_S^\frac{3}{4}\ .
\end{equation}
\end{proof}

We remark that by taking suitable $\sigma$, for the inequality \eqref{ineq2}, the requirement of the Sobolev regularity in $S$ norm may be $N\geq 7$. 

\medskip
We also remark that the operators $Q_{\le N}$, $P_{\le N}$ and the multiplication by $\varphi(\cdot/N)$ are bounded in $Z$, $S$, $S^+$, uniformly in $N$.

In this paper, we make often use of the following elementary bound to sum-up the $1d$ estimates,
\begin{equation}\label{sump}
\Big\Vert
\sum_{q-r+s=p}
c^1_qc^2_rc^3_s\Big\Vert_{\ell^2_{p}}
\lesssim
\min_{\{j,k,\ell\}=\{1,2,3\}}\Vert c^j\Vert_{\ell^2_p}\Vert c^k\Vert_{\ell^1_p}\Vert c^\ell\Vert_{\ell^1_p}\ .
\end{equation}

The following lemma shows the bounds on the non-linearity $\mathcal{N}^t$ in the $S$ and $S^+$ norms.
\begin{lemma}\label{Nestimate}\cite[Lemma~2.1]{HPTV}
\begin{equation}\label{nestimate}
\begin{split}
&\Vert \mathcal{N}^t[F,G,H]\Vert_S\lesssim (1+\vert t\vert)^{-1}\Vert F\Vert_S\Vert G\Vert_S\Vert H\Vert_S\ ,\\
&\Vert \mathcal{N}^t[F^1,F^2,F^3]\Vert_{S^+}\lesssim (1+\vert t\vert)^{-1}\max_{\{j,k,\ell\}=\{1,2,3\}}\Vert F^j\Vert_{S^+}\Vert F^k\Vert_{S}\Vert F^\ell\Vert_{S}\ .
\end{split}
\end{equation}
\end{lemma}
\begin{proof}
Due to Lemma~\ref{SS+} in the appendix, it is sufficient to prove
\begin{equation}\label{c1}
\Vert \mathcal{N}^t[F^1,F^2,F^3]\Vert_{L_{x,y}^2}\lesssim (1+\vert t\vert)^{-1}\min_{\{j,k,\ell\}=\{1,2,3\}}\Vert F^j\Vert_{L_{x,y}^2}\Vert F^k\Vert_S\Vert F^\ell\Vert_S\ .
\end{equation}
Coming back to \eqref{fourierN}, 
\begin{align}\label{c2}
\Vert \mathcal{N}^t[F^1,F^2,F^3]\Vert_{L_{x,y}^2}\lesssim\Vert\sum_{q-r+s=p}\Vert \mathcal{I}^t[F_q^1,F_r^2,F_s^3]\Vert_{L_x^2}\Vert_{\ell_p^2}\ ,
\end{align}
thus we only need to calculate $\Vert \mathcal{I}^t[f^1,f^2,f^3]\Vert_{L_x^2}$. By the definition of $\mathcal{I}^t$ \eqref{DefOfI}, we have the energy bound
\begin{align*}
\Vert \mathcal{I}^t[f^1,f^2,f^3]\Vert_{L_x^2}&=\Big\Vert {\mathrm{e}}^{-it\partial_{xx}}\Big({\mathrm{e}}^{it\partial_{xx}}f^1\overline{{\mathrm{e}}^{it\partial_{xx}}f^2}{\mathrm{e}}^{it\partial_{xx}}f^3\Big)\Big\Vert_{L_x^2}\\&\lesssim\min_{\{j,k,\ell\}=\{1,2,3\}}\Vert f^j\Vert_{L_x^2}\Vert{\mathrm{e}}^{it\partial_{xx}} f^k\Vert_{L_x^\infty}\Vert{\mathrm{e}}^{it\partial_{xx}} f^\ell\Vert_{L_x^\infty}\ .
\end{align*}
Then by \eqref{sump},
\begin{align}\label{c3}
\Vert \mathcal{N}^t[F^1,F^2,F^3]\Vert_{L_{x,y}^2}\lesssim\min_{\{j,k,\ell\}=\{1,2,3\}}\Vert F^j\Vert_{L_{x,y}^2}\sum_r\Vert {\mathrm{e}}^{it\partial_{xx}}F_r^k\Vert_{L_x^\infty}\sum_s\Vert {\mathrm{e}}^{it\partial_{xx}}F_s^\ell\Vert_{L_x^\infty}\ .
\end{align}
For $|t|>1$, the factor $(1+|t|)^{-1}$ comes from the dispersive estimate
\begin{equation}\label{dispersiveest}
\Vert {\mathrm{e}}^{it\partial_{xx}}f\Vert_{L_x^\infty}\lesssim \vert t\vert^{-\frac12}\Vert f\Vert_{L_x^1}\lesssim \vert t\vert^{-\frac12}\Vert f\Vert_{L_x^2}^{\frac12}\Vert xf\Vert_{L_x^2}^{\frac12}\ ,
\end{equation}
then 
\begin{align*}
\sum_p\Vert{\mathrm{e}}^{it\partial_{xx}}F_p\Vert_{L_x^\infty}&\lesssim  |t|^{-1/2}\sum_p\Vert F_p\Vert_{L_x^2}^{\frac12}\Vert xF_p\Vert_{L_x^2}^{\frac12}\\
&= |t|^{-1/2}\sum_p\vert p\vert^{-s}\vert p\vert^{s}\Vert F_p\Vert_{L_x^2}^{\frac12}\Vert xF_p\Vert_{L_x^2}^{\frac12}\\
&\leq  |t|^{-1/2}(\sum_p \vert p\vert^{-2s})^{1/2}(\sum_p\vert p\vert^{4s}\Vert F_p\Vert_{L_x^2}^2 )^{1/4}(\sum_p\Vert xF_p\Vert_{L_x^2}^2)^{1/4}\\
&\leq |t|^{-1/2}\Vert F\Vert_S\ ,
\end{align*} 
where we took $s>1/2$ in the second and third inequalities.
While for $|t|\leq1$, one may use Sobolev estimate instead of the dispersive estimate,
$$\Vert {\mathrm{e}}^{it\partial_{xx}}f\Vert_{L_x^\infty}\lesssim \Vert f\Vert_{H_x^1}\ ,$$
then
\begin{align*}
\sum_p\Vert{\mathrm{e}}^{it\partial_{xx}}F_p\Vert_{L_x^\infty}&\lesssim \sum_p\Vert F_p\Vert_{H_x^1}=\sum_p\vert p\vert^{-s}\vert p\vert^{s}\Vert F_p\Vert_{H_x^1}\\
&\leq (\sum_p \vert p\vert^{-2s})^{1/2}(\sum_p\vert p\vert^{2s}\Vert F_p\Vert_{H_x^1}^2 )^{1/2}\leq \Vert F\Vert_S\ ,
\end{align*} 
with $s>1/2$. Thus for any $t$,
\begin{equation}\label{c4}
\sum_{p\in\Z} \Vert{\mathrm{e}}^{it\partial_{xx}}F_p\Vert_{L_x^\infty}\lesssim (1+|t|)^{-1/2}\Vert F\Vert_S\ .
\end{equation}
Plugging \eqref{c4} into \eqref{c3}, we get \eqref{c1} and complete the proof of Lemma~\ref{Nestimate}.

\end{proof}


\section{Structure of the non-linearity}
The purpose of this section is to extract the key effective interactions from the full non-linearity in \eqref{HWS}. We are to gain the decomposition
\begin{equation}
\mathcal{N}^t[F,G,H]=\frac{\pi}{t}\mathcal{R}[F, G,H]+\mathcal{E}^t[F,G,H]\ ,
\end{equation}
where $\mathcal{R}$ is the resonant part,
\begin{equation}\label{DefR}
\mathcal{F}\mathcal{R}[F,G,H](\xi,p)=\sum\limits_{(p,q,r,s)\in\Gamma_0}\widehat{F}_q(\xi)\overline{\widehat{G}_r(\xi)}\widehat{H}_s(\xi)\  ,
\end{equation}
and $\mathcal{E}^t$ is a remainder term, which is estimated in Proposition~\ref{Nonlinearity} below. We will see later that this $\mathcal{R}[G,G,G]$ is exactly the same one as in \eqref{RSS}. 

Our main result in this section is the following proposition.
\begin{proposition}\label{Nonlinearity}
Assume that for $T^\ast\geq 1$,  $F$, $G$, $H$: $\mathbb{R}\to S$ satisfy
\begin{equation}\label{BA}
\Vert F\Vert_{X_{T^\ast}}+\Vert G\Vert_{X_{T^\ast}}+\Vert H\Vert_{X_{T^\ast}}\leq 1\ .
\end{equation}
Then we can write
$$
\mathcal{E}^{t}[F(t),G(t),H(t)]=\mathcal{E}_{1}^{t}[F(t),G(t),H(t)]+\mathcal{E}_{2}^{t}[F(t),G(t),H(t)]\ ,
$$
and if for $j=1,2$ we note $\mathcal{E}_j(t):=\mathcal{E}_j^{t}[F(t),G(t),H(t)]$ then the following estimates hold uniformly in $T^\ast\geq 1$,
\begin{equation*} 
\begin{split}
&\sup_{1\leq T\leq T^\ast}T^{-\delta}\Vert \int_{T/2}^T\mathcal{E}_j(t)dt\Vert_{S}\lesssim 1,\quad j=1,2\ ,\\
&\sup_{1\leq t\leq T^\ast}(1+\vert t\vert)^{1+\delta}\Vert \mathcal{E}_{1}(t)\Vert_{Z}\lesssim \sup_{1\leq t\leq T^\ast}(1+\vert t\vert)^{1+\delta}\Vert \mathcal{E}_{1}(t)\Vert_{Y^s}\lesssim 1\ ,\ s>1\ ,\\
&\sup_{1\leq t\leq T^\ast}(1+\vert t\vert)^{1/10}\Vert \mathcal{E}_3(t)\Vert_{S}\lesssim 1\ ,
\end{split}
\end{equation*}
where $\mathcal{E}_{2}(t)=\partial_t\mathcal{E}_3(t)$.
Assuming in addition
\begin{equation}\label{BA+}
\Vert F\Vert_{X^+_{T^\ast}}+\Vert G\Vert_{X^+_{T^\ast}}+\Vert H\Vert_{X^+_{T^\ast}}\leq 1\ ,
\end{equation}
we also have that
\begin{equation*} 
\sup_{1\leq T\leq T^\ast}T^{-5\delta}\Vert \int_{T/2}^T\mathcal{E}_j(t)dt\Vert_{S^+}\lesssim 1,\ \sup_{1\leq T\leq T^\ast} T^{2\delta}\Vert \int_{T/2}^T\mathcal{E}_j(t)dt\Vert_{S}\lesssim 1,\ j=1,2\ .\\
\end{equation*}
\end{proposition}

The statement of Proposition~\ref{Nonlinearity} says that if the remainder $\mathcal{E}^t$ has inputs bounded in $Z$ and slightly growing in $S$ then $\mathcal{E}^t$ reproduces the same growth in $S$ and even decays in $Z$. To prove this proposition, we first present several reductions by performing a decomposition of the non-linearity as
$$\sum_{A,B,C-\text{dyadic}}\mathcal{N}^t[Q_AF(t),Q_BG(t),Q_CH(t)]\ .$$ 

\subsection{The High Frequency Estimates}
In this subsection, we are going to prove a decay estimate on the non-linearity $\mathcal{N}^t[Q_AF(t),Q_BG(t),Q_CH(t)]$ for $t\sim T$, $T\ge1$, in the regime $\max(A,B,C)\ge T^{\frac{1}{6}}$. In the case when two inputs have high frequencies, we can simply conclude by using energy estimates, while in the case when the highest frequency is much higher than the others, we invoke the bilinear refinements of the Strichartz estimate on $\R$. 
\begin{lemma}\label{1dBE}\cite{CKSTTSIAM}
Assume that $\lambda/10 \geq  \mu\geq 1$ and that $u(t)={\mathrm{e}}^{it\partial_{xx}}u_0$, $v(t)={\mathrm{e}}^{it\partial_{xx}}v_0$.
Then, we have the bound
\begin{equation}\label{BS}
\Vert
Q_{\lambda}u\overline{Q_{\mu} v}
\Vert_{L^2_{x,t}(\mathbb{R}\times\mathbb{R})}\lesssim \lambda^{-\frac{1}{2}}\Vert u_0\Vert_{L^2_x(\mathbb{R})}\Vert v_0\Vert_{L^2_x(\mathbb{R})}.
\end{equation}
\end{lemma}

One may refer to \cite{CKSTTSIAM} for the proof. 

\medskip
 Slight modifications of the proof of the corresponding result in \cite[Lemma 3.2]{HPTV} lead to the following estimates. We reproduce the proof here for the readers' convenience.
\begin{lemma}\label{highfreq}
Assume that $ T\ge 1$. The following estimates hold uniformly in $T$:
\begin{align}
&\Big\Vert \sum_{\substack{A,B,C\\\max(A,B,C)\ge T^{\frac{1}{6}}}}\mathcal{N}^t[Q_AF,Q_BG,Q_CH]\Big\Vert_{Y^s}\notag\\
&\qquad\qquad\qquad\qquad\qquad\qquad\lesssim T^{-\frac{5}{4}}
\Vert F\Vert_{S}\Vert G\Vert_{S}\Vert H\Vert_{S},\ s>1,\ \forall t\ge T/4,\\
&\Big\Vert \sum_{\substack{A,B,C\\\max(A,B,C)\ge T^{\frac{1}{6}}}}\int_{\frac{T}{2}}^T\mathcal{N}^t[Q_AF(t),Q_BG(t),Q_CH(t)]dt\Big\Vert_S\notag\\
&\qquad\qquad\qquad\qquad\qquad\qquad\lesssim T^{-\frac{1}{50}}\Vert F\Vert_{X_T}\Vert G\Vert_{X_T}\Vert H\Vert_{X_T},\\
&\Big\Vert \sum_{\substack{A,B,C\\\max(A,B,C)\ge T^{\frac{1}{6}}}}\int_{\frac{T}{2}}^T \mathcal{N}^t[Q_AF(t),Q_BG(t),Q_CH(t)]dt\Big\Vert_{S^+}\notag\\
&\qquad\qquad\qquad\qquad\qquad\qquad\lesssim T^{-\frac{1}{50}}\Vert F\Vert_{X_T^+}\Vert G\Vert_{X_T^+}\Vert H\Vert_{X_T^+}.
\end{align}
\end{lemma}

\medskip
\begin{proof}
Let us begin with the first inequality. 
Let $K\in L_{x,y}^2$, then we need to bound
\begin{align*}
I_K&=\left\langle K,\ \mathcal{N}^t[Q_AF,Q_BG,Q_CH]\right\rangle\\
&\leq \left\vert\int_{\R\times\T} e^{it\mathcal{A}}(Q_AF)\cdot\overline{e^{it\mathcal{A}}(Q_BG)}\cdot e^{it\mathcal{A}}(Q_CH)\cdot\overline{e^{it\mathcal{A}}(K)}\right\vert\ .
\end{align*}
By Sobolev embedding, we see that
\begin{eqnarray*}
&&\left\vert\int_{\R\times\T} {\mathrm{e}}^{it\mathcal{A}}(Q_AF)\cdot\overline{{\mathrm{e}}^{it\mathcal{A}}(Q_BG)}\cdot {\mathrm{e}}^{it\mathcal{A}}(Q_CH)\cdot\overline{{\mathrm{e}}^{it\mathcal{A}}(K)}\right\vert\\
&\lesssim & \Vert  {\mathrm{e}}^{it\mathcal{A}}Q_{A}F\Vert_{L^6_{x,y}}\Vert  {\mathrm{e}}^{it\mathcal{A}}Q_{B}G\Vert_{L^6_{x,y}}\Vert  {\mathrm{e}}^{it\mathcal{A}}Q_{C}H\Vert_{L^6_{x,y}}\|K\|_{L^2_{x,y}}\\
&\lesssim &  \Vert  {\mathrm{e}}^{it\mathcal{A}}Q_{A}F\Vert_{H^{s}_{x,y}}\Vert  {\mathrm{e}}^{it\mathcal{A}}Q_{B}G\Vert_{H^{s}_{x,y}}\Vert  {\mathrm{e}}^{it\mathcal{A}}Q_{C}H\Vert_{H^{s}_{x,y}}\|K\|_{L^2_{x,y}}\\
&=&  \Vert  Q_{A}F\Vert_{H^{s}_{x,y}}\Vert Q_{B}G\Vert_{H^{s}_{x,y}}\Vert  Q_{C}H\Vert_{H^{s}_{x,y}}\|K\|_{L^2_{x,y}}\\
&\lesssim &(ABC)^{-13+s}\|Q_{A}F\|_{H^{13}_{x,y}}\|Q_{B}G\|_{H^{13}_{x,y}}\|Q_{C}H\|_{H^{13}_{x,y}}\|K\|_{L^2_{x,y}}\ ,
\end{eqnarray*}
with $s>2/3$. Then by duality, taking $s=1$, we have 
\begin{equation}\label{l2sss}
\begin{split}
&\Big\Vert \mathcal{N}^t[Q_AF,Q_BG,Q_CH]\Big\Vert_{L^2_{x,y}}\\
&\qquad\lesssim (ABC)^{-12}\|Q_{A}F\|_{H^{13}_{x,y}}\|Q_{B}G\|_{H^{13}_{x,y}}\|Q_{C}H\|_{H^{13}_{x,y}}\\
&\qquad\lesssim (ABC)^{-12}\|Q_{A}F\|_{S}\|Q_{B}G\|_{S}\|Q_{C}H\|_{S}\ .
\end{split}
\end{equation}
Then By\eqref{ZSNorm}, 
\begin{align*}
&\Big\Vert \sum_{\substack{A,B,C\\\max(A,B,C)\ge T^{\frac{1}{6}}}}\mathcal{N}^t[Q_AF,Q_BG,Q_CH]\Big\Vert_{Y^s}\\
&\qquad\lesssim \sum_{\substack{A,B,C\\\max(A,B,C)\ge T^{\frac{1}{6}}}}\Big\Vert\mathcal{N}^t[Q_AF,Q_BG,Q_CH]\Big\Vert_{Y^s}\\
&\qquad\lesssim \sum_{\substack{A,B,C\\\max(A,B,C)\ge T^{\frac{1}{6}}}}\Big \Vert\mathcal{N}^t[Q_AF,Q_BG,Q_CH]\Big\Vert_S^{3/4}\ \Big\Vert \mathcal{N}^t[Q_AF,Q_BG,Q_CH]\Big\Vert_{L^2}^{1/4}\\
&\qquad\lesssim T^{-3/4}\sum_{\substack{A,B,C\\\max(A,B,C)\ge T^{\frac{1}{6}}}}(ABC)^{-3}\|Q_{A}F\|_{S}\|Q_{B}G\|_{S}\|Q_{C}H\|_{S}\\
&\qquad\lesssim T^{-5/4}\Vert F\Vert_S\Vert G\Vert_S\Vert H\Vert_S\ ,
\end{align*}
where in the third inequality we used Lemma~\ref{Nestimate} and \eqref{l2sss}.

For the other two estimates, we must be more careful. First of all, we will split the set $\{(A,B,C):\ \max(A,B,C)\geq T^{\frac{1}{6}}\}$ into two parts $\Lambda$ and its relative complement $\Lambda^c$. Here the set $\Lambda$ is defined as
$\Lambda:=\Big\{(A,B,C):\ \mathrm{med}(A,B,C)\leq T^{\frac{1}{6}}/16,\ \max(A,B,C)\geq T^{\frac{1}{6}}\Big\}$,
with $\mathrm{med}(A,B,C)$ denote the second largest dyadic number among $(A,B,C)$.

Let us start with the case $(A,B,C)\in\Lambda^c$, we claim
\begin{equation}\label{lamc}
\Big\Vert \sum_{(A,B,C)\in\Lambda^c}\mathcal{N}^t[Q_AF,Q_BG,Q_CH]\Big\Vert_{S^{(+)}}\lesssim T^{-\frac{11}{6}}
\Vert F\Vert_{S^{(+)}}\Vert G\Vert_{S^{(+)}}\Vert H\Vert_{S^{(+)}}\ 
\end{equation}
By Lemma~\ref{SS+}, we only need to control $\Big\Vert \sum\limits_{(A,B,C)\in\Lambda^c}\mathcal{N}^t[Q_AF,Q_BG,Q_CH]\Big\Vert_{L^2}$, the main strategy is similar to the proof above, but this time we should not lose derivatives on all of the $F,G,H$, let us check the condition \eqref{LB}. 
Let $K\in L_{x,y}^2$, then we need to bound
\begin{align*}
I_K&=\left\langle K,\ \sum_{(A,B,C)\in\Lambda^c}\mathcal{N}^t[Q_AF,Q_BG,Q_CH]\right\rangle\\
&\leq \sum_{(A,B,C)\in\Lambda^c}\left\vert\int_{\R\times\T} {\mathrm{e}}^{it\mathcal{A}}(Q_AF)\cdot\overline{{\mathrm{e}}^{it\mathcal{A}}(Q_BG)}\cdot {\mathrm{e}}^{it\mathcal{A}}(Q_CH)\cdot\overline{{\mathrm{e}}^{it\mathcal{A}}(K)}\right\vert\\
&\lesssim \sum_{(A,B,C)\in\Lambda^c}\|Q_{A}F\|_{L^2_{x,y}}\| {\mathrm{e}}^{it\mathcal{A}}Q_{B}G\|_{L_{x,y}^\infty}\| {\mathrm{e}}^{it\mathcal{A}}Q_{C}H\|_{L_{x,y}^{\infty}}\|K\|_{L^2_{x,y}}
\end{align*}
\begin{align*}
&\lesssim \sum_{(A,B,C)\in\Lambda^c}\|Q_{A}F\|_{L^2_{x,y}}\|Q_{B}G\|_{H_{x,y}^2}\|Q_{C}H\|_{H_{x,y}^2}\|K\|_{L^2_{x,y}}\\
&\lesssim \sum_{(A,B,C)\in\Lambda^c}(BC)^{-11}\|Q_{A}F\|_{L^2_{x,y}}\|Q_{B}G\|_{H^{13}_{x,y}}\|Q_{C}H\|_{H_{x,y}^{13}}\|K\|_{L^2_{x,y}}\\
&\lesssim \Big(\sum_{(A,B,C)\in\Lambda^c}(\mathrm{med}(A,B,C))^{-11}\Big)\|F\|_{L^2_{x,y}}\|G\|_{H^{13}_{x,y}}\|H\|_{H_{x,y}^{13}}\|K\|_{L^2_{x,y}}\\
&\lesssim T^{-11/6}\Vert F\Vert_{L_{x,y}^2}\Vert K\Vert_{L_{x,y}^2} \|G\|_S\|H\|_S\ ,
\end{align*}
then by duality,
\begin{equation}
\Big\Vert \sum\limits_{(A,B,C)\in\Lambda^c}\mathcal{N}^t[Q_AF,Q_BG,Q_CH]\Big\Vert_{L^2}\lesssim T^{-11/6}\Vert F\Vert_{L_{x,y}^2}\|G\|_S\|H\|_S\ .
\end{equation}
The inequality above holds by replacing $F$ with $G$, $H$, then we get \eqref{lamc} by applying Lemma~\ref{SS+}.

\medskip
Now we turn to the case $(A,B,C)\in\Lambda$, we are to show
\begin{equation}\label{Slam}
\begin{split}
&\Vert \sum_{\substack{A,B,C\\ (A,B,C)\in\Lambda}}\int_{\frac{T}{2}}^T\mathcal{N}^t[Q_AF(t),Q_BG(t),Q_CH(t)]dt\Vert_{S^{(+)}}\\&\qquad\lesssim T^{-\frac{1}{50}}\Vert F\Vert_{X_T^{(+)}}\Vert G\Vert_{X_T^{(+)}}\Vert H\Vert_{X_T^{(+)}}\ .
\end{split}\end{equation}
We will only prove the case with norms $S$ and $X_T$, the proof of the case with $S^+,X_T^+$ is similar. The main tool of this part is the bilinear Strichartz estimate from Lemma~\ref{1dBE}. We consider a decomposition 
\begin{equation}
[T/4,2T]=\bigcup_{j\in J} I_j\ ,\ I_j=[jT^\frac{9}{10},(j+1)T^\frac{9}{10}]=[t_j,t_{j+1}]\ ,\ \#J\lesssim T^{\frac{1}{10}}
\end{equation}
and consider $\chi\in C^\infty_c(\mathbb{R})$, $\chi\geq 0$ such that $\chi(s)=0$ if $\vert s\vert\ge 2$ and
\begin{equation*}
\sum_{k\in\mathbb{Z}}\chi(s-k)\equiv 1\ .
\end{equation*}
The left hand-side of \eqref{Slam} can be estimated by $C(E_1+E_2)$, where
\begin{multline*}
E_1=\Big\| \sum_{j\in J} \sum_{(A,B,C)\in \Lambda}\int_{\frac{T}{2}}^T\chi\big(\frac{t}{T^{\frac{9}{10}}}-j\big)\\
\Big(\mathcal{N}^t[Q_AF(t),Q_BG(t),Q_CH(t)]-\mathcal{N}^t[Q_AF(t_j),Q_BG(t_j),Q_CH(t_j)]\Big)dt\Big\|_{S}
\end{multline*}
and
\begin{equation*}
E_2= \Big\|\sum_{j\in J} \sum_{(A,B,C)\in \Lambda}
\int_{\frac{T}{2}}^T\chi\big(\frac{t}{T^{\frac{9}{10}}}-j\big)\mathcal{N}^t[Q_AF(t_j),Q_BG(t_j),Q_CH(t_j)]dt\Big\|_S\ .
\end{equation*}
Notice that $F(t_j),\ G(t_j),\ H(t_j)$ do not depend on $t$.

\medskip
Let us start to estimate $E_1$,
\begin{equation}\label{E1Dec}
E_1\leq  \sum_{j\in J} \int_{\frac{T}{2}}^T\chi\big(\frac{t}{T^{\frac{9}{10}}}-j\big)E_{1,j}(t)dt\ 
\end{equation}
with
\begin{equation*}
\begin{split}
&E_{1,j}(t):=\\
&\qquad\Big\|\sum_{(A,B,C)\in \Lambda}\Big(\mathcal{N}^t[Q_AF(t),Q_BG(t),Q_CH(t)]-\mathcal{N}^t[Q_AF(t_j),Q_BG(t_j),Q_CH(t_j)]\Big)\Big\|_S\ .
\end{split}\end{equation*}

Denote by $Q_+:=Q_{\ge T^\frac{1}{6}}$ and $ Q_-:=Q_{\le T^\frac{1}{6}/16}$, then due to the structure of $\Lambda$, one of $A,B,C$ is larger than $T^{\frac16}$ and the other two are smaller than $ T^\frac{1}{6}/16$, we decompose
\begin{equation*}
\begin{split}
&\sum_{(A,B,C)\in\Lambda}\mathcal{N}^t[Q_AF,Q_BG,Q_CH]=\mathcal{N}^t[Q_{+}F,Q_{-}G,Q_{-}H]\\
&\qquad\qquad\qquad\qquad\qquad\qquad+\mathcal{N}^t[Q_{-}F,Q_{+}G,Q_{-}H]+\mathcal{N}^t[Q_{-}F,Q_{-}G,Q_{+}H]\ .
\end{split}
\end{equation*}
We rearrange the terms in $E_{1,j}$ two by two, and rewrite each pair as follows
\begin{align*}
&\mathcal{N}^t[Q_{+}F(t),Q_{-}G(t),Q_{-}H(t)]-\mathcal{N}^t[Q_{+}F(t_j),Q_{-}G(t_j),Q_{-}H(t_j)]\\
&\qquad=\mathcal{N}^t[Q_{+}(F(t)-F(t_j)),Q_{-}G,Q_{-}H(t)]+\mathcal{N}^t[Q_{+}F(t_j),Q_{-}(G(t)-Gt_j)),Q_{-}H(t)]\\
&\qquad+\mathcal{N}^t[Q_{+}F(t_j),Q_{-}G(t_j),Q_{-}(H(t)-H(t_j))]\ ,
\end{align*}
then by Lemma~\ref{Nestimate}, and the boundedness of $Q_\pm$ on $S^{(+)}$, we see that
\begin{align*}
\Vert \mathcal{N}^t[Q_{+}(F(t)-F(t_j)),Q_{-}G,Q_{-}H(t)]\Vert_S\lesssim (1+\vert t\vert)^{-1}\Vert F(t)-F(t_j)\Vert_S\Vert G(t)\Vert_S\Vert H(t)\Vert_S\ ,
\end{align*}
We bound the other terms similarly, and finally we have an estimate on $E_{1,j}$,
\begin{equation}\label{EstimE1j}
\begin{split}
E_{1,j}(t)&
\le (1+\vert t\vert)^{-1}\Big[\Vert F(t)-F(t_j)\Vert_S\Vert G(t)\Vert_S\Vert H(t)\Vert_S\\
&+\Vert F(t_j)\Vert_S\Vert G(t)-G(t_j)\Vert_S\Vert H(t)\Vert_{S}\\
&+\Vert F(t_j)\Vert_S\Vert G(t_j)\Vert_S\Vert H(t)-H(t_j)\Vert_S\Big]\ .
\end{split}
\end{equation}
Since $\vert t-t_j\vert\leq T^{\frac{9}{10}}$, 
$$\Vert F(t)-F(t_j)\Vert_S\leq \int_{t_j}^t\|\partial_t F(\theta)\|_S d\theta\leq T^{\frac{9}{10}}\sup_t\|\partial_t F(t)\|_S\ .$$
Notice that this is the advantage of introducing the partition of time interval provided by $\chi$. Comparing with the definition of $X_T$ (see \eqref{DefX}), we have
\begin{align*}
\Vert F(t)-F(t_j)\Vert_S &\leq T^{-\frac{1}{10}+3\delta}\|F\|_{X_T}\ ,\\
\Vert F(t) \Vert_S &\leq T^{\delta}\|F\|_{X_T}\ .
\end{align*}
Therefore,
$$E_{1,j}\lesssim  T^{-\frac{11}{10}+5\delta}\|F\|_{X_T}\|G\|_{X_T}\|H\|_{X_T}\ ,$$
then
$$E_1\lesssim \int_{T/2}^T \sum_{j\in J}\chi(\frac{t}{T^{\frac{9}{10}}}-j)E_{1,j(t)}dt\lesssim T^{-\frac{1}{10}+5\delta}\|F\|_{X_T}\|G\|_{X_T}\|H\|_{X_T}\ .$$

\medskip
We now turn to $E_2$, recall
$$E_2= \Big\|\sum_{j\in J} \sum_{(A,B,C)\in \Lambda}
\int_{\frac{T}{2}}^T\chi\big(\frac{t}{T^{\frac{9}{10}}}-j\big)\mathcal{N}^t[Q_AF(t_j),Q_BG(t_j),Q_CH(t_j)]dt\Big\|_S\ ,$$
with $Q_AF(t_j),Q_BG(t_j),Q_CH(t_j)$ do not depend on $t$.
Denoting
\begin{equation*}
E_{2,j}^{A,B,C}= \Big\| \int_{\frac{T}{2}}^{T}\chi\big(\frac{t}{T^{\frac{9}{10}}}-j\big)\mathcal{N}^t[Q_AF(t_j),Q_BG(t_j),Q_CH(t_j)]dt\ \Big\|_S\ ,
\end{equation*}
then
$$
E_2\leq  \sum_{j\in J} \sum_{(A,B,C)\in \Lambda}E_{2,j}^{A,B,C}\ .
$$

We claim 
\begin{equation}\label{SuffE2}
\begin{split}
&\Big\Vert \int_{\frac{T}{2}}^{T}\chi\big(\frac{t}{T^{\frac{9}{10}}}-j\big)
\mathcal{N}^t[Q_AF^a,Q_BF^b,Q_CF^c]dt\Big\Vert_{L^2_{x,y}}\\
&\lesssim (\max(A,B,C))^{-1}\min_{\{\alpha,\beta,\gamma\}=\{a,b,c\}}\Vert F^{\alpha}\Vert_{L^2_{x,y}}\Vert F^{\beta}\Vert_S\Vert F^{\gamma}\Vert_S.
\end{split}
\end{equation}
Then by Lemma~\ref{SS+}, $\Vert E_{2,j}^{A,B,C}\Vert_S\lesssim  (\max(A,B,C))^{-1}\Vert F\Vert_S\Vert G\Vert_S\Vert H\Vert_S$, the estimate for $E_2$ will come out by summing up.
Let us prove \eqref{SuffE2},  assuming $K\in L^2_{x,y}$, we consider with functions $F^a, F^b,F^c$ independent on $t$,
\begin{equation*}
\begin{split}
I_K&=\sum_{p-q+r-s=0}{\mathrm{e}}^{it\omega}\langle K_p,\int_{\frac{T}{2}}^{T}\chi\big(\frac{t}{T^{\frac{9}{10}}}-j\big)\mathcal{N}^t[Q_AF_q^a,Q_BF_r^b,Q_CF_s^c]dt\rangle_{L^2_{x}\times L^2_{x}}\\
&=\sum_{p-q+r-s=0}{\mathrm{e}}^{it\omega} \int_{\frac{T}{2}}^{T}\int_{\R\times\T}\chi\big(\frac{t}{T^{\frac{9}{10}}}-j\big){\mathrm{e}}^{it\partial_{xx}}(Q_A F_q^a)\overline{{\mathrm{e}}^{it\partial_{xx}}(Q_B F_r^b)}{\mathrm{e}}^{it\partial_{xx}}(Q_C F_s^c)\overline{{\mathrm{e}}^{it\partial_{xx}}K_p}dxdt
\end{split}
\end{equation*}
where we may assume that $K=Q_DK$, $D\simeq \max(A,B,C)$. Without loss of generality, we assume $A=\max(A,B,C)$,  then by H\"older's inequality,
\begin{align*}
&\Big\vert \int_{\frac{T}{2}}^{T}\int_{\R\times\T}\chi\big(\frac{t}{T^{\frac{9}{10}}}-j\big){\mathrm{e}}^{it\partial_{xx}}(Q_A F_q^a)\overline{{\mathrm{e}}^{it\partial_{xx}}(Q_B F_r^b)}{\mathrm{e}}^{it\partial_{xx}}(Q_C F_s^c)\overline{{\mathrm{e}}^{it\partial_{xx}}Q_DK_p}dxdt\Big\vert\\
& \leq\Vert{\mathrm{e}}^{it\partial_{xx}}(Q_A F_q^a)\overline{{\mathrm{e}}^{it\partial_{xx}}(Q_B F_r^b)}\Vert_{L^2_{x,t}}\Vert{\mathrm{e}}^{it\partial_{xx}}(Q_C F_s^c)\overline{{\mathrm{e}}^{it\partial_{xx}}Q_DK_p}\Vert_{L^2_{x,t}}\ ,
\end{align*}
since $A\geq16B$, $D\geq16C$, applying the bilinear Strichartz estimate from Lemma~\ref{1dBE} below, we then have
\begin{align*}
&\Vert{\mathrm{e}}^{it\partial_{xx}}(Q_A F_q^a)\overline{{\mathrm{e}}^{it\partial_{xx}}(Q_B F_r^b)}\Vert_{L^2_{x,t}}\lesssim A^{-1/2}\Vert F^a_q\Vert_{L^2_x}\Vert F^b_s\Vert_{L^2_x}\\
&\Vert{\mathrm{e}}^{it\partial_{xx}}(Q_C F_s^c)\overline{{\mathrm{e}}^{it\partial_{xx}}Q_DK_p}\Vert_{L^2_{x,t}}\lesssim D^{-1/2}\Vert  F^c_r\Vert_{L^2_x}\Vert K_p\Vert_{L^2_x}\ .
\end{align*}
Applying Cauchy-Schwarz and \eqref{sump} on the summation on the right hand side of $I_K$,  we have
\begin{equation*}
\begin{split}
I_K&\lesssim \sum_{p-q+r-s=0}(\max(A,B,C))^{-1}\|F_q^a\|_{L^2_{x}}\|F_r^b\|_{L^2_{x}}\|F_s^c\|_{L^2_{x}}\Vert K_p\Vert_{L^2_{x}}\\
&\lesssim (\max(A,B,C))^{-1}\Big\Vert \sum_{p=q-r+s}\|F_q^a\|_{L^2_{x}}\|F_r^b\|_{L^2_{x}}\|F_s^c\|_{L^2_{x}}\Big\Vert_{\ell_p^2}\Vert K_p\Vert_{L^2_{x,y}}\\
&\lesssim (\max(A,B,C))^{-1}\min_{\{\alpha,\beta,\gamma\}=\{a,b,c\}}\|F^\alpha\|_{L^2_{x,y}}\sum_p\|F_p^\beta\|_{L^2_{x}}\sum_p\|F_p^\gamma\|_{L^2_{x}}\Vert K\Vert_{L^2_{x,y}}\\
&\lesssim(\max(A,B,C))^{-1} \min_{\{\alpha,\beta,\gamma\}=\{a,b,c\}}\Big(\|F^\alpha\|_{L^2_{x,y}}\sum_p\big(\vert p\vert^{-s}\vert p\vert^s\vert\|F_p^\beta\|_{L^2_{x}}\big)\\
&\qquad\qquad\qquad\qquad\qquad\qquad\qquad\qquad\qquad\qquad\cdot\sum_p\big(\vert p\vert^{-s}\vert p\vert^s\|F^\gamma\|_{L^2_{x}}\big)\Vert K\Vert_{L^2_{x,y}}\Big)\\
&\lesssim (\max(A,B,C))^{-1}\min_{\{\alpha,\beta,\gamma\}=\{a,b,c\}}\|F^\alpha\|_{L^2_{x,y}}\|F^\beta\|_{S}\|F^\gamma\|_{S}\Vert K\Vert_{L^2_{x,y}}\ ,
\end{split}
\end{equation*}
where we took $s>1/2$. The result \eqref{SuffE2} turns out by duality. Applying Lemma~\ref{SS+}, we get
\begin{align*}
E_{2,j}^{A,B,C}\lesssim (\max(A,B,C))^{-1}\Vert F\Vert_{S}\Vert G\Vert_S\Vert H\Vert_S\ ,
\end{align*} 
then
\begin{align*}
E_2\leq \sum_{j\in J}\sum_{(A,B,C)\in\Lambda}E_{2,j}^{A,B,C}\lesssim \#J \sum_{(A,B,C)\in\Lambda}(\max(A,B,C))^{-1}\Vert F\Vert_{S}\Vert G\Vert_S\Vert H\Vert_S\ .
\end{align*}
Without loss of generality, we assume $A=\max(A,B,C)$, then
\begin{equation*}
\sum_{(A,B,C)\in \Lambda}(\max(A,B,C))^{-1}=\big(\sum_{A\geq T^{1/6}}A^{-1}\big)\big(\#\{B:B\leq T^{1/6}/16\}\big)^2 \lesssim T^{-1/6+\delta}
\end{equation*}
while using the definition \eqref{DefX},
\begin{equation*}
\Vert F(t_j)\Vert_S\Vert G(t_j)\Vert_S\Vert H(t_j)\Vert_S\leq T^{3\delta} \Vert F\Vert_{X_T}\Vert G\Vert_{X_T}\Vert H\Vert_{X_T}\ ,
\end{equation*}
thus
\begin{equation}
E_2\lesssim T^{-1/15+\delta}\Vert F\Vert_{X_T}\Vert G\Vert_{X_T}\Vert H\Vert_{X_T}\ ,
\end{equation}
which is a stronger version of \eqref{Slam}. The proof of Lemma~\ref{highfreq} is complete.
\end{proof}

\medskip
Thus we may suppose that the $x$ frequencies of $F,G,H$ are $\lesssim T^{\frac{1}{6}}$. It is natural to introduce the first decomposition
\begin{align}\label{Dec2}
&\mathcal{N}^t[F,G,H]=\mathcal{N}_0^t[F,G,H]+\widetilde{\mathcal{N}}^t[F,G,H]\ ,\\
&\mathcal{F}\mathcal{N}_0^t(\xi, p):=\sum_{(p,q,r,s)\in\Gamma_0}\mathcal{F}_x\big(\mathcal{I}^t[F_q,G_r,H_s]\big)(\xi)\ .
\end{align}
\subsection{The fast oscillations}
Firstly, we present another elementary estimate here.
\begin{lemma}\label{CM}
Let $\frac{1}{p}=\frac{1}{q}+\frac{1}{r}+\frac{1}{s}$ with $ 1\leq p,q,r,s\le\infty$,
then
\begin{equation*}
\begin{split}
\big\Vert \int_{\R^3}{\mathrm{e}}^{ix\xi}m(\eta,\kappa)\widehat{f}(\xi-\eta)\overline{\widehat{g}}(\xi-\eta-\kappa)\widehat{h}(\xi-\kappa)d\eta d\kappa d\xi\big\Vert_{L_{x}^p}\lesssim\Vert\mathcal{F}^{-1}m\Vert_{L^1(\R^2)}\Vert f\Vert_{L^q}\Vert g\Vert_{L^r}\Vert h\Vert_{L^s}.
\end{split}
\end{equation*}
\end{lemma}

\begin{proof}
\begin{equation*}
\begin{split} 
I&=
\int_{\mathbb{R}^3}{\mathrm{e}}^{ix\xi}m( \eta, k)\widehat{f}(\xi-\eta)\overline{\widehat{g}}(\xi-\eta-k)\widehat{h}(\xi-k)d\eta d\kappa d\xi\\
&=\int_{\mathbb{R}^3\times \mathbb{R}^2}\left(\int_{\mathbb{R}^3}{\mathrm e}^{i\xi(x-\alpha+\beta-\gamma)}{\mathrm e}^{-i\eta(y-\alpha+\beta)}{\mathrm e}^{i\kappa(z+\beta-\gamma)}d\xi d\eta d\kappa\right)\\
&\qquad\qquad\qquad\qquad\qquad\qquad\qquad\qquad\mathcal{F}^{-1}m(y,z)f(\alpha)\overline{g}(\beta)h(\gamma)dydzd\alpha d\beta d\gamma\\
&=\int_{\mathbb{R}^2}\mathcal{F}^{-1}m(y,z)f(x-z)\overline{g}(x-y-z)h(x-y)dydz\ ,
\end{split}
\end{equation*}
then
\begin{equation*}
\begin{split}
\Vert I\Vert_{L^p_x}&\leq \int_{\mathbb{R}^3}\vert\mathcal{F}^{-1}m(y,z)\vert \Vert f\overline{g}h\Vert_{L^p_x}dydz\\
&=\Vert\mathcal{F}^{-1}m\Vert_{L^1(\mathbb{R}^2)}\Vert f\overline{g}h\Vert_{L^p(\mathbb{R})} \\
&\leq \Vert\mathcal{F}^{-1}m\Vert_{L^1(\R^2)}\Vert f\Vert_{L^q}\Vert g\Vert_{L^r}\Vert h\Vert_{L^s}\ ,
\end{split}
\end{equation*}
the last inequality comes from the H\"older's inequality and the assumption $\frac1p=\frac1q+\frac1r+\frac1s.$
\end{proof}
\begin{remark}
Similar result holds for the case $m=m(\xi,\eta,\kappa)$, one may refer to \cite[Lemma 7.5]{HPTV}.
\end{remark}

\medskip

The main purpose of this subsection is to estimate of $\widetilde{\mathcal{N}}^t$. 
\begin{lemma}\label{fastosc}
Let $1\le T\le T^\ast$. Assume that $F$, $G$, $H$: $\mathbb{R} \to S$ satisfy \eqref{BA} and
\begin{equation*}
F=Q_{\le T^{1/6}}F, \quad G=Q_{\le T^{1/6}}G, \quad H=Q_{\le T^{1/6}}H\ .
\end{equation*}
Then we can write
$$
\widetilde{\mathcal{N}}^t[F(t),G(t),H(t)]=\widetilde{\mathcal{E}}_1^t[F(t),G(t),H(t)]+\mathcal{E}_2^t[F(t),G(t),H(t)]\ ,
$$
and if we set $\widetilde{\mathcal{E}_1}(t):=\widetilde{\mathcal{E}}^t_1[F(t),G(t),H(t)]$
and $\mathcal{E}_2(t):=\mathcal{E}_2^t[F(t),G(t),H(t)]$ then it holds that, uniformly in $1\leq T\leq T^\ast$ ,
\begin{equation*}
T^{1+2\delta}\sup_{T/4\leq t\le T^\ast}\Vert \widetilde{\mathcal{E}_1}(t)\Vert_S\lesssim 1\ ,\quad
T^{1/10}\sup_{T/4\leq t\le T^\ast}\Vert \mathcal{E}_3(t)\Vert_S\lesssim 1\ ,
\end{equation*}
where $\mathcal{E}_2(t)=\partial_t\mathcal{E}_3(t)$.
Assuming in addition that \eqref{BA+} holds we have 
\begin{equation*}
T^{1+2\delta}\sup_{T/4\leq t\le T^\ast}\Vert \widetilde{\mathcal{E}_1}(t)\Vert_{S^+}\lesssim 1\ ,
\qquad T^{1/10}\sup_{T/4\le t\le T^\ast}\Vert \mathcal{E}_3(t)\Vert_{S^+}\lesssim 1\ .
\end{equation*}
\end{lemma}
\begin{proof}
To prove this lemma, we start by decomposing $\widetilde{\mathcal N}^t$ along the non-resonant level sets as follows:
Set 
$$F^a=Q_{\le T^{1/6}}F^a,\quad F^b=Q_{\le T^{1/6}}F^b,\quad F^c=Q_{\le T^{1/6}}F^c\ ,$$
\begin{align}
\mathcal{F}\widetilde{\mathcal{N}}^t[F^a,F^b,F^c](\xi,p)&=
\sum_{\omega\ne 0}\sum_{(p,q,r,s)\in\Gamma_\omega}{\mathrm{e}}^{it\omega}\left(\mathcal{O}^t_1[F^a_q,F^b_r,F^c_s](\xi)+\mathcal{O}^t_2[F^a_q,F^b_r,F^c_s](\xi)\right)\ ,\\
\mathcal{O}^t_1[f^a,f^b,f^c](\xi)&:=\int_{\mathbb{R}^2}{\mathrm{e}}^{2it\eta\kappa}(1-\varphi(t^{\frac{1}{4}}\eta\kappa))\widehat{f^a}(\xi-\eta)\overline{\widehat{f^b}}(\xi-\eta-\kappa)\widehat{f^c}(\xi-\kappa)d\eta d\kappa\ ,
\notag\\
\mathcal{O}^t_2[f^a,f^b,f^c](\xi)&:=\int_{\mathbb{R}^2}{\mathrm{e}}^{2it\eta\kappa}\varphi(t^{\frac{1}{4}}\eta\kappa)\widehat{f^a}(\xi-\eta)\overline{\widehat{f^b}}(\xi-\eta-\kappa)\widehat{f^c}(\xi-\kappa)d\eta d\kappa\ .\notag
\end{align}
We may rewrite for $\omega\neq0$,
\begin{equation}\label{O2}
\begin{split}
&{\mathrm{e}}^{it\omega}\mathcal{O}_{2}^t[f^a, f^b, f^c]=\partial_t\left(\frac{{\mathrm{e}}^{it\omega}}{i\omega}\mathcal{O}_2^t[f^a, f^b, f^c]\right)-\frac{{\mathrm{e}}^{it\omega}}{i\omega}\left(\partial_t\mathcal{O}_2^t\right)[f^a, f^b, f^c]\\
&\qquad-\frac{{\mathrm{e}}^{it\omega}}{i\omega}\mathcal{O}_2^t[\partial_t f^a, f^b, f^c]-\frac{{\mathrm{e}}^{it\omega}}{i\omega}\mathcal{O}_2^t[f^a,\partial_t f^b, f^c]-\frac{{\mathrm{e}}^{it\omega}}{i\omega}\mathcal{O}_2^t[f^a, f^b,\partial_t f^c]\\
&\qquad:=\partial_t\left(\frac{{\mathrm{e}}^{it\omega}}{i\omega}\mathcal{O}_2^t[f^a, f^b, f^c]\right)+{\mathrm{e}}^{it\omega}\mathcal{L}^t[f^a, f^b, f^c]\ ,
\end{split}
\end{equation}
where
$$\left(\partial_t\mathcal{O}_2^t\right)[f^a, f^b, f^c]:=\int_{\mathbb{R}^2}\partial_t\left( {\mathrm{e}}^{2it\eta\kappa}\varphi(t^{\frac{1}{4}}\eta\kappa)\right)\widehat{f^a}(\xi-\eta)\overline{\widehat{f^b}}(\xi-\eta-\kappa)\widehat{f^c}(\xi-\kappa)d\eta d\kappa\ .$$
Thus we define $\mathcal{E}^t_2[F^a,F^b,F^c]=\partial_t \mathcal{E}^t_3[F^a,F^b,F^c]$ with
\begin{equation}\label{S3}
\mathcal{F}\mathcal{E}^t_3[F^a,F^b,F^c](\xi,p):=\sum_{\omega\ne 0}\sum_{(p,q,r,s)\in\Gamma_\omega}\left(\frac{{\mathrm{e}}^{it\omega}}{i\omega}\mathcal{O}_2^t[F_q^a,F _r^b, F_s^c]\right)\ ,
\end{equation}
and define $ \widetilde{\mathcal{E}}^t_1$ with $\mathcal{O}_1^t$ and the last four terms in \eqref{O2},
\begin{equation}\label{S1}
\mathcal{F}\widetilde{\mathcal{E}}^t_1(\xi,p):=\sum_{\omega\ne 0}\sum_{(p,q,r,s)\in\Gamma_\omega}{\mathrm{e}}^{it\omega}\big(\mathcal{O}_1^t[F_q^a,F_r^b,F_s^c]+\mathcal{L}^t[F_q^a,F_r^b,F_s^c]\big)\ .
\end{equation}

\medskip
\noindent {\bf 1. Estimation of $\mathcal{E}_3(t)$.} We define the multiplier appearing in the definition of $\mathcal{O}_2^t$ by
\begin{equation*}
m(\eta,\kappa):=\varphi(t^{\frac{1}{4}}\eta\kappa)\varphi((10T)^{\frac{-1}{6}}\eta)\varphi((10T)^{\frac{-1}{6}}\kappa)\ .
\end{equation*}
>From Lemma~\ref{multiplier} at the end of this subsection, it is bounded by $\|\mathcal{F}_{\eta\kappa}\widetilde{m}\|_{L^1(\R^2)}\lesssim t^{\frac{\delta}{100}}$.  
Applying Lemma ~\ref{CM}, we get
\begin{align*}
\Vert\mathcal{O}_2^t[f^a,f^b,f^c]\Vert_{L_{\xi}^2}&\lesssim (1+\vert t\vert)^{\frac{\delta}{100}}\min_{\{\alpha,\beta,\gamma\}=\{a,b,c\}}\Vert f^{\alpha}\Vert_{L_x^2}\Vert {\mathrm{e}}^{it\partial_{xx}}f^{\beta}\Vert_{L_x^\infty}\Vert {\mathrm{e}}^{it\partial_{xx}}f^{\gamma}\Vert_{L_x^\infty}\ .
\end{align*}
Then
\begin{align*}
\Vert \mathcal{E}_3(t)\Vert_{L^2_{x,y}}&\lesssim \Big\Vert\sum_{\omega\ne 0}\sum_{(p,q,r,s)\in\Gamma_\omega}\left(\frac{{\mathrm{e}}^{it\omega}}{i\omega}\mathcal{O}_2^t[F_q^a,F _r^b, F_s^c]\right)\Big\Vert_{L_x^2\ell_p^2}\\
&\lesssim (1+\vert t\vert)^{\frac{\delta}{100}}\min_{\{\alpha,\beta,\gamma\}=\{a,b,c\}}\Big\Vert\sum_{p-q+r-s=0} \Vert F_q^{\alpha}\Vert_{L_x^2}\Vert {\mathrm{e}}^{it\partial_{xx}}F_r^{\beta}\Vert_{L_x^\infty}\Vert {\mathrm{e}}^{it\partial_{xx}}F_s^{\gamma}\Vert_{L_x^\infty}\Big\Vert_{\ell_p^2}\\
\text{using \eqref{sump}}\\
&\lesssim (1+\vert t\vert)^{\frac{\delta}{100}}\min_{\{\alpha,\beta,\gamma\}=\{a,b,c\}}\Vert F^{\alpha}\Vert_{L_{x,y}^2}\sum_r\Vert {\mathrm{e}}^{it\partial_{xx}}F_r^{\beta}\Vert_{L_x^\infty}\sum_s\Vert {\mathrm{e}}^{it\partial_{xx}}F_s^{\gamma}\Vert_{L_x^\infty}\\
\text{ using \eqref{dispersiveest} }\\
&\lesssim (1+\vert t\vert)^{-1+\frac{\delta}{100}}\min_{\{\alpha,\beta,\gamma\}=\{a,b,c\}}\Big(\Vert F^{\alpha}\Vert_{L_{x,y}^2}\sum_r\big(\Vert F_r^{\beta}\Vert_{L^2_x}^{1/2}\Vert xF_r^{\beta}\Vert_{L^2_x}^{1/2}\big)\\
&\qquad\qquad\qquad\qquad\qquad\qquad\qquad\qquad\qquad\qquad\cdot\sum_s\big(\Vert F_s^{\gamma}\Vert_{L^2_x}^{1/2}\Vert xF_s^{\gamma}\Vert_{L^2_x}^{1/2}\big)\Big)\ .
\end{align*}
Noticing that for the last inequality, we have 
\begin{align*}
\sum_r(\vert a_r\vert^{1/2}\vert b_r\vert^{1/2})&\leq \sum_r(\vert a_r\vert^{1/2}\vert r\vert^{\theta}\vert r\vert^{-\theta}\vert b_r\vert^{1/2})\\
&\leq \Vert a_r\Vert_{h_r^{2\theta}} ^{1/4}(\sum_r\vert r\vert^{-2\theta})^{1/2}\Vert b_r\Vert_{\ell_r^2}\\
&\lesssim \Vert a_r\Vert_{h_r^{2\theta}} ^{1/2}\Vert b_r\Vert_{\ell_r^2}
\end{align*}
with $\theta>1/2$.
Then
\begin{align}
\Vert \mathcal{E}_3(t)\Vert_{L^2_{x,y}}&\lesssim (1+\vert t\vert)^{-1+\frac{\delta}{100}}\min_{\{\alpha,\beta,\gamma\}=\{a,b,c\}}\Vert F^{\alpha}\Vert_{L_{x,y}^2}\Vert F^{\beta}\Vert_{H^{2\theta}_{x,y}}^{1/2}\Vert xF^{\beta}\Vert_{L^2_{x,y}}^{1/2}\Vert F^{\gamma}\Vert_{H^{2\theta}_{x,y}}^{1/2}\Vert xF^{\gamma}\Vert_{L^2_{x,y}}^{1/2}\notag\\
&\lesssim  (1+\vert t\vert)^{-1+\frac{\delta}{100}}\min_{\{\alpha,\beta,\gamma\}=\{a,b,c\}}\Vert F^{\alpha}\Vert_{L_{x,y}^2}\Vert F^{\beta}\Vert_S\Vert F^{\gamma}\Vert_S\ .
\end{align}
Therefore, an application of Lemma~\ref{SS+} shows that the $S$ norms of $S_3$ is controlled as follows, 
\begin{equation}
\Vert \mathcal{E}_3(t)\Vert_S\lesssim  (\vert T\vert)^{-1+\frac{\delta}{100}}\Vert F^a\Vert_S\Vert F^b\Vert_S\Vert F^c\Vert_S\lesssim (\vert T\vert)^{-1+\frac{\delta}{100}}\ ,
\end{equation}
the last inequality comes from \eqref{BA}. Combining with inequality \eqref{BA+}, we can also gain
\begin{equation}
\Vert S_3(t)\Vert_{S^+}\lesssim (\vert T\vert)^{-1+\frac{\delta}{100}}\ .
\end{equation}

\medskip
\noindent {\bf 2. Estimation of $\widetilde{\mathcal{E}}_1(t)$.} Again, we need to control the $L^2$ norm first, and then the $S$ norm. $\widetilde{\mathcal{E}}_1(t)$ is composed by two parts, one is from $\mathcal{O}_1^t$, and the other one $\mathcal{L}^t$ is from the last four terms in \eqref{O2}, 
\begin{equation}
\mathcal{F}\widetilde{\mathcal{E}}^t_1[F^a,F^b,F^c](\xi,p):=\sum_{\omega\ne 0}\sum_{(p,q,r,s)\in\Gamma_\omega}{\mathrm{e}}^{it\omega}\big(\mathcal{O}_1^t[F_q^a,F_r^b,F_s^c]+\mathcal{L}^t[F_q^a,F_r^b,F_s^c]\big)\ ,
\end{equation}
with
$$i\omega \mathcal{L}^t[f^a,f^b,f^c]:=-\left(\partial_t\mathcal{O}_2^t\right)[f^a, f^b, f^c]-\mathcal{O}_2^t[\partial_t f^a, f^b, f^c]-\mathcal{O}_2^t[f^a,\partial_t f^b, f^c]-\mathcal{O}_2^t[f^a, f^b,\partial_t f^c]\ .$$

The term $\sum\limits_{\omega\ne 0}\sum_{(p,q,r,s)\in\Gamma_\omega}{\mathrm{e}}^{it\omega}\mathcal{L}^t[F_q^a,F_r^b,F_s^c]$ can be estimated similarly as $\Vert \mathcal{E}_3(t)\Vert_S$. Actually, we may gain a better estimate here, since for the first term, we can get an extra $T^{-1/4}$ which comes from the $t$ derivative of the multiplier, while for the other three terms, by the definition of $X_T$ norm, we have $\Vert\partial_t F\Vert_S\leq T^{-1+3\delta}\Vert F\Vert_{X_T}$. Let us focus on 
$$\sum\limits_{\omega\ne 0}\sum\limits_{(p,q,r,s)\in\Gamma_\omega}{\mathrm{e}}^{it\omega}\mathcal{O}_1^t[F_q^a,F_r^b,F_s^c]\ .$$

We claim that
\begin{equation}\label{O1small}
\Big\Vert\sum\limits_{\omega\ne 0}\sum\limits_{(p,q,r,s)\in\Gamma_\omega}{\mathrm{e}}^{it\omega}\mathcal{O}^t_1[F_q^a,F_r^b,F_s^c]\Big\Vert_{L^2}\lesssim   T^{-1-\delta}\min_{\{\alpha,\beta,\gamma\}=\{a,b,c\}}\Vert F^\alpha\Vert_{L^2_{x,y}}\Vert F^\beta\Vert_{S}\Vert F^\gamma\Vert_{S}\ .
\end{equation}
As we did for $\mathcal{O}_2^t$, we still have
\begin{equation}
\Vert\mathcal{O}^t_1[f^a,f^b,f^c]\Vert_{L^2}\lesssim  (1+|t|)^{\delta/100}\min_{\{\alpha,\beta,\gamma\}=\{a,b,c\}}\Vert f^\alpha\Vert_{L^2}\Vert{\mathrm{e}}^{it\partial_{xx}}f^\beta\Vert_{L^\infty}\Vert {\mathrm{e}}^{it\partial_{xx}}f^\gamma\Vert_{L^\infty}\ .
\end{equation}

We then need to estimate $\Vert {\mathrm{e}}^{it\partial_{xx}}f\Vert_{L^\infty_x}$. We notice that for all $\frac12<\alpha\leq1$,
\begin{equation}\label{o11}
\Vert {\mathrm{e}}^{it\partial_{xx}}f\Vert _{L^\infty(\R)}
\lesssim \langle t\rangle^{-\frac12}\Vert f\Vert _{L^1(\R)}
\lesssim  \langle t\rangle^{-\frac12}\Vert\langle x\rangle^{-\alpha}\langle x\rangle^{\alpha}f\Vert_{L^1(\R)}
\lesssim \langle t\rangle^{-\frac12}\Vert\langle x\rangle^{\alpha}f\Vert_{L^2(\R)} \ ,
\end{equation}
we may take $\alpha=7/9$, then for $f$ supported on $\vert x\vert\geq R$, 
\begin{align}\label{d1}
\Vert {\mathrm{e}}^{it\partial_{xx}}f\Vert_{L^\infty}\lesssim \langle t\rangle^{-\frac12}R^{-1/9}\Vert\langle x\rangle^{8/9}f\Vert_{L^2}\ .
\end{align}

Therefore, we decompose $f=f_c+f_e$ with $f_c(x):=\varphi(\frac{x}{T^{1/4}})f(x)$, then 
$$\mathcal{O}_1^t[f^a,f^b,f^c]=\mathcal{O}_1^t[f_c^a+{f^a_e},f_c^b+{f^b_e},f_c^c+{f^c_e}]\ .$$
then by \eqref{d1}, if one of $f^a,f^b,f^c$ is supported on $|x|\geq 2T^{1/4}$, for example, $f^b=f_e^b$, then
\begin{align*}
\Vert \mathcal{O}_1^t[f^a,f_e^b,f^c]\Vert_{L^2_x}=&\lesssim(1+|t|)^{\delta/100}\Vert f^a\Vert_{L^2}\Vert {\mathrm{e}}^{it\partial_{xx}}f_e^b\Vert_{L^\infty}\Vert {\mathrm{e}}^{it\partial_{xx}}f^c\Vert_{L^\infty}\\
&\lesssim (1+|t|)^{\delta/100}\Vert f^a\Vert_{L^2}\Vert {\mathrm{e}}^{it\partial_{xx}}f_e^b\Vert_{L^\infty}\Vert {\mathrm{e}}^{it\partial_{xx}}f^c\Vert_{L^\infty}\\
&\lesssim T^{-1-1/36+\delta/100}\Vert f^a\Vert_{L^2}\Vert \langle x\rangle^{8/9}f_e^b\Vert_{L^2}\Vert\langle x\rangle^{7/9}f^c\Vert_{L^2}\ ,
\end{align*}
in the last inequality comes from \eqref{o11} and \eqref{d1}. Then using \eqref{sump},
\begin{equation}
\begin{split}
&\Big\Vert\sum\limits_{\omega\ne 0}\sum\limits_{(p,q,r,s)\in\Gamma_\omega}{\mathrm{e}}^{it\omega}\mathcal{O}^t_1[F_q^a,F_{r,e}^b,F_s^c]\Big\Vert_{L^2}\\
&\lesssim  T^{-1-1/36+\delta/100}\Vert F^a\Vert_{L^2}\sum_r\Vert\langle x\rangle^{8/9}F_r^b\Vert_{L^2_x}\sum_s\Vert\langle x\rangle^{7/9}F_s^c\Vert_{L^2_x}\ .
\end{split}
\end{equation}
For $0<\alpha<1$, 
\begin{equation}\label{e11}
\sum_r\Vert x^{\alpha}F_r\Vert_{L^2_x}\lesssim \Vert F\Vert_{S}\ ,
\end{equation}
indeed,
\begin{align*}
\sum_r\Vert \langle x\rangle^{\alpha}F_r\Vert_{L^2_x}&=\sum_r\Vert( \langle x\rangle F_r)^{\alpha}F_r^{1-\alpha}\Vert_{L^2_x}\leq \sum_r \Vert \langle x\rangle F_r\Vert_{L^2}^{\alpha}\Vert F_r\Vert_{L^2}^{1-\alpha}\\
& \leq \sum_r \Vert \langle x\rangle F_r\Vert_{L^2}^{\alpha}\langle r\rangle^s\Vert F_r\Vert_{L^2}^{1-\alpha}\langle r\rangle^{-s}\leq\Vert \langle x\rangle F\Vert_{L_{x,y}^2}^{\alpha}\Vert F\Vert_{H_{x,y}^{\frac{s}{1-\alpha}}}^{1-\alpha}\leq \Vert F\Vert_S\ ,
\end{align*}
with $s>1/2$. Thus
\begin{equation}
\begin{split}
\Big\Vert\sum\limits_{\omega\ne 0}\sum\limits_{(p,q,r,s)\in\Gamma_\omega}{\mathrm{e}}^{it\omega}\mathcal{O}^t_1[F_q^a,F_{r,e}^b,F_s^c]\Big\Vert_{L^2}\lesssim  T^{-1-1/36+\delta/100}\Vert F^a\Vert_{L^2}\Vert F^b\Vert_{S}\Vert F^c\Vert_{S}\ .
\end{split}
\end{equation}
Let us turn to the case $ \mathcal{O}_1^t[f^a,f_c^b,f_c^c]$. By replacing ${\mathrm{e}}^{2it\eta\kappa}$ by $(2it\eta)^{-1}\partial_\kappa ({\mathrm{e}}^{2it\eta\kappa})$, we can rewrite $\mathcal{O}_1^t$ as
\begin{equation}
\begin{split}
&\mathcal{O}^t_1[f^a,f^b,f^c](\xi)=\int_{\mathbb{R}^2}{\mathrm{e}}^{2it\eta\kappa}(1-\varphi(t^{\frac{1}{4}}\eta\kappa))\widehat{f^a}(\xi-\eta)\overline{\widehat{f^b}}(\xi-\eta-\kappa)\widehat{f^c}(\xi-\kappa)d\eta d\kappa\\
&\qquad=\int_{\mathbb{R}^2}(2it\eta)^{-1}\partial_\kappa ({\mathrm{e}}^{2it\eta\kappa})(1-\varphi(t^{\frac{1}{4}}\eta\kappa))\widehat{f^a}(\xi-\eta)\overline{\widehat{f^b}}(\xi-\eta-\kappa)\widehat{f^c}(\xi-\kappa)d\eta d\kappa\\
&\qquad=\int_{\mathbb{R}^2}(2it\eta)^{-1}{\mathrm{e}}^{2it\eta\kappa}\partial_\kappa \big((1-\varphi(t^{\frac{1}{4}}\eta\kappa))\widehat{f^a}(\xi-\eta)\overline{\widehat{f^b}}(\xi-\eta-\kappa)\widehat{f^c}(\xi-\kappa)\big)d\eta d\kappa\ .
\end{split}
\end{equation}
Firstly, it is easy to deal with the case when the $\kappa$ derivative falls on $1-\varphi$, which turns out to be
$$(2i)^{-1}t^{-3/4}\int_{\mathbb{R}^2}{\mathrm{e}}^{2it\eta\kappa}\varphi'(t^{\frac{1}{4}}\eta\kappa)\widehat{f^a}(\xi-\eta)\overline{\widehat{f^b}}(\xi-\eta-\kappa)\widehat{f^c}(\xi-\kappa)d\eta d\kappa\ ,$$ 
then we get the required estimate with the similar strategy we used to estimate $\mathcal{O}_2^t$ since $\varphi'$ admits similar properties as $\varphi$. 

For the other case, we calculate the case when $\kappa$ derivative falls on $f^b$ for example, which is denoted by $\mathcal{O}_{1,b}$,
\begin{equation}\begin{split}
\mathcal{O}_{1,b}&:=\int_{\mathbb{R}^2}(2it\eta)^{-1}{\mathrm{e}}^{2it\eta\kappa}(1-\varphi(t^{\frac{1}{4}}\eta\kappa))\widehat{f^a}(\xi-\eta)\partial_\kappa\big( \overline{\widehat{f^b}}(\xi-\eta-\kappa)\big)\widehat{f^c}(\xi-\kappa)d\eta d\kappa\\
&=\int_{\mathbb{R}^2}(2it\eta)^{-1}{\mathrm{e}}^{2it\eta\kappa}(1-\varphi(t^{\frac{1}{4}}\eta\kappa))\widehat{f^a}(\xi-\eta)\overline{\widehat{xf^b}}(\xi-\eta-\kappa)\widehat{f^c}(\xi-\kappa)d\eta d\kappa\ .
\end{split}\end{equation}
Noticing that on the support of the integration, $ |t||\eta|\gtrsim |t|^{-3/4}|\kappa|^{-1}\gtrsim T^{-7/12}$, we still have an $L^2$ estimate
\begin{equation}\label{o12}
\Vert \mathcal{O}_{1,b}\Vert_{L^2_\xi}\lesssim T^{-7/12+\frac{\delta}{100}}\Vert f^a \Vert_{L^2}\cdot\Vert{\mathrm{e}}^{it\partial_{xx}}(xf^b)\Vert_{L^\infty}\cdot \Vert{\mathrm{e}}^{it\partial_{xx}}f^c\Vert_{L^\infty}\ .
\end{equation}
By \eqref{o11}, for $f$ supported on $\vert x\vert\leq T^{1/4}$, we have
\begin{align}\label{d2}
\Vert {\mathrm{e}}^{it\partial_{xx}}xf\Vert_{L^\infty}&\lesssim \langle t\rangle^{-\frac12}T^{1/4}\Vert \langle x\rangle^{7/9} f\Vert_{L^2}\ .
\end{align}
using \eqref{o11} and \eqref{d2},
\begin{align*}
\Vert\mathcal{O}_{1,b}\Vert_{L^2_\xi}&\lesssim T^{-7/12+\frac{\delta}{100}}\Vert f^a \Vert_{L^2}\cdot\Vert{\mathrm{e}}^{it\partial_{xx}}(xf_c^b)\Vert_{L^\infty}\cdot \Vert{\mathrm{e}}^{it\partial_{xx}}f^c\Vert_{L^\infty}\\
&\lesssim T^{-4/3+\frac{\delta}{100}}\Vert f^a\Vert_{L_x^2}\Vert \langle x\rangle^{7/9}f^b\Vert_{L_x^2}\Vert \langle x\rangle^{7/9}f^c\Vert_{L_x^2}\ .
\end{align*}
Once again we use \eqref{e11},
\begin{equation}
\begin{split}
&\Big\Vert\sum\limits_{\omega\ne 0}\sum\limits_{(p,q,r,s)\in\Gamma_\omega}{\mathrm{e}}^{it\omega}\mathcal{O}^t_1[F_q^a,F_r^b,F_s^c]\Big\Vert_{L^2}\\&\lesssim   T^{-4/3+\frac{\delta}{100}}\Big\Vert\sum\limits_{\omega\ne 0}\sum\limits_{(p,q,r,s)\in\Gamma_\omega}\Vert F_q^a\Vert_{L_x^2}\Vert \langle x\rangle^{7/9}F_r^b\Vert_{L_x^2}\Vert \langle x\rangle^{7/9}F_s^c\Vert_{L_x^2}\Big\Vert_{\ell_p^2}\\
&\lesssim T^{-4/3+\frac{\delta}{100}} \Vert F^a\Vert_{L_{x,y}^2}\Vert F^b\Vert_S\Vert F^c\Vert_S\ .
\end{split}
\end{equation}
By replacing $F^a$ by $F^b$ or $F^c$, we proved \eqref{O1small} and then the estimate of $\widetilde{\mathcal{E}}_1(t)$.
\end{proof}
\begin{lemma}\label{multiplier}\cite[Remark 3.5]{HPTV}
For $T>1$, $\varphi\in C_c^\infty(\R)$, $\varphi(x)=1$ when $\vert x \vert\leq 1$ and $\varphi(x)=0$ when $\vert x\vert\geq2$, we define for $T/2\leq t\leq T$,
$$\widetilde{m}(\eta,\kappa):=\varphi(t^{\frac{1}{4}}\eta\kappa)\varphi((10T)^{\frac{-1}{6}}\eta)\varphi((10T)^{\frac{-1}{6}}\kappa)\ .$$
Then $\|\mathcal{F}_{\eta\kappa}\widetilde{m}\|_{L^1(\R^2)}\lesssim t^{\frac{\delta}{100}}$ .
\end{lemma}
\begin{proof}
$$
\Vert \mathcal{F}_{\eta\kappa}\widetilde{m}\Vert_{L^1(\mathbb{R}^2)}=\Vert I(x_1,x_2)\Vert_{L^1_{x_1,x_2}}\ ,
$$
where
$$
I(x_1,x_2)=\int_{\mathbb{R}^2}{\mathrm{e}}^{ix_1\eta}{\mathrm{e}}^{ix_2\kappa}\varphi(S\eta\kappa)\varphi(\eta)\varphi(\kappa)d\eta d\kappa,\quad S\approx T^\frac{7}{12}\ .
$$
Then one may show that
\begin{equation*}
\vert I(x_1,x_2)\vert+\vert x_1I(x_1,x_2)\vert+\vert x_2I(x_1,x_2)\vert \lesssim 1,\quad
\vert x_1x_2I(x_1,x_2)\vert\lesssim \log(1+T)\ .
\end{equation*}
Indeed, 
\begin{align*}
\vert x_1 I(x_1,x_2)\vert&=\left\vert\int_{\mathbb{R}^2}\frac{1}{i}\partial_{\eta}({\mathrm{e}}^{ix_1\eta}){\mathrm{e}}^{ix_2\kappa}\varphi(S\eta\kappa)\varphi(\eta)\varphi(\kappa)d\eta d\kappa\right\vert\\ 
&=\left\vert\int_{\mathbb{R}^2}{\mathrm{e}}^{ix_1\eta}{\mathrm{e}}^{ix_2\kappa}[S\kappa\varphi'(S\eta\kappa)\varphi(\eta)\varphi(\kappa)+\varphi(S\eta\kappa)\varphi'(\eta)\varphi(\kappa)]d\eta d\kappa\right\vert\\
&\lesssim 1+\left\vert\int_{\mathbb{R}^2}{\mathrm{e}}^{ix_1\eta}{\mathrm{e}}^{ix_2\kappa}(S\kappa\varphi'(S\eta\kappa)\varphi(\eta)\varphi(\kappa))d\eta d\kappa\right\vert\ .
\end{align*}
Notice that $\vert S\eta\kappa\vert\leq 2$, then the second term turns out to be
\begin{align*}
\left\vert\int_{\mathbb{R}^2}{\mathrm{e}}^{ix_1\eta}{\mathrm{e}}^{ix_2\kappa}(S\kappa\varphi'(S\eta\kappa)\varphi(\eta)\varphi(\kappa))d\eta d\kappa\right\vert
\lesssim \int_{D:=\{\vert S\eta\kappa\vert, \vert\eta\vert,\vert\kappa\vert\leq 2\}}\vert S\kappa\vert d\eta d\kappa
\lesssim 1.
\end{align*}
Thus we get the first inequality, and we use the similar strategy to prove the second one.
 \begin{align*}
&\vert x_1x_2I(x_1,x_2)\vert\lesssim\int_{\mathbb{R}^2}\vert\partial_{\eta}\partial_{\kappa}\big(\varphi(S\eta\kappa)\varphi(\eta)\varphi(\kappa)\big)\vert d\eta d\kappa\\
&\qquad\lesssim\int_{D}\vert S\kappa\varphi'(S\eta\kappa)\varphi(\eta)\varphi'(\kappa)\vert+\vert S\kappa S\eta\varphi''(S\eta\kappa)\varphi(\eta)\varphi(\kappa)\vert\\
&\qquad\qquad+\vert S\eta \varphi'(S\eta\kappa)\varphi'(\eta)\varphi(\kappa)\vert +\vert\varphi(S\eta\kappa)\varphi'(\eta)\varphi'(\kappa)\vert+\vert S\varphi'(S\eta\kappa)\varphi(\eta)\varphi(\kappa)\vert d\eta d\kappa\\
&\qquad\lesssim (\int_0^{T^{-7/12}}\int_0^2+\int_{T^{-7/12}}^2\int_0^{\frac{2T^{-7/12}}{\kappa}}) [1+\vert S\kappa\vert +\vert S\eta\vert+\vert S\kappa S\eta\vert+S]d\eta d\kappa\\
&\qquad\lesssim \log(1+T)\ .
\end{align*}
Then
$$(1+|x_1|)(1+|x_2|)|I(x_1,x_2)|\lesssim \log(1+T)\ .$$
One also have a polynomial in $T$ bound
$$(1+|x_1|^2)(1+|x_2|^2)|I(x_1,x_2)|\lesssim T^{7/12}\ .$$
Therefore by interpolation one obtains that
for every $0<\varepsilon<7/12$, there exists $\kappa>1/2$ such that
$$
|I(x_1,x_2)|\lesssim (1+T)^{\varepsilon}(1+|x_1|^2)^{-\kappa}(1+|x_2|^2)^{-\kappa}\ . 
$$
We hence deduce that $\|\mathcal F_{\eta\kappa} \widetilde{m}\|_{L^1(\mathbb{R}^2)}\lesssim t^{\frac{\delta}{100}}$. 
\end{proof}
\subsection{The Resonant Level sets}\mbox{}\\
\indent We now turn to the contribution of the resonant part in \eqref{Dec2},
\begin{equation*}
\mathcal{F}\mathcal{N}_0^t[F,G,H](\xi,p)=\sum_{(p,q,r,s)\in\Gamma_0}\mathcal{F}_x\mathcal{I}^t[F_q(t),G_r(t),H_s(t)](\xi).
\end{equation*}
This term yields the main contribution in Proposition \ref{Nonlinearity} and in particular is responsible for the slowest $1/t$ decay. We show that it gives rise to a contribution which grows slowly in $S$, $S^+$ and that it can be well approximated by the resonant system in the $Z$ norm.

In this subsection, we will bound quantities in terms of
\begin{equation*}
\Vert F\Vert_{\widetilde{Z}_t}:=\Vert F\Vert_Z+(1+\vert t\vert)^{-\delta}\Vert F\Vert_S\ ,                                                    
\end{equation*}
so that $F(t)$ remains uniformly bounded in $\widetilde{Z}_t$ under the assumption of Proposition \ref{Nonlinearity} due to the definition of $X_T$ and $\widetilde{Z}_t$ norm. Our main statement of this subsection is as follows.
\begin{lemma}\label{Res}
Let $t\ge 1$. There holds that
\begin{equation}\label{n01}
\Vert \mathcal{N}_0^t[F^a,F^b,F^c]\Vert_{S}\lesssim (1+\vert t\vert)^{-1}\sum_{\{\alpha,\beta,\gamma\}=\{a,b,c\}}\Vert F^\alpha\Vert_{\widetilde{Z}_t}\cdot \Vert F^\beta\Vert_{\widetilde{Z}_t}\cdot \Vert F^\gamma\Vert_{S}
\end{equation}
and
\begin{equation}\label{n02}
\begin{split}
\Vert \mathcal{N}_0^t[F^a,F^b,F^c]\Vert_{S^+}\lesssim& (1+\vert t\vert)^{-1}\sum_{\{\alpha,\beta,\gamma\}=\{a,b,c\}}\Vert F^\alpha\Vert_{\widetilde{Z}_t}\cdot \Vert F^\beta\Vert_{\widetilde{Z}_t}\cdot \Vert F^\gamma\Vert_{S^+}\\
&+(1+\vert t\vert)^{-1+2\delta}\sum_{\{\alpha,\beta,\gamma\}=\{a,b,c\}}\Vert F^\alpha\Vert_{\widetilde{Z}_t}\cdot \Vert F^\beta\Vert_{S}\cdot \Vert F^\gamma\Vert_{S}.
\end{split}
\end{equation}
Moreover,
\begin{equation}\label{n03}
\Vert \mathcal{N}_0^t[F,G,H]-\frac{\pi}{t}\mathcal{R}[F,G,H]\Vert_{Y^s}\lesssim (1+\vert t\vert)^{-1-20\delta}\Vert F\Vert_S\Vert G\Vert_S\Vert H\Vert_S\ .
\end{equation}
and
\begin{equation}\label{n04}
\Vert \mathcal{N}_0^t[F,G,H]-\frac{\pi}{t}\mathcal{R}[F,G,H]\Vert_{S}\lesssim (1+\vert t\vert)^{-1-20\delta}\Vert F\Vert_{S^+}\Vert G\Vert_{S^+}\Vert H\Vert_{S^+}\ .
\end{equation}
In addition, we also have
\begin{align}
\Vert \mathcal{R}[F^a,F^b,F^c]\Vert_{S}&\lesssim \sum_{\{\alpha,\beta,\gamma\}=\{a,b,c\}}\Vert F^\alpha\Vert_{\widetilde{Z}_t}\cdot \Vert F^\beta\Vert_{\widetilde{Z}_t}\cdot \Vert F^\gamma\Vert_{S}\label{n05}\\
\Vert \mathcal{R}[F^a,F^b,F^c]\Vert_{S^+}\lesssim& \sum_{\{\alpha,\beta,\gamma\}=\{a,b,c\}}\Vert F^\alpha\Vert_{\widetilde{Z}_t}\cdot \Vert F^\beta\Vert_{\widetilde{Z}_t}\cdot \Vert F^\gamma\Vert_{S^+}\label{n06}\\
&+(1+\vert t\vert)^{2\delta}\sum_{\{\alpha,\beta,\gamma\}=\{a,b,c\}}\Vert F^\alpha\Vert_{\widetilde{Z}_t}\cdot \Vert F^\beta\Vert_{S}\cdot \Vert F^\gamma\Vert_{S}\notag.
\end{align}
\end{lemma}

\begin{proof}
As before, we will study the $L^2$ norm and then apply Lemma~\ref{SS+} to get the $S^{(+)}$ norm estimate. Using \eqref{sump}, we have
\begin{align*}
\Vert \mathcal{N}_0^t[F^a,F^b,F^c]\Vert_{L_{x,y}^2}&\leq \Big\Vert \sum_{(p,q,r,s)\in\Gamma_0}\vert {\mathrm{e}}^{it\partial_{xx}}F^a_q\vert\cdot\vert {\mathrm{e}}^{it\partial_{xx}}F^b_r\vert\cdot\vert {\mathrm{e}}^{it\partial_{xx}}F^c_s\vert\Big\Vert_{\ell_p^2L_x^2}\notag\\
&\lesssim\min_{\{\alpha,\beta,\gamma\}=\{a,b,c\}}\Vert F^\alpha\Vert_{L_{x,y}^2}\sum_p\Vert {\mathrm{e}}^{it\partial_{xx}}F_p^\beta\Vert_{L_x^{\infty}}\sum_p\Vert {\mathrm{e}}^{it\partial_{xx}}F_p^\gamma\Vert_{L_x^{\infty}}\ .
\end{align*}

To calculate $\sum\limits_p\Vert {\mathrm{e}}^{it\partial_{xx}}F_p\Vert_{L_x^{\infty}}$, we start with the following estimate for $|t|>1$,
\begin{equation}\label{f1}
\Big\vert {\mathrm{e}}^{it\partial_{xx}}f(x)-c\frac{{\mathrm{e}}^{-ix^2/(4t)}}{\sqrt{t}}\widehat{f}(-\frac{x}{2t})\Big\vert\lesssim |t|^{-3/4}\Vert xf\Vert_{L^2}\ , c \text{ is a constant}.
\end{equation}
One may refer to \cite[Lemma 7.3]{HPTV} for the proof of \eqref{f1}. Then
\begin{equation}\label{f2}
\vert {\mathrm{e}}^{it\partial_{xx}}f(x)\vert\lesssim |t|^{-1/2}\sup_\xi \vert\widehat{f}(\xi)\vert+|t|^{-3/4}\Vert xf\Vert_{L^2}\ .
\end{equation}
Then
\begin{align*}
\sum_{|p|\leq t^{1/8}}\Vert {\mathrm{e}}^{it\partial_{xx}}F_p\Vert_{L_x^{\infty}}&\lesssim t^{-1/2}\sup_\xi \sum_p\vert\widehat{F_p}(\xi)\vert+t^{-3/4}\sum_{|p|\leq t^{1/8}}\Vert xF_p\Vert_{L^2}\\
&\lesssim t^{-1/2}\Vert F\Vert_Z+t^{-5/8}\Vert xF\Vert_{S}\ ,
\end{align*}
while
\begin{align*}
\sum_{|p|>t^{1/8}}\Vert {\mathrm{e}}^{it\partial_{xx}}F_p\Vert_{L_x^{\infty}}&\lesssim \sum_{|p|>t^{1/8}}\Vert F_p\Vert_{H^1}\\
&= \sum_{|p|>t^{1/8}}(1+|p|^2)^{N/2}\Vert F_p\Vert_{H^1}(1+|p|^2)^{-N/2}\\&\lesssim t^{-\frac{2N-1}{16}}\Vert F\Vert_{H^{N+1}_{x,y}}\ ,
\end{align*}
therefore
\begin{equation}\label{f3}
\Vert \mathcal{N}_{0}^t[F^a,F^b,F^c]\Vert_{L_{x,y}^2}\lesssim (1+\vert t\vert)^{-1}\min_{\{\alpha,\beta,\gamma\}=\{a,b,c\}} \Vert F^\alpha\Vert_{L_{x,y}^2}\Vert F^\beta\Vert_{\widetilde{Z}_t}\Vert F^\gamma\Vert_{\widetilde{Z}_t}\ .
\end{equation}
Apply the first part of  Lemma~\ref{SS+}, we get \eqref{n01}. The proof of  \eqref{n02} follows from the second part of  Lemma~\ref{SS+}, and we only need to check $\widetilde{ Z}_t$ norm satisfies \eqref{YB}.  Due to the definition of $S^+$, we only need to prove the following inequality,
\begin{equation}\label{f4}
\Vert (1-\partial_{xx})^4F\Vert_Z+\Vert xF\Vert_Z\lesssim T^{-\delta}\Vert F\Vert_{S^+}+T^{2\delta}\Vert F\Vert_S\ .
\end{equation}
Indeed, following the proof of \eqref{GN}, we are able to show 
$$\Vert (1-\partial_{xx})^4F\Vert_Z\lesssim \Vert F\Vert_S\ ,$$
thus we only need to prove \eqref{f4} for $\Vert xF\Vert_Z$. Since $H^s(\T)\subset B^1$ with $s>1$, then
\begin{align*}
\Vert xF\Vert_Z^2&=\sup_\xi (1+|\xi|^2)^2\Vert \widehat{xF}\Vert_{B_y^1}^2\\
&\lesssim\sup_M (1+M^2)\sum_p(1+|p|^2)^s\vert \mathcal{F}_xQ_M(xF_p)\vert^2\ .
\end{align*}
We notice that for any $M,|p|\ne 0$, denote $R=(1+M^2)(1+|p|^2)^sT^{2\delta}$, and we decompose $xF_p(x)$ into two parts
$$x(1-\varphi(\frac{x}{R}))F_p\text{ and }x\varphi(\frac{x}{R})F_p\ ,$$
then
\begin{align*}
\Vert xF\Vert_Z^2&\lesssim\sup_M (1+M^2)\sum_p(1+|p|^2)^s\Vert \mathcal{F}_xQ_M\big\{x(1-\varphi(\frac{x}{R}))F_p\big\}\Vert_{L^\infty_\xi}^2:=\mathrm{I}\\
&\qquad+\sup_M (1+M^2)\sum_p(1+|p|^2)^s\Vert\mathcal{F}_xQ_M\big\{x\varphi(\frac{x}{R})F_p\big\}\Vert_{L^\infty_\xi}:=\mathrm{II}\ .
\end{align*}

\begin{align*}
\mathrm{I}&\lesssim\sup_M (1+M^2)\sum_p(1+|p|^2)^s\Vert xF_p\Vert_{L^1_x(|x|>R)}^2\\
&\lesssim \sup_M (1+M^2)\sum_p(1+|p|^2)^sR^{-1}\Vert x^2F_p\Vert_{L^2}^2\\
&\lesssim T^{-2\delta}\Vert x^2F\Vert_{L^2}\leq T^{-2\delta}\Vert F\Vert_{S^+}\ ,
\end{align*}
while
\begin{align*}
\mathrm{II}&\lesssim \sup_M (1+M^2)\sum_p(1+|p|^2)^s\Vert Q_M\big\{ x\varphi(\frac{x}{R})F_p\big\}\Vert_{L^1_x(|x|\leq R)}^2\\
&\lesssim\sup_M (1+M^2)\sum_p(1+|p|^2)^sR^2\Vert F_p\Vert_{L^2_x}\Vert xF_p\Vert_{L^2_x}\\
&\lesssim \sum_p T^{4\delta}\Vert F_p\Vert_{L^2_x}\Vert xF_p\Vert_{L^2_x}\lesssim T^{4\delta}\Vert F\Vert_{S}^2\ ,
\end{align*}
 thus we proved \eqref{f4}, the estimate on $\mathcal{R}$ is the same.

Now we turn to the proof of the error estimates \eqref{n03} and \eqref{n04}. 
We first decompose the functions as we did for estimating $\mathcal{O}_1^t$, 
\begin{equation*}
F=F_c+F_e,\ \text{ with }F_c \text{ compactly supported as }F_c=\varphi(\frac{x}{t^{1/4}})F\ ,
\end{equation*}
and reduce the problem to the estimates on $F_c,G_c,H_c$. We start with the $L^2$ estimates of
$$\mathcal{N}_0^t[F,G,H]-\mathcal{N}_0^t[F_c,G_c,H_c] \quad \text{ and }\quad \mathcal{R}[F,G,H]-\mathcal{R}[F_c,G_c,H_c] \ ,$$
without loss of generalities, it suffices to consider $\mathcal{N}_0^t[F_e,G,H] \text{ and }\frac{1}{t} \mathcal{R}[F_e,G,H]$.
Indeed, using \eqref{f3} and the definition of $F_e$,
\begin{align}
\Vert \mathcal{N}_0^t[F_e,G,H]\Vert_{L^2} +\frac{1}{t}\Vert \mathcal{R}[F_e,G,H]\Vert_{L^2}&\lesssim (1+|t|)^{-1}\Vert F_e\Vert_{L^2}\Vert G\Vert_S\Vert H\Vert_S\notag\\
&\lesssim (1+|t|)^{-5/4}\Vert F\Vert_S\Vert G\Vert_S\Vert H\Vert_S\ ,
\end{align}
while
\begin{align}\label{f5}
\Vert \mathcal{N}_0^t[F_e,G,H]\Vert_{S} +\frac{1}{t}\Vert \mathcal{R}[F_e,G,H]\Vert_{S}\lesssim (1+|t|)^{-1}\Vert F\Vert_{S}\Vert G\Vert_S\Vert H\Vert_S\ .
\end{align}
Thus by \eqref{ZSNorm}, we are able to bound
\begin{equation}
\begin{split}
&\Vert\mathcal{N}_0^t[F,G,H]-\mathcal{N}_0^t[F_c,G_c,H_c]\Vert_{Y^s}+\frac{1}{t}\Vert \mathcal{R}[F,G,H]-\mathcal{R}[F_c,G_c,H_c] \Vert_{Y^s}\\
&\qquad\lesssim (1+|t|)^{-17/16}\Vert F\Vert_{S}\Vert G\Vert_S\Vert H\Vert_S\ .
\end{split}
\end{equation}
For the $S$ norm estimate, we use \eqref{f5} again,
\begin{equation}
\begin{split}
\Vert \mathcal{N}_0^t[F_e,G,H]\Vert_{S} +\frac{1}{t}\Vert \mathcal{R}[F_e,G,H]\Vert_{S}&\lesssim (1+|t|)^{-1}\Vert F_e\Vert_{S}\Vert G\Vert_S\Vert H\Vert_S\\
&\lesssim (1+|t|)^{-5/4}\Vert F\Vert_{S^+}\Vert G\Vert_{S^+}\Vert H\Vert_{S^+}\ .
\end{split}
\end{equation}

Therefore, we only need to show the inequalities below to complete our proof of this lemma,
\begin{align}
&\Vert \mathcal{N}_0^t[F_c,G_c,H_c]-\frac{\pi}{t}\mathcal{R}[F_c,G_c,H_c]\Vert_{Y^s}\lesssim (1+\vert t\vert)^{-1-20\delta}\Vert F\Vert_S\Vert G\Vert_S\Vert H\Vert_S\ ,\label{NR1}\\
&\Vert \mathcal{N}_0^t[F_c,G_c,H_c]-\frac{\pi}{t}\mathcal{R}[F_c,G_c,H_c]\Vert_{S}\lesssim (1+\vert t\vert)^{-1-20\delta}\Vert F\Vert_{S^+}\Vert G\Vert_{S^+}\Vert H\Vert_{S^+}\label{NR2}\ .
\end{align}

For abbreviation, we assume for the rest part of proof, $F=F_c,G=G_c,H=H_c$. 
\begin{equation}
\begin{split}
&\mathcal{F}\big(\mathcal{N}_0^t[F,G,H]-\frac{\pi}{t}\mathcal{R}[F,G,H]\big)(\xi,p)\\
=&\sum_{(p,q,r,s)\in\Gamma_0}\int_{\mathbb{R}^2}e^{it2\eta\kappa}\widehat{F_q}(\xi-\eta)\overline{\widehat{G_r}}(\xi-\eta-\kappa)\widehat{H_s}(\xi-\kappa)d\kappa d\eta- \frac{\pi}{t}\widehat{F_q}(\xi)\overline{\widehat{G_r}}(\xi)\widehat{H_s}(\xi)\ .
\end{split}
\end{equation}
Rewrite the integration part,
\begin{equation*}
\begin{split}
&\int_{\mathbb{R}^2}{\mathrm{e}}^{it2\eta\kappa}\widehat{F_q}(\xi-\eta)\overline{\widehat{G_r}}(\xi-\eta-\kappa)\widehat{H_s}(\xi-\kappa)d\kappa d\eta\\
&= \int_{\mathbb{R}^3}F_q(x_1)\overline{G_r}(x_2)H_s(x_3)\int_{\mathbb{R}^2}{\mathrm{e}}^{it2\eta\kappa}{\mathrm{e}}^{-ix_1(\xi-\eta)-ix_2(\xi-\eta-\kappa)-ix_3(\xi-\kappa)}d\kappa d\eta dx_1dx_2dx_3\\
&=\frac{1}{2t}\int_{\mathbb{R}^3}F_q(x_1)\overline{G_r}(x_2)H_s(x_3){\mathrm{e}}^{-i\xi(x_1-x_2+x_3)}{\mathrm{e}}^{-i\frac{x_1-x_2}{\sqrt{2t}}\frac{x_3-x_2}{\sqrt{2t}}}\\
&\qquad\qquad\qquad\qquad\qquad\qquad\qquad\qquad\left\{\int_{\mathbb{R}^2}{\mathrm{e}}^{i\left[\eta+\frac{x_3-x_2}{\sqrt{2t}}\right]\left[\kappa+\frac{x_1-x_2}{\sqrt{2t}}\right]}d\eta d\kappa\right\} dx_1dx_2dx_3\\
&=\frac{\pi}{t}\int_{\mathbb{R}^3}F_q(x_1)\overline{G_r}(x_2)H_s(x_3){\mathrm{e}}^{-i\xi(x_1-x_2+x_3)}{\mathrm{e}}^{-i\frac{x_1-x_2}{\sqrt{2t}}\frac{x_3-x_2}{\sqrt{2t}}}dx_1dx_2dx_3\ ,
\end{split}
\end{equation*}
then
\begin{equation*}
\begin{split}
&\left\vert\int_{\mathbb{R}^2}{\mathrm{e}}^{it2\eta\kappa}\widehat{F_q}(\xi-\eta)\overline{\widehat{G_r}}(\xi-\eta-\kappa)\widehat{H_s}(\xi-\kappa)d\kappa d\eta- \frac{\pi}{t}\widehat{F_q}(\xi)\overline{\widehat{G_r}}(\xi)\widehat{H_s}(\xi)\right\vert\\
&=\frac{\pi}{|t|}\left\vert\int_{\mathbb{R}^3}F_q(x_1)\overline{G_r}(x_2)H_s(x_3){\mathrm{e}}^{-i\xi(x_1-x_2+x_3)}\Big({\mathrm{e}}^{-i\frac{x_1-x_2}{\sqrt{2t}}\frac{x_3-x_2}{\sqrt{2t}}}-1\Big)dx_1dx_2dx_3\right\vert\\
&\lesssim |t|^{-11/10}\Vert F_q\Vert_{L_x^2}\Vert G_r\Vert_{L_x^2}\Vert H_s\Vert_{L_x^2}\ .
\end{split}
\end{equation*}
Actually, using the proof above, we may obtain for any integer $m$,
\begin{equation}\label{f6}
\begin{split}
&\vert \xi\vert^m\left\vert\int_{\mathbb{R}^2}{\mathrm{e}}^{it2\eta\kappa}\widehat{F_q}(\xi-\eta)\overline{\widehat{G_r}}(\xi-\eta-\kappa)\widehat{H_s}(\xi-\kappa)d\kappa d\eta- \frac{\pi}{t}\widehat{F_q}(\xi)\overline{\widehat{G_r}}(\xi)\widehat{H_s}(\xi)\right\vert\\
&\lesssim  |t|^{-11/10}\Vert F_q\Vert_{H_x^m}\Vert G_r\Vert_{L_x^2}\Vert H_s\Vert_{L_x^2}\ .
\end{split}
\end{equation}
Due to the definition of $Y^s$ norm \eqref{DefYs} and $S$ norm \eqref{DefS}, and the fact that $H^s(\mathbb{T})$ , $s>1$ is an algebra, the proof of \eqref{NR1} follows from \eqref{f6}. For \eqref{NR2}, recall that the functions are spectrally compacted supported, then the terms
$$\Vert \mathcal{N}_0^t[F_c,G_c,H_c]-\frac{\pi}{t}\mathcal{R}[F_c,G_c,H_c]\Vert_{L^2_xH_y^N} \text{ and } \Vert x\Big(\mathcal{N}_0^t[F_c,G_c,H_c]-\frac{\pi}{t}\mathcal{R}[F_c,G_c,H_c]\Big)\Vert_{L_{x,y}^2}$$
are easy to deal with by \eqref{f6} and \eqref{sump}. We should be more careful with the terms admitting $x$ derivatives, since this $x$ derivative may fall on $\varphi(\frac{x}{t^{1/4}})$. Anyhow, since $\varphi'$ holds the similar properties as $\varphi$, \eqref{f6} still works, and we are able to get the estimate \eqref{NR2}. The proof is complete.
\end{proof}

\subsection{Proof of Proposition~\ref{Nonlinearity}}
Now, we can give the proof of Proposition~\ref{Nonlinearity}. 
\begin{proof}[ Proof of Proposition~\ref{Nonlinearity}]
We may firstly decompose the non-linearity $\mathcal{N}^t$ as the high frequency part and then the lower frequency part combined with the resonant and non resonant parts,
\begin{equation*}
\begin{split}
&\mathcal{N}^t[F,G,H]=\sum_{\substack{A,B,C\\\max(A,B,C)\geq T^{\frac{1}{6}}}}\mathcal{N}^t[Q_AF(t),Q_BG(t),Q_CH(t)]\\
&\qquad+\widetilde{\mathcal N}^t[Q_{\leq T^{\frac{1}{6}}}F(t),Q_{\leq T^{\frac{1}{6}}}G(t),Q_{\leq T^{\frac{1}{6}}}H(t)]+
\mathcal{N}_0^t[Q_{\leq T^{\frac{1}{6}}}F(t),Q_{\leq T^{\frac{1}{6}}}G(t),Q_{\leq T^{\frac{1}{6}}}H(t)]\ .
\end{split}
\end{equation*}
Then, we rewrite the last term as
\begin{align*}
&\mathcal{N}_0^t[Q_{\leq T^{\frac{1}{6}}}F(t),Q_{\leq T^{\frac{1}{6}}}G(t),Q_{\leq T^{\frac{1}{6}}}H(t)]=\frac{\pi}{t}\mathcal R[F(t),G(t),H(t)]\\
&\qquad+\Big(\mathcal{N}_0^t[Q_{\leq T^{\frac{1}{6}}}F(t),Q_{\leq T^{\frac{1}{6}}}G(t),Q_{\leq T^{\frac{1}{6}}}H(t)]-\frac{\pi}{t}\mathcal{R}[Q_{\leq T^{\frac{1}{6}}}F(t),Q_{\leq T^{\frac{1}{6}}}G(t),Q_{\leq T^{\frac{1}{6}}}H(t)]\Big)\\
&\qquad-\frac{\pi}{t}\sum_{\substack{A,B,C\\\max(A,B,C)\geq T^{\frac{1}{6}}}}\mathcal{R}[Q_AF(t),Q_BG(t),Q_CH(t)]\ .
\end{align*}

Finally, we have the formula for the remainder
\begin{align*}
&\mathcal{E}^t[F,G,H]=\\
&\qquad\sum_{\substack{A,B,C\\\max(A,B,C)\ge T^{\frac{1}{6}}}}\mathcal{N}^t[Q_AF(t),Q_BG(t),Q_CH(t)]+\widetilde{\mathcal N}^t[Q_{\leq T^{\frac{1}{6}}}F(t),Q_{\leq T^{\frac{1}{6}}}G(t),Q_{\leq T^{\frac{1}{6}}}H(t)]\\
&\qquad+\Big(\mathcal{N}_0^t[Q_{\leq T^{\frac{1}{6}}}F(t),Q_{\leq T^{\frac{1}{6}}}G(t),Q_{\leq T^{\frac{1}{6}}}H(t)]-\frac{\pi}{t}\mathcal{R}[Q_{\leq T^{\frac{1}{6}}}F(t),Q_{\leq T^{\frac{1}{6}}}G(t),Q_{\leq T^{\frac{1}{6}}}H(t)]\Big)\\
&\qquad-\frac{\pi}{t}\sum_{\substack{A,B,C\\\max(A,B,C)\geq T^{\frac{1}{6}}}}\mathcal{R}[Q_AF(t),Q_BG(t),Q_CH(t)]\ .
\end{align*}
Let us exam the terms on the right hand side one by one. The first term contributes to $\mathcal{E}_1$ by Lemma~\ref{highfreq}. The second term contains $\mathcal{E}_2$ as it can be written by lemma~\ref{fastosc} as $\widetilde{\mathcal{E}}_1+\mathcal{E}_2$ with $\widetilde{\mathcal{E}}_1$ contributing to $\mathcal{E}_1$. The last two terms contributes to $\mathcal{E}_1$ by Lemma~\ref{Res} and its remark. This finishes the proof of Proposition~\ref{Nonlinearity}.
\end{proof}

\section{The Resonant System}
In this section, we will study the following resonant system
\begin{equation}\label{resonantsyst}
i\partial_t G=\mathcal{R}[G,G,G]\ .
\end{equation}

Before further discussions, let us recall a useful result on the structure of the resonances at first. 
\begin{lemma}\label{resset}\cite[Lemma 1]{GGHW}
Given ${(p_1,p_2,p_3,p_4)\in \Gamma_0}$, namely, 
$${p_1-p_2+p_3-p_4=0\  {\rm and}\ \vert p_1\vert -\vert p_2\vert+\vert p_3\vert-\vert p_4\vert=0}$$
if and only if at least one of the following properties holds :
\begin{enumerate}
\item ${\forall j,\  p_j\geq 0\ ;}$
\item ${\forall j,\  p_j\leq 0\ ;}$
\item ${p_1=p_2\ ,\ p_3=p_4\ ;}$
\item ${p_1=p_4\ ,\ p_3=p_2\ .}$
\end{enumerate}
\end{lemma}

The following proposition shows us that we are able to get rid of the resonances corresponding to cases  $(3)$ and $(4)$, and deduce our resonant system to a decoupling system, which only contains cubic Szeg\H{o} equations.
\begin{proposition}\label{decoupling}
Given $G_0\in L^2_xH^s_y$, $s>1$,  $\Vert G_0\Vert_{L^2_xH^s_y}=\varepsilon$, $\varepsilon>0$ and $G_0(x,y)=-G_0(x,y+\pi)$. Set $G^1(t)= {\mathrm{e}}^{2it\Vert G_0\Vert_{L^2}^2}G(t)$ with $G$ as the corresponding solution to the resonant system \eqref{resonantsyst}, then $G^1(t)$ satisfies the following cubic Szeg\H{o} equation,
\begin{equation}\label{Rss+}
i\partial_t G^1_\pm=\mathcal{R}_\pm[G^1_\pm,G^1_\pm,G^1_\pm]\ ,
\end{equation}
where
\begin{equation}
\mathcal{F}_x\mathcal{R}_{\pm}[G^1_\pm,G^1_\pm,G^1_\pm](\xi,y)=\Pi_\pm(\vert \widehat{G^1}_\pm\vert^2\widehat{G^1}_\pm)(\xi,y)\ ,
\end{equation}
with $G^1_+=\Pi_+(G^1):=\sum\limits_{p>0}G^1_p(x){\mathrm e}^{ipy}$ and $G^1_-=\Pi_-(G^1):=\sum\limits_{p<0}G^1_p(x){\mathrm e}^{ipy}$.
\end{proposition}

\begin{proof}
The proof of the proposition above is easy. First, by the transformation, 
$$G^1(t)= {\mathrm{e}}^{2it\Vert G_0\Vert_{L^2}^2}G(t)\ ,$$
and using the fact that the $L^2$ norm is conserved, we get our first reduction to the resonant system corresponding to cases  $(1)$ and $(2)$. And thanks to our initial condition $G_0(x,y+\pi)=-G_0(x,y)$, we have
 $$\mathcal{F}_yG_0(x,p)=0,\quad p \text{ even numbers}\ ,$$ 
which insures the decoupling.
\end{proof}

\subsection{The cubic Szeg\H{o} equation}\mbox{}\\
\indent Let us begin with a simpler model, a resonant system for a vector $a=\{a_p\}_{p>0}$,
\begin{equation}\label{Rssa}
i\partial_t a_p(t)=\sum_{(p,q,r,s)\in\Gamma_{0,+}}a_q(t)\overline{a_r(t)}a_s(t):=R_+[a(t),a(t),a(t)]_p\ ,
\end{equation}
where $\Gamma_{0,+}:=\{(p_1,p_2,p_3,p_4):\ p_1-p_2+p_3-p_4=0, p_j>0\ \forall j\}$. If we denote $v(t,y):=\sum\limits_{p>0}a_p(t)e^{ipy}$, then $v$ satisfies the following cubic Szeg\H{o} equation
\begin{equation}\label{Szego}
i\partial_t v=\Pi_+(\vert v\vert^2v)\ .
\end{equation}

Let us recall more for the cubic Szeg\H{o} equation \eqref{Szego}, especially the Lax pair structure and its conserved quantities. G\'erard and Grellier have showed that the cubic Szeg\H{o} equation is a completely integrable system with two Lax pairs. One may refer to \cite{GGASENS,GGINV} for more details. To define the Lax pairs, one may need to introduce the Hankel operator $H_v$ and the Toeplitz operator $T_b$ with $v\in H^{\frac12}_+(\T)$, $b\in L^\infty(\T)$,
\begin{equation}\label{HanToep}
H_v h:=\Pi_+(v\bar{h})\ ,T_b h:=\Pi_+(bh)\ , h\in L^\infty\ .
\end{equation}
We remark that $H_v$ is $\mathbb{C}-$antilinear, and is a Hilbert-Schmidt operator. 
Now we are able to introduce the Lax pair structure of the cubic Szeg\H{o} equation \eqref{Szego},
\begin{theorem}\cite[Theorem 3.1]{GGASENS}\label{szlax}
Let $v\in C(\R,H^s_+(\T))$ for some $s>\frac12$. The cubic Szeg\H{o} equation \eqref{Szego}
has a Lax pair $(H_v, B_v)$, namely, if $v$ solves \eqref{Szego}, then
\begin{equation}
\frac{d H_v}{dt}=[B_v,H_v]\ , 
\end{equation}
where $B_v=\frac{i}{2}H_v^2-i T_{\vert v\vert ^2}$, 
\end{theorem}
A direct consequence of this Lax pair structure is that the spectrum of the trace class operator $H_v^2$, is conserved by the evolution, in particular, the trace norm of $H_v^2$ is conserved by the flow. A theorem by Peller \cite[Theorem2, P454]{Peller}, says that the trace norm of a Hankel operator $H_v$ is equivalent to the Besov norm $B^1_{1,1}(\T)$ of $v$. One may also see \cite{GGHW}.

\subsection{Estimation of solutions to the resonant system}\mbox{}\\
\indent We are now able to state a result concerning the long time behavior and stability of the asymptotic system \eqref{resonantsyst}.
\begin{lemma}\label{estofR}
For every function $G^1,\ G^2,\ G^3$, the following estimates hold true
\begin{align}
\Vert\mathcal{R}[G^1,G^2,G^3]\Vert_{L^2_{x,y}}&\lesssim\min_{\{j,k,\ell\}=\{1,2,3\}}\Vert G^j\Vert_{L^2_{x,y}}\Vert G^k\Vert_Z\Vert G^\ell\Vert_Z\ ,\label{R1}\\
\Vert\mathcal{R}[G^1,G^2,G^3]\Vert_Z&\lesssim\Vert G^1\Vert_Z\Vert G^2\Vert_Z\Vert G^3\Vert_Z\ ,\label{R2}\\
\Vert\mathcal{R}[G^1,G^2,G^3]\Vert_{S}&\lesssim\max_{\{j,k,\ell\}=\{1,2,3\}}\Vert G^j\Vert_S\Vert G^k\Vert_Z\Vert G^\ell\Vert_Z\ .\label{R3}
\end{align}
\end{lemma}

\begin{proof}
The first inequality comes from \eqref{sump}. Indeed, by the definition of $\mathcal{R}$,
\begin{align*}
\Vert\mathcal{R}[G^1,G^2,G^3]\Vert_{L^2_{x,y}}&=\Big\Vert \sum_{\Gamma_{0,+}\cup\Gamma_{0,-}} \widehat{G}^1_q\overline{\widehat{G}^2_r}\widehat{G}^3_s\Big\Vert_{L^2_\xi\ell_p^2}\\
&\lesssim\min_{\{j,k,\ell\}=\{1,2,3\}}\Vert G^j\Vert_{L^2_{x,y}}\Vert \widehat{G}^k\Vert_{L^\infty_\xi B^1_y}\Vert \widehat{G}^\ell\Vert_{L^\infty_\xi B^1_y}\\
&\lesssim \min_{\{j,k,\ell\}=\{1,2,3\}}\Vert G^j\Vert_{L^2_{x,y}}\Vert G^k\Vert_{Z}\Vert G^\ell\Vert_Z\ ,
\end{align*}
Apply Lemma~\ref{SS+}, we get the third inequality. The second inequality comes from the fact that $B^1$ is an algebra.
\end{proof}

\begin{proposition}\label{Gnorm}
Assume $G_0\in S^{(+)}$, $\Vert G_0\Vert_{S^{(+)}}=\varepsilon$ with $\varepsilon$ small enough, and $G$ evolves according to \eqref{resonantsyst}.
Then there holds that for $t\ge1$,
\begin{align}
&\Vert G(\pi\ln t)\Vert_Z\simeq\Vert G_0\Vert_Z\ ,\label{BddG1}\\
&\Vert G(\pi\ln t)\Vert_{S}\lesssim(1+\vert t\vert)^{\delta'}\Vert G_0\Vert_{S}\ ,\label{BddG2}\\
&\Vert G(\pi\ln t)\Vert_{S^{+}}\lesssim(1+\vert t\vert)^{\delta''}\Vert G_0\Vert_{S^{+}}\ ,\label{BddG3}
\end{align}
with $\delta'\simeq\Vert G_0\Vert_Z^2$, $\delta''\simeq\Vert G_0\Vert_S^3\Vert G_0\Vert_Z^{-1}$.
\end{proposition}

\begin{proof}
For the first conservation, we use the complete integrability of the cubic Szeg\H{o} equation, especially its Lax pair and the conservation of the $B^1$ norm, which is stated in the previous subsection. First, one may use Proposition~\ref{decoupling} to reduce our problem to the cubic Szeg\H{o} equation, and the transformation we used keeps the $Z$ norm. Then we use Peller's theorem to obtain
\begin{align*}
\Vert \widehat{G}(\xi, t)\Vert_{B^1}\simeq {\mathrm Tr}\vert H_{\widehat{G}(\xi, t)}\vert\ .
\end{align*}
Combined with the Lax Pair structure, we have
\begin{align*}
\Vert \widehat{G}(\xi, t)\Vert_{B^1}\simeq {\mathrm Tr}\vert H_{\widehat{G}(\xi, t)}\vert\simeq  {\mathrm Tr}\vert H_{\widehat{G}_0(\xi)}\vert \simeq \Vert \widehat{G}_0(\xi)\Vert_{B^1}\ .
\end{align*}

For the second one, taking $\widetilde{G}( t)=G(\pi\ln t)$, then $\widetilde{G}$ satisfies
\begin{equation}\label{resonantsyst1}
i\partial_t \widetilde{G}=\frac{\pi}{t}\mathcal{R}[\widetilde{G},\widetilde{G},\widetilde{G}]\ .
\end{equation}
The main idea is to estimate the $S^{(+)}$ norm of $\mathcal{R}[\widetilde{G},\widetilde{G},\widetilde{G}]$, and then apply the Gronwall's inequality.

Indeed, using \eqref{R3}, 
\begin{align*}
\partial_t \Vert \widetilde{G}\Vert_S\lesssim \frac1t\Vert \mathcal{R}[\widetilde{G},\widetilde{G},\widetilde{G}]\Vert_S\lesssim \frac1t\Vert \widetilde{G}\Vert_Z^2\Vert \widetilde{G}\Vert_S\lesssim \frac1t\Vert G_0\Vert_Z^2\Vert \widetilde{G}\Vert_S\ ,
\end{align*}
thus we get the $S$ norm estimate by Gronwall's inequality.

We now turn to the $S^+$ norm estimate, by the proof of the estimate \eqref{f4}, we may gain another more general version,
$$\Vert (1-\partial_{xx})^4G\Vert_Z+\Vert xG\Vert_Z\lesssim t^{-\delta'}\Vert G_0\Vert_{S}^{-1}\Vert G_0\Vert_Z\Vert G\Vert_{S^+}+t^{2\delta''}\Vert G_0\Vert_{S_+}\Vert G_0\Vert_S\Vert G_0\Vert_Z^{-2}\Vert G\Vert_S\ ,$$
with $\delta'\lesssim\delta''\simeq \Vert G_0\Vert_S^3\Vert G_0\Vert_Z^{-1}$. We apply the second part of Lemma~\ref{SS+},
\begin{align*}
\Vert \mathcal{R}[\widetilde{G},\widetilde{G},\widetilde{G}]\Vert_{S^+}&\lesssim\Vert \widetilde{G}\Vert_{S^+}\Big(\Vert G_0\Vert_Z^2+t^{-\delta'}\Vert G_0\Vert_{S}^{-1}\Vert G_0\Vert_Z\Vert\widetilde{G}\Vert_S\Big)\\&+t^{2\delta''}\Vert G_0\Vert_{S_+}\Vert G_0\Vert_S\Vert G_0\Vert_Z^{-1}\Vert\widetilde{G}\Vert_S^2\ ,
\end{align*}
then plugging the estimate of $\Vert\widetilde{G}\Vert_S$,
\begin{align*}
\frac{d}{dt}\Vert \widetilde{G}\Vert_{S^+}\lesssim t^{-1}\Vert \mathcal{R}[\widetilde{G},\widetilde{G},\widetilde{G}]\Vert_{S^+}\lesssim t^{-1}\Vert \widetilde{G}\Vert_{S^+}\Vert G_0\Vert_Z^2+t^{-1+4\delta''}\Vert G_0\Vert_{S_+}\Vert G_0\Vert_S^{3}\Vert G_0\Vert_Z^{-1}\ ,
\end{align*}
thus using the inhomogeneous Gronwall's inequality, we gain the estimate of the $S^+$ norm in \eqref{BddG3}. 
\end{proof}

\begin{proposition}\label{continuity}
If $A=\Pi_+A$ and $B=\Pi_+B$ solve \eqref{Rss+} with $\mathcal{R}_+$ and satisfy
\begin{equation*}
\sup_{0\le t\le T}\left\{\Vert A(t)\Vert_{Z}+\Vert B(t)\Vert_{Z}\right\}\le \theta
\end{equation*}
and
\begin{equation*}
\Vert A(0)-B(0)\Vert_{S^{(+)}}\le \delta
\end{equation*}
then, there holds that, for $0\le t\le T$,
\begin{equation}\label{StabRSS}
\Vert A(t)-B(t)\Vert_{S^{(+)}}\le \delta {\mathrm{e}}^{C\theta^2 t}. 
\end{equation}
\end{proposition}
\begin{proof}
By \eqref{Rss+}, $A-B$ satisfies
\begin{align*}
i\partial_t\Big(\widehat{A}_p(\xi)-\widehat{B}_p(\xi)\Big)&=\mathcal{R_+}[\widehat{A}(\xi)-\widehat{B}(\xi),\widehat{A}(\xi),\widehat{A}(\xi)]_p\\
&+\mathcal{R_+}[\widehat{B}(\xi),\widehat{A}(\xi)-\widehat{B}(\xi),\widehat{A}(\xi)]_p+\mathcal{R_+}[\widehat{B}(\xi),\widehat{B}(\xi),\widehat{A}(\xi)-\widehat{B}(\xi)]_p\ ,\end{align*}
then an application of Lemma~\ref{estofR} completes the proof.
\end{proof}

\section{The main results}
In this section, we will prove our main theorems. We will start with constructing a modified wave operator and gain the small data scattering as the theorem below.
\subsection{Modified scattering }
Given a small initial data in $S^+$, we may find a solution to our original system \eqref{HWS} by constructing a corresponding solution to the resonant system \eqref{RSS}, which also leads to the global well-posedness of our wave guide Schr\"odinger equation with small data. In the other hand, the solution with small initial data admits some modified scattering property.
\begin{theorem}\label{mscattering}
There exists $\varepsilon>0$ such that if $U_0\in S^+$ satisfies
\begin{equation}\label{MW}
\Vert U_0\Vert_{S^+}\le\varepsilon\ ,
\end{equation}
{\em (1)} If $\widetilde{G}$ is the solution of \eqref{RSS} with initial data $U_0$, then there exists a unique solution $U$ of \eqref{HWS} such that
${\mathrm{e}}^{-it\mathcal{A}}U(t)\in C([0,\infty):S)$ and
\begin{equation*}
\begin{split}
\Vert {\mathrm{e}}^{-it\mathcal{A}}U(t)-\widetilde{G}(\pi\ln t)\Vert_S\to0\,\hbox{ as }\,t\to+\infty\ .
\end{split}
\end{equation*}
{\em (2)} Conversely, consider the corresponding solution $U$ of \eqref{HWS} with initial data $U_0$ satisfying \eqref{MW}, if $\varepsilon$ is small enough, then there exists a solution $\widetilde{G}$ of \eqref{RSS}, such that
\begin{equation}\label{ModScat}
\Vert {\mathrm{e}}^{-it\mathcal{A}}U(t)-\widetilde{G}(\pi\ln t)\Vert_{S}\to 0\qquad\text{ as }t\to+\infty\ .
\end{equation}
\end{theorem}
\begin{proof}
Let us begin with (1). Set
$$G(t)=\widetilde{G}(\pi\ln t),\ K(t)={\mathrm{e}}^{-it\mathcal{A}}U(t)-G(t)$$
and define a mapping
\begin{equation*}
\Phi(K)(t)=-i\int_t^\infty\left\{\mathcal{N}^\sigma[K+G,K+G,K+G]-\frac{\pi}{\sigma}\mathcal{R}[G(\sigma),G(\sigma),G(\sigma)]\right\}d\sigma\ .
\end{equation*}
The main idea is to find a fixed point for $\Phi$ in a suitable space. Define
\begin{equation*}
\begin{split}
\mathfrak{A}:=&\{K\in C^1([1,\infty):S)\,:\,\,\Vert K\Vert_\mathfrak{A}<\infty\}\\
\Vert K\Vert_\mathfrak{A}:=&\sup_{t> 1}\left\{(1+\vert t\vert)^\delta\Vert K(t)\Vert_{S}+(1+\vert t\vert)^{2\delta}\Vert K(t)\Vert_{Z}+(1+\vert t\vert)^{1-\delta}\Vert \partial_tK(t)\Vert_S\right\}\ .
\end{split}
\end{equation*}

We claim that if $\varepsilon$ is sufficiently small, there exists $\varepsilon_1$ such that $\Phi$ defines a contraction on the complete metric space $\{K\in\mathfrak{A}:\  \Vert K\Vert_\mathfrak{A}\leq\varepsilon_1\}$. As in \cite[Theorem~5.1]{HPTV}, we decompose
\begin{equation}\label{dec}
\begin{split}
\mathcal{N}^t[K+G,K+G,K+G]-\frac{\pi}{t}\mathcal{R}[G,G,G]&=\mathcal{E}^t[G,G,G]+\mathcal{L}^t[K,G]+\mathcal{Q}^t[K,G]\\
\end{split}
\end{equation}
where
\begin{equation*}\left\{
\begin{split}
\mathcal{E}^t[G,G,G]&:=\mathcal{N}^t[G,G,G]-\frac{\pi}{t}\mathcal{R}[G,G,G]\ ,\\
\mathcal{L}^t[K,G]&:=\mathcal{N}^t[G,G,K]+\mathcal{N}^t[K,G,G]+\mathcal{N}^t[G,K,G]\ ,\\
\mathcal{Q}^t[K,G]&:=\mathcal{N}^t[K,K,G]+\mathcal{N}^t[G,K,K]+\mathcal{N}^t[K,G,K]+\mathcal{N}^t[K,K,K]\ .
\end{split}\right.
\end{equation*}
For $K\in\mathfrak{A}$, we have
\begin{equation}\label{estF}
(1+\vert t\vert)^{2\delta}\Vert K(t)\Vert_Z+(1+\vert t\vert)^\delta \Vert K(t)\Vert_S+(1+\vert t\vert)^{1-\delta}\Vert\partial_t K(t)\Vert_S\lesssim\varepsilon_1\ ,
\end{equation}
taking $\varepsilon\lesssim\delta^{1/2}$, by Proposition~\ref{Gnorm}, we have
\begin{equation}\label{estG}
\begin{split}
\Vert  G(t)\Vert_{S^+}+(1+\vert t\vert)\Vert\partial_tG(t)\Vert_{S^+}&\lesssim \varepsilon(1+\vert t\vert)^{\delta/100}\ ,\\
\Vert G(t)\Vert_Z&\lesssim\varepsilon\ .
\end{split}
\end{equation}

To show our claim, it suffices to show that the quantities below are small with $K,K_1,K_2\in\mathfrak{A}$,
\begin{align}
&\Vert \int_t^\infty\mathcal{E}^\sigma[G,G,G]d\sigma\Vert_\mathfrak{A}\lesssim \varepsilon^3\ ,\label{esmall}\\
&\Vert \int_t^\infty\mathcal{L}^\sigma[K,G]d\sigma\Vert_\mathfrak{A}\lesssim \varepsilon^2\Vert K\Vert_\mathfrak{A}\ ,\label{lsmall}\\
&\Vert \int_t^\infty\mathcal{Q}^\sigma[K,G]d\sigma\Vert_\mathfrak{A}\lesssim \varepsilon\Vert K\Vert_\mathfrak{A}^2\ ,\label{qsmall}\\
&\Vert \int_t^\infty\left\{\mathcal{Q}^\sigma[K_1,G]-\mathcal{Q}^\sigma[K_2,G]\right\}d\sigma\Vert_\mathfrak{A}\lesssim \varepsilon\varepsilon_1\Vert K_1-K_2\Vert_\mathfrak{A}\ .\label{qqsmall}
\end{align}

\noindent {\bf Proof of \eqref{esmall}.} Because of the definition of $\mathcal{E}^t$, we can easily gain for $t>1$,
\begin{equation}\begin{split}
\Vert \mathcal{E}^t[G,G,G]\Vert_S&=\Vert\mathcal{N}^t[G,G,G]-\frac1t\mathcal{R}[G,G,G]\Vert_S\\
&\leq \Vert \mathcal{N}^t[G,G,G]\Vert_S+\frac{1}{t}\Vert \mathcal{R}[G,G,G]\Vert_S\ .
\end{split}\end{equation}
Using \eqref{nestimate},
\begin{equation*}
\Vert \mathcal{N}^t[G,G,G]\Vert_S\leq t^{-1}\Vert G\Vert_S^3\leq t^{-1+\delta}\varepsilon^3\ ,
\end{equation*}
while by \eqref{R3},
\begin{equation*}
\Vert \mathcal{R}^t[G,G,G]\Vert_S\leq \Vert G\Vert_S^2\Vert G\Vert_Z\leq t^{\delta}\varepsilon^3\ ,
\end{equation*}
then
\begin{equation*}
\Vert \mathcal{E}^t[G,G,G]\Vert_S\leq t^{-1+\delta}\varepsilon^3\ ,
\end{equation*}
this controls the time derivative in the $\mathfrak{A}$ norm,
\begin{equation*}
t^{1-\delta}\left\Vert \partial_t\Big(\int_t^\infty\mathcal{E}^\sigma[G,G,G]d\sigma\Big)\right\Vert_S\leq \varepsilon^3\ .
\end{equation*}
 By \eqref{estG}, we have $\Vert G\Vert_{X_T^{(+)}}\leq \varepsilon$ for any $T>1$, so the other two terms of the $\mathfrak{A}$ norm, $\Vert\int_t^\infty\mathcal{E}^\sigma[G,G,G]d\sigma\Vert_S$ and $\Vert\int_t^\infty\mathcal{E}^\sigma[G,G,G]d\sigma\Vert_Z$ can be deduced by the estimates in Proposition~\ref{Nonlinearity}. 

\medskip

\noindent {\bf Proof of \eqref{lsmall}.} We estimate the norm $\Vert \int_t^\infty \mathcal{N}^\sigma[G,G,K]d\sigma\Vert_{\mathfrak{A}}$ for example. 

As in the proof of \eqref{esmall}, using \eqref{nestimate} and \eqref{estG}, we have the following estimate which controls the time derivative  in the $\mathfrak{A}$ norm,
\begin{equation*}
\Vert\mathcal{N}^t[G,G,K]\Vert_S\lesssim t^{-1}\Vert G\Vert_S^2\Vert K\Vert_S\lesssim t^{-1+\delta}\varepsilon^2\Vert K\Vert_{\mathfrak{A}}\ .
\end{equation*}
For the other two term in the $\mathfrak{A}$ norm, we shall reproduce the decomposition as in the proof of Proposition~\ref{Nonlinearity} on $\mathcal{N}^t[G,G,K]$. Using Lemma \ref{highfreq} and Lemma \ref{fastosc}, it only remains to show that
\begin{align}
&\Vert \mathcal{R}[G,G,K]\Vert_{Z}\lesssim (1+\vert t\vert)^{-2\delta}\varepsilon^2\varepsilon_1\ ,\\
&\Vert \mathcal{N}_0^t[G,G,K]-\frac{\pi}{t}\mathcal{R}[G,G,K]\Vert_{Z}\lesssim (1+\vert t\vert)^{-1-2\delta}\varepsilon^2\varepsilon_1\ ,\\
&\Vert \mathcal{N}_0^t[G,G,K]\Vert_{S}\lesssim (1+\vert t\vert)^{-1-\delta}\varepsilon^2\varepsilon_1\ .
\end{align}
The first estimate follows from \eqref{R2}, 
\begin{align*}
\Vert \mathcal{R}[G,G,K]\Vert_{Z}\lesssim \Vert G\Vert_Z^2\Vert K\Vert_Z\lesssim (1+\vert t\vert)^{-2\delta}\varepsilon^2\varepsilon_1\ .
\end{align*}
The second estimate follows from \eqref{n03},
\begin{align*}
\Vert \mathcal{N}_0^t[G,G,K]-\frac{\pi}{t}\mathcal{R}[G,G,K]\Vert_{Z}\lesssim (1+\vert t\vert)^{-1-20\delta}\Vert G\Vert_{S}^2\Vert K\Vert_{S}\lesssim(1+\vert t\vert)^{-1-2\delta}\varepsilon^2\varepsilon_1\ .
\end{align*}
 For the third estimate, we use \eqref{n01} to get
\begin{equation*}
\begin{split}
(1+\vert t\vert)\left\{\Vert \mathcal{N}_0^t[G,G,K]\Vert_{S}+\Vert \mathcal{N}_0^t[G,K,G]\Vert_{S}\right\}
&\lesssim \Vert G\Vert_{\widetilde{Z}_t}^2\Vert K\Vert_S+\Vert G\Vert_{\widetilde{Z}_t}\Vert K\Vert_{\widetilde{Z}_t}\Vert G\Vert_S\\
&\lesssim \varepsilon^2\varepsilon_1(1+\vert t\vert)^{-\delta}.
\end{split}
\end{equation*}

\medskip

\noindent {\bf Proof of \eqref{qsmall}.} The proof of \eqref{qsmall} is similar to the proof of \eqref{lsmall}.

\medskip

\noindent {\bf Proof of \eqref{qqsmall}.} We may rewrite 
\begin{align*}
\mathcal{N}^t[K_1,K_1,G]-\mathcal{N}^t[K_2,K_2,G]&=\mathcal{N}^t[K_1,K_1,G]-\mathcal{N}^t[K_1,K_2,G]+\mathcal{N}^t[K_1,K_1,G]\\
&\qquad-\mathcal{N}^t[K_2,K_2,G]\\
&=\mathcal{N}^t[K_1,K_1-K_2,G]+\mathcal{N}^t[K_1-K_2,K_2,G]\ ,
\end{align*}
we take similar decompositions on the terms $\mathcal{N}^t[K_1,G,K_1]-\mathcal{N}^t[K_2,G,K_2]$, $\mathcal{N}^t[G,K_1,K_1]-\mathcal{N}^t[G,K_2,K_2]$ and  $\mathcal{N}^t[K_1,K_1,K_1]-\mathcal{N}^t[K_2,K_2,K_2]$. Similar strategy we used to prove \eqref{lsmall} can be applied to obtain the estimate on the norm $\Vert\mathcal{N}^t[F_1,F_2,F_3]\Vert_{\mathfrak{A}}$ with one of these $F_j$ be $K_1-K_2$ while the other two functions belong to $\{K_1,K_2,G\}$. The proof of the first part is complete.
\bigskip

Let us turn to (2), we will prove it in two steps.\\
\noindent{\bf Step 1: Global existence and bounds.}
Let $U_0\in S^+$, $\Vert U_0\Vert_{S^+}\leq\varepsilon$ with $\varepsilon$ small enough. The local existence is classical via its integral equation. We denote $F(t):=e^{-it\mathcal{A}}U(t)$, then \eqref{HWS} can be rewritten as
\begin{equation}
 i\partial_tF=\mathcal{N}^t[F,F,F]\ .
\end{equation}
\begin{equation}
F(t)=U_0-i\int_0^t\mathcal{N}^\sigma[F,F,F]d\sigma\ .
\end{equation}
By the estimate \eqref{nestimate}, we have
\begin{align*}
\Vert \mathcal{N}^t[F,F,F]\Vert_{S^+}\lesssim (1+|t|)^{-1}\Vert F\Vert_{S^+}^3\ .
\end{align*}
This allows us to use a fixed point argument on a small time interval $[0,T]$, and $t\mapsto \Vert F(t)\Vert_{S^+}$ is $C^1$. We claim that \begin{equation}\label{gexist}
\Vert F\Vert_{X^+_T}\le \Vert U_0\Vert_{S^+}+C\Vert F\Vert_{X^+_T}^3
\end{equation}
for all $T>0$ and all $U$ solving \eqref{HWS} such that $\Vert F\Vert_{X^+_T}\le\sqrt{\varepsilon}$. 
Then by a bootstrap argument, we gain the global existence and the solution satisfies for all $T>0$
 \begin{equation}\label{Beps} 
\Vert FU(t)\Vert_{X^+_T}\leq 2\varepsilon\ .
\end{equation}

Let us begin the proof of our claim \eqref{gexist}. Recall the definition of the $X_T^+$ norm \eqref{DefX}, we have to consider the $S$ and $S^+$ norm of $F$ and $\partial_t F$ and also the $Z$ norm of $F$.

It is easy to deduce from the equation on $F$ that $\Vert \partial_tF\Vert_{S^{(+)}}=\Vert \mathcal{N}^t[F,F,F]\Vert_{S^{(+)}}$. Thanks to \eqref{nestimate}, we have
\begin{align*}
&\Vert \partial_tF\Vert_S= \Vert \mathcal{N}^t[F,F,F]\Vert_{S}\leq  (1+|t|)^{-1}\Vert F\Vert_S^3\ ,\\
&\Vert \partial_tF\Vert_{S^+}=\Vert \mathcal{N}^t[F,F,F]\Vert_{S^+}\leq (1+|t|)^{-1}\Vert F\Vert_S^2\Vert F\Vert_{S^+}\ ,
\end{align*}
thus
\begin{align}
&(1+|t|)^{1-3\delta}\Vert \partial_tF\Vert_S\leq \big((1+|t|)^{-\delta}\Vert F\Vert_S\big)^3+\leq \Vert F\Vert_{X_T^+}^3\ ,\label{p1}\\
&(1+|t|)^{1-7\delta}\Vert \partial_tF\Vert_{S^+}\leq \big((1+|t|)^{-\delta}\Vert F\Vert_S\big)^2(1+|t|)^{-5\delta}\Vert F\Vert_{S^+}\leq \Vert F\Vert_{X_T^+}^3\label{p2}\ .
\end{align}

We now turn to estimate $\Vert F\Vert_Z$, by the decomposition result of $\mathcal{N}^t$ in Proposition~\ref{Nonlinearity}, and notice that $R$ defined as \eqref{Rssa} is self-adjoint on $\ell_p^2$ and that there is a cancellation
$$\langle i\mathcal{F}\mathcal{R}[F,F,F](\xi),\ \mathcal{F}F(\xi)\rangle_{h_p^\sigma,h_p^\sigma}=0\ .$$
So we will study the $\Vert F\Vert_{Y^\sigma}$ with $\sigma>1$ where $Y^\sigma$ is defined in \eqref{DefYs}, then to control the $Z$ norm.
\begin{equation}\label{HCancellation}\begin{split}
\frac{d}{ds}\frac{1}{2}\Vert \widehat{F}_p(\xi,s)\Vert_{h^\sigma_p}^2&=\Big\langle \mathcal{F}\mathcal{N}^t[F,F,F](\xi,s),\ \widehat{F}_p(\xi,s)\Big\rangle_{h_p^\sigma,h_p^\sigma}\\
&=\langle\widehat{\mathcal{E}}_1(\xi,p,s),\widehat{F}_p(\xi,s)\rangle_{h^\sigma_p\times h^\sigma_p}+\langle\partial_s\widehat{\mathcal{E}}_3(\xi,p,s),\widehat{F}_p(\xi,s)\rangle_{h^\sigma_p\times h^\sigma_p}\ .
\end{split}\end{equation}
Thus multiplying with $(1+\vert \xi\vert^2)$, using the estimates of $\Vert\mathcal{E}_j\Vert_{Y^\sigma}$ in Proposition~\ref{Nonlinearity}, then we have for any $\xi$, we have
$$(1+\vert \xi\vert^2)\vert\int_0^t\langle\widehat{\mathcal{E}}_1(\xi,p,s),\widehat{F}_p(\xi,s)\rangle_{h^\sigma_p\times h^\sigma_p}\vert\lesssim \Vert F\Vert_{X_T^+}^3\int_0^t(1+|s|)^{-1-\delta}ds\cdot\sup_{[0,t]}\Vert F(s)\Vert_{Y^\sigma}\ ,$$
while
\begin{equation*}
\begin{split}
&[1+|\xi|^2]\left\vert \int_0^t\langle \partial_t\widehat{\mathcal{E}}_3(\xi,p,s),\widehat{F}_p(\xi,s)\rangle_{h^\sigma\times h^\sigma_p}ds\right\vert
\le [1+|\xi|^2] \left \vert \langle \widehat{\mathcal{E}}_3(\xi,p,t),\widehat{F}_p(\xi,t)\rangle_{h^\sigma\times h^\sigma_p}\right\vert\\
&\quad+[1+|\xi|^2]\left\vert\langle\widehat{ \mathcal{E}}_3(\xi,p,0), \widehat{F}_p(\xi,0)\rangle_{h^\sigma\times h^\sigma_p}\right\vert
+[1+|\xi|^2]\left\vert\int_0^t\langle\widehat{\mathcal{E}}_3(\xi,p,s),\partial_t\widehat{F}_p(\xi,s)\rangle_{h^\sigma\times h^\sigma_p}\right\vert\\
&\lesssim \Vert F\Vert_{X_T^+}^3\cdot\sup_{[0,t]}\Vert F(s)\Vert_{Y^\sigma}+\Vert F\Vert_{X_T^+}^6\ .
\end{split}
\end{equation*}
Combining the above estimates, we have
\begin{equation}\label{p3}
\Vert F(t)\Vert_Z\le \Vert F(t)\Vert_{Y^\sigma}+C\Vert F\Vert_{X_T^+}^3\ .
\end{equation}

For the norm $\Vert F(t)\Vert_{S^{(+)}}$, when $0\leq t\leq 1$, 
\begin{align*}
\Vert F\Vert_{S}\leq\Vert F\Vert_{S^+}&\leq \Vert U_0\Vert_{S^+}+\Vert F(t)-F(0)\Vert_{S^+}\\
&\leq\Vert U_0\Vert_{S^+}+\sup_{0\leq t\leq1}\Vert \partial_t F\Vert_{S^+}\\
&\leq\Vert F\Vert_{X_T^+}^3\ .
\end{align*}
While $1\leq t\leq T$, using Proposition~\ref{Nonlinearity}, we have
\begin{align*}
\Vert F(t)-F(1)\Vert_{S^{(+)}}&\leq \Big\Vert\int_1^t \mathcal{N}^\sigma[F,F,F]d\sigma\Big\Vert_{S^{(+)}}\\
&\leq \Big\Vert\int_1^t \mathcal{R}[F,F,F]d\sigma/\sigma\Big\Vert_{S^{(+)}}+\Big\Vert\int_1^t\big( \mathcal{E}_1(\sigma)+\mathcal{E}_2(\sigma)\big)d\sigma\Big\Vert_{S^{(+)}}\ ,
\end{align*}
then using \eqref{n05} and the, we have
\begin{align}
\Big\Vert\int_1^t \mathcal{R}[F,F,F]d\sigma/\sigma\Big\Vert_{S}&\lesssim \int_1^t \sigma^{-1}\Vert F\Vert_{\widetilde{Z}_t}^2\Vert F\Vert_S d\sigma\\
&\lesssim \int_1^t \sigma^{-1+\delta}d\sigma\Vert F\Vert_{X_T^+}^3\leq t^\delta\Vert F\Vert_{X_T^+}^3\ ,
\end{align}
while by \eqref{n06},
\begin{align}
\Big\Vert\int_1^t \mathcal{R}[F,F,F]d\sigma/\sigma\Big\Vert_{S^+}
&\leq \int_1^t \Big(\sigma^{-1}\Vert F\Vert_{\widetilde{Z}_t}^2\Vert F\Vert_{S^+} +\sigma^{-1+2\delta}\Vert F\Vert_{\widetilde{Z}_t}\Vert F\Vert_{S}^2 \Big)d\sigma\\
&\leq \int_1^t \sigma^{-1+5\delta}d\sigma \Vert F\Vert_{X_T^+}^3\leq t^{5\delta}\Vert F\Vert_{X_T^+}^3\ ,
\end{align}
together with the estimates in Proposition~\ref{Nonlinearity},
\begin{align}
\Vert F(t)-F(1)\Vert_{S}&\leq (1+|t|)^{\delta}\Vert F\Vert_{X_T^+}^3\ ,\\
\Vert F(t)-F(1)\Vert_{S^+}&\leq (1+|t|)^{5\delta}\Vert F\Vert_{X_T^+}^3\ .
\end{align}
Hence, we finally gain
\begin{equation}\label{p4}
(1+|t|)^{-\delta}\Vert F\Vert_S+(1+|t|)^{-5\delta}\Vert F\Vert_{S^+}\leq \Vert F\Vert_{X_T^+}\ .
\end{equation}
Our priori estimate \eqref{gexist} comes out from \eqref{p1}, \eqref{p2}, \eqref{p3} and \eqref{p4}.

\medskip

\noindent{\bf Step 2: Asymptotic behavior.}
Define $T_n={\mathrm{e}}^{n/\pi}$ and $G_n(t)=\widetilde{G}_n(\pi\ln t)$, where $\widetilde{G}_n$ solves \eqref{RSS} with Cauchy data such that $\widetilde{G}_n(n)=G_n(T_n)=F(T_n)$.
We claim that for all $t\ge T_n$,
\begin{equation}\label{BdsGN}
\begin{split}
\Vert G_n(t)\Vert_Z+(1+\vert t\vert)^{-\delta}\Vert G_n(t)\Vert_S+(1+\vert t\vert)^{-5\delta}\Vert G_n(t)\Vert_{S^+}+(1+\vert t\vert)^{1-\delta}\Vert \partial_tG_n(t)\Vert_S\lesssim\varepsilon
\end{split}
\end{equation}
uniformly in $n\ge 0$. Indeed, first we get from the global bounds result \eqref{Beps} that uniformly in $n$,
\begin{equation*}
\Vert G_n(t)\Vert_Z=\Vert\widetilde{G}_n(\pi \ln t)\Vert_Z=\Vert\widetilde{G}_n(n)\Vert_Z=\Vert F(T_n)\Vert_Z\lesssim\varepsilon\ ,
\end{equation*}
$$\Vert G_n(T_n)\Vert_{S}\lesssim \varepsilon T_n^\delta\ ,$$
and by \eqref{R3},
\begin{equation}\label{gt}
\Vert \partial_t G_n(s)\Vert_{S}\lesssim s^{-1}\Vert G_n\Vert_Z^2\Vert G_n(s)\Vert_S\lesssim \varepsilon^2s^{-1}\Vert G_n(s)\Vert_S\ .
\end{equation}
An application of Gronwall's lemma gives, for $\varepsilon$ small enough,
\begin{equation*}
\Vert G_n(s)\Vert_{S}\lesssim \varepsilon s^\delta,\quad s\ge T_n
\end{equation*}
which, combined with \eqref{gt}, provides control of the second and last term in \eqref{BdsGN}. We can estimate the $S^+$ norm similarly, using \eqref{n06},
\begin{equation*}
\Vert \partial_t G_n(s)\Vert_{S^+}\lesssim s^{-1}\varepsilon^2\Vert G_n(s)\Vert_{S^+}+\varepsilon^3s^{-1+4\delta},\qquad \Vert G_n(T_n)\Vert_{S^+}\lesssim\varepsilon T_n^{5\delta}.
\end{equation*}
This concludes the proof of \eqref{BdsGN}.

Since 
\begin{align*}
i\partial_t F&=\mathcal{N}^t[F,F,F]\ ,\\
i\partial_t G_n&=\frac{t}{\pi}\mathcal{R}[G_n,G_n,G_n]\ ,
\end{align*}
and
$$F(T_n)=G_n(T_n)\ ,$$
then
\begin{align*}
F(t)-G_n(t)&=i\int_{T_n}^t\Big(\mathcal{N}^\sigma[F,F,F]-\frac{\sigma}{\pi}\mathcal{R}[G_n,G_n,G_n]\Big)d\sigma\\
&=i\int_{T_n}^t\mathcal{E}^\sigma[F,F,F]d\sigma+i\int_{T_n}^t\frac{\sigma}{\pi}\Big(\mathcal{R}[F,F,F]-\mathcal{R}[G_n,G_n,G_n]\Big)d\sigma\ .
\end{align*}
Using the estimates in Proposition~\ref{Nonlinearity}, we gain for $t>T_n$,
\begin{align*}
\Vert F-G_n\Vert_Z&\lesssim \varepsilon^3T_n^{-2\delta}+\int_{T_n}^t\big(\Vert F\Vert_Z^2+\Vert G_n\Vert_Z^2\big)\Vert F-G_n\Vert_Z\frac{d\sigma}\sigma\\
&\lesssim  \varepsilon^3T_n^{-2\delta}+\varepsilon^2\int_{T_n}^t\Vert F-G_n\Vert_Z\frac{d\sigma}\sigma\ ,
\end{align*}
we may then deduce by Gronwall,
\begin{align*}
\Vert F-G_n\Vert_Z\lesssim \varepsilon^3T_n^{-2\delta} \text{ for }T_n\leq t\leq T_{n+4}\ .
\end{align*}
We may deduce the estimate on $\Vert F-G_n\Vert_S$ similarly and
\begin{align*}
\Vert F-G_n\Vert_S&\lesssim \varepsilon^3T_n^{-2\delta}+\int_{T_n}^t\big(\Vert F\Vert_Z^2+\Vert G_n\Vert_Z^2\big)\Vert F-G_n\Vert_S\frac{d\sigma}\sigma\\
&+\int_{T_n}^t\big(\Vert F\Vert_Z+\Vert G_n\Vert_Z\big)\Vert F-G_n\Vert_Z\big(\Vert F\Vert_S+\Vert G_n\Vert_S\big)\frac{d\sigma}\sigma\\
&\lesssim  \varepsilon^3T_n^{-\delta}+\varepsilon^2\int_{T_n}^t\Vert F-G_n\Vert_S\frac{d\sigma}\sigma\ ,
\end{align*}
Again, we use Gronwall's inequality and get
\begin{equation}
\sup_{T_n\leq t\leq T_{n+4}}\Vert F(t)-G_n(t)\Vert_S\lesssim \varepsilon^3T_n^{-\delta}\ .
\end{equation}
Therefore,
\begin{equation}
\Vert \widetilde{G}_{n+1}(n+1)-\widetilde{G}_n(n+1)\Vert_S=\Vert F(T_{n+1})-G_{n}(T_{n+1})\Vert_{S}\lesssim \varepsilon^3T_n^{-\delta}\ ,
\end{equation}
and thus by Lemma~\ref{continuity}, we gain
\begin{equation*}
\Vert \widetilde{G}_{n+1}(0)-\widetilde{G}_n(0)\Vert_S\lesssim \varepsilon^3{\mathrm{e}}^{-n\delta/2}\ .
\end{equation*}

We see that $\{\widetilde{G}_n(0)\}_n$ is a Cauchy sequence in $S$ and therefore converges to an element $G_{0,\infty}\in S$ which satisfies that
\begin{equation*}
\Vert G_{0,\infty}\Vert_Z\lesssim\varepsilon,\qquad \Vert \widetilde{G}_n(0)-G_{0,\infty}\Vert_S\lesssim \varepsilon^3{\mathrm{e}}^{-n\delta/2}.
\end{equation*}
By Proposition~\ref{continuity},
\begin{equation*}
\begin{split}
\sup_{[0,T_{n+2}]}\Vert G_\infty(t)-G_n(t)\Vert_S\lesssim \varepsilon^3 {\mathrm{e}}^{-n\delta/4}\\
\end{split}
\end{equation*}
where $G_\infty(t)=\widetilde{G}_\infty(\pi\ln t)$ with $\widetilde{G}_\infty$ the solution of \eqref{RSS} with initial data $\widetilde{G}_\infty(0)=G_{0,\infty}$.
Now we have
\begin{equation*}
\begin{split}
\sup_{T_n\le t\le T_{n+1}}\Vert G_\infty(t)-F(t)\Vert_{S}&\le \sup_{T_n\le t\le T_{n+1}}\Vert G_\infty(t)-G_n(t)\Vert_{S}+\sup_{T_n\le t\le T_{n+1}}\Vert G_n(t)-F(t)\Vert_{S}\\
&\lesssim \varepsilon^3e^{-n\delta/4}.
\end{split}
\end{equation*}
This finishes the proof.
\end{proof}

\subsection{Large time Sobolev unboundedness} We will firstly study the dynamics of the resonant system \eqref{RSS}, then we apply the modified scattering results above to gain the large time behavior of the wave guide Schr\"odinger equation. The following strategy allows us to transfer informations from a global solution $a(t)$ of \eqref{Rssa} to a solution of \eqref{Rss+}, all we need to do is to take an initial datum of the form
\begin{equation*}
G_0(x,y)=\check{\varphi}(x)g(y),\ \varphi\in\mathcal{S}(\mathbb{R})\ ,
\end{equation*}
where $g_p=a_p(0)$, and $ \check{\varphi}(x)$ is the inverse Fourier transform of $\varphi$. The solution $G(t)$ to \eqref{Rss+} with initial data $G_0$ as above is given in Fourier space by 
\begin{equation}\label{defofG}
\widehat{G}_p(t, \xi)=\varphi(\xi)a_p(\varphi(\xi)^2 t)\ .
\end{equation} 
In particular, if $\varphi=1$ on an open interval $I$, then $\widehat G_p(t, \xi)=a_p(t)$ for all $t\in \R$ and $\xi \in I$. For $\xi\in I$, the resonant system turns out to be the cubic Szeg\H{o} equation.

Let us recall the infinite cascade result for the cubic Szeg\H{o} equation.
\begin{theorem}\cite{GGbook}\label{SZEGOUNBDD} 
For any $v_0\in C^\infty_+:=\cap_s H^s$, for any $M$, for any $r>\frac12$, there exists a sequence $(v_0^n)$ of elements of $C_+^\infty$ tending to $v_0$ in $C_+^\infty$ and a sequence of time $t_n$, $|t_n|$ tending to $\infty$, such that the corresponding solution $v_n$ of the cubic Szeg\H{o} equation
\begin{equation}\label{szegoc3}
i\partial_tv=\Pi_+(|v|^2v)\ ,\ v(0)=v_0^n\ ,
\end{equation} 
satisfies
\begin{equation}
\frac{\Vert v^n(t^n)\Vert_{H^r}}{|t_n|^M}\to\infty\ ,\  n\to\infty\ .
\end{equation}
\end{theorem}

\begin{assumption}\label{assumption}
$G_0\in S^+$ with $\Vert G_0\Vert_{S^+}$ small, $G_0(y)=-G_0(y+\pi)$, and there exists some non empty open set $I\ne\emptyset$, such that $\ G_0(\xi)=v_0\ \forall \xi\in I$, while the corresponding solution of the cubic Szeg\H{o} equation \eqref{szegoc3} with $v_0$ as the initial data admits an unbounded trajectory as described in Theorem~\ref{SZEGOUNBDD}.
\end{assumption}

Let $G$ be a solution to
\begin{equation}
\left\{\begin{split}
i\partial_t G&=\mathcal{R}[G,G,G]\\
G(0)&=G_0
\end{split}\right.
\end{equation}
with $G_0$ satisfies Assumption~\ref{assumption}, then
\begin{equation}\label{blG}
\Vert G(t)\Vert_{L_x^2H_y^s}\geq\big(\int_I \Vert\widehat{G}(t,\xi)\Vert_{H_y^s}^2d\xi\big)^{1/2}=|I|^{1/2} \Vert v(t)\Vert_{H_y^s}\to\infty\ ,
\end{equation}
for any $s>1/2$.

By Theorem~\ref{mscattering}, for the solutions to the resonant system $G$ as above, there exists solutions to the wave guide Schr\"odinger equation \eqref{HWS}, such that \eqref{MW} holds. Then the large time behavior of $G(t)$ \eqref{blG} leads to the large time unbounded Sobolev trajectories of solutions to the equation \eqref{HWS}.
\begin{corollary}
Given $N\geq13$, then for any $\varepsilon>0$, there exists $U_0\in S^+$ with $\Vert U_0\Vert_{S^+}\leq\varepsilon$ such that the corresponding solution to \eqref{HWS} satisfies
\begin{equation}
\limsup_{t\to\infty}\frac{\Vert U(t)\Vert_{L_x^2H_y^s}}{(1+\log|t|)^M}=\infty\ ,\quad \forall s> 1/2\ ,\ \forall M\ . 
\end{equation} 
\end{corollary}

\begin{remark}
As we announced in the introduction of this paper, the unbounded Sobolev norms in our theorem are just above the energy norm. The growth is as large as $(\log |t|)^M$ for any $M$ for solutions with small initial data in $S^+$, which is almost optimal for the dispersive wave guide Schr\"odinger equation.

Moreover, in view of Theorem~\ref{SZEGOUNBDD} by G\'erard and Grellier \cite{GGbook}, we expect that there exist some Banach space $B$, such that the set
$$ G:=\left \{ U_0\in B\, :\,  \forall s>\frac 12\ ,\forall M\in \Z_+ \ ,\ \limsup _{\vert t\vert\rightarrow +\infty } \frac{\Vert U(t)\Vert _{H^s}}{(\log|t|)^M} =+\infty \right \}$$
is a dense $G_\delta$ subset of $B$. The difficulty comes from the gap between $S$ and $S^+$ in the modified scattering argument, which already exists in the early results of Ozawa \cite{Ozawa} and Hayashi--Naumkin--Shinomura--Tonegawa \cite{HNST}.
\end{remark}
\section{Appendix}
We now turn to our basic lemma allowing to transform suitable $L^2_{x,y}$ bounds to bounds in terms of the $L^2_{x,y}$-based spaces $S$ and $S^+$.
We define an \emph{LP-family} $\widetilde{Q}=\{\widetilde{Q}_A\}_A$ to be a family of operators (indexed by the dyadic integers) of the form
\begin{equation*}
\widehat{\widetilde{Q}_1f}(\xi)=\widetilde{\varphi}(\xi)\widehat{f}(\xi),\qquad \widehat{\widetilde{Q}_Af}(\xi)=\widetilde{\phi}(\frac{\xi}{A})\widehat{f}(\xi), \quad A\ge 2
\end{equation*}
for two smooth functions $\widetilde{\varphi},\widetilde{\phi}\in C^\infty_c(\mathbb{R})$ with $\widetilde{\phi}\equiv0$ in a neighborhood of $0$.

We define the set of \emph{admissible transformations} to be the family of operators $\{T_A\}$ where for any dyadic number $A$,
\begin{equation*}
T_A=\lambda_A\widetilde{Q}_A,\qquad \vert\lambda_A\vert\le 1
\end{equation*}
for some LP-family $\widetilde{Q}$. 

If $F\in\mathcal{B}$, then for any admissible transformation family $T=\{T_A: A \text{ dyadic numbers }\}$, $\sum\limits_A T_AF$ converges in $\mathcal{B}$. And this norm $\mathcal{B}$ is called admissible if
\begin{equation}\label{admissible}
\Vert \sum_A T_AF\Vert_{\mathcal{B}}\lesssim \Vert F\Vert_{\mathcal{B}}.
\end{equation}

\begin{lemma}\label{Admissible}
Recall the definitions of the norms $S$, $S^+$, $Z$ and $\widetilde{Z}_t$, 
\begin{align*}
\Vert F\Vert_{S}:=&\Vert F\Vert_{H^N_{x,y}}+\Vert xF\Vert_{L^2_{x,y}}\ ,\quad
\Vert F\Vert_{S^+}:=\Vert F\Vert_S+\Vert (1-\partial_{xx})^4F\Vert_{S}+\Vert xF\Vert_{S}\ ,\\
\Vert F\Vert_{Z}:=&\sup_{\xi\in\R}\left[1+\vert \xi\vert^2\right]\|\widehat{F}(\xi,\cdot)\|_{B^1}\ ,\quad 
\Vert F\Vert_{\widetilde{Z}_t}:=\Vert F\Vert_Z+(1+\vert t\vert)^{-\delta}\Vert F\Vert_S\ .                                                    
\end{align*}
All these norms are admissible.
\end{lemma}
\begin{proof}
Due to the definition of admissible transformation, we may only deal with functions independent on $y$. Let us prove with the $S$ norm for example. Indeed,
\begin{align*}
\Vert \sum_A T_Af\Vert_{H^N}^2&=\int_{\mathbb{R}}\langle\xi\rangle^{2N}\Big\vert\lambda_1\widetilde{\varphi}(\xi)\widehat{f}(\xi)+\lambda_A\widetilde{\phi}(\frac{\xi}{A})\widehat{f}(\xi)\Big\vert^2d\xi\\
&\leq \int_{\mathbb{R}}\Big(\vert\lambda_1\vert^2\widetilde{\varphi}(\xi)+\vert\lambda_A\vert^2\widetilde{\phi}(\frac{\xi}{A})\Big)\langle\xi\rangle^{2N}\vert\widehat{f}(\xi)\vert^2d\xi\\
&\leq\Vert f\Vert_{H^N}^2\ ,
\end{align*}
while
\begin{align*}
\Vert x\sum_A T_Af\Vert_{L^2}^2&=\int_{\mathbb{R}}\Big\vert\partial_\xi\big(\lambda_1\widetilde{\varphi}(\xi)\widehat{f}(\xi)+\lambda_A\widetilde{\phi}(\frac{\xi}{A})\widehat{f}(\xi)\big)\Big\vert^2d\xi\\
&\leq \int_{\mathbb{R}}\Big(\vert\lambda_1\vert^2\widetilde{\varphi}(\xi)+\vert\lambda_A\vert^2\widetilde{\phi}(\frac{\xi}{A})\Big)\vert\partial_\xi\widehat{f}(\xi)\vert^2d\xi\\
&\qquad+\int_{\mathbb{R}}\Big(\vert\lambda_1\vert^2\widetilde{\varphi'}(\xi)+\frac{\vert\lambda_A\vert^2}{A}\widetilde{\phi'}(\frac{\xi}{A})\Big)\vert\widehat{f}(\xi)\vert^2d\xi\\
&\leq\Vert xf\Vert_{L^2}^2+\Vert f\Vert_{L^2}^2\ ,
\end{align*}
thus
\begin{equation*}
\Vert \sum_A T_A f\Vert_S\lesssim \Vert f\Vert_S\ .
\end{equation*}
\end{proof}

Given a trilinear operator 
$\mathfrak{T}$ and a set $\Lambda$ of 4-tuples of dyadic integers, we define an \emph{admissible realization} of $\mathfrak{T}$ at $\Lambda$ to be an operator of the form which converges in $L^2$,
\begin{equation}
\mathfrak{T}_\Lambda[F,G,H]=\sum_{(A,B,C,D)\in\Lambda}T_D\mathfrak{T}[T^\prime_AF,T^{\prime\prime}_BG,T^{\prime\prime\prime}_CH]
\end{equation}
for some admissible transformations $T$, $T^\prime$, $T^{\prime\prime}$, $T^{\prime\prime\prime}$.
\begin{lemma}\label{SS+}
Assume that a trilinear operator $\mathfrak{T}$ satisfies
\begin{equation}\label{LeibnitzRule}
\begin{split}
Z\mathfrak{T}[F,G,H]= \mathfrak{T}[ZF,G,H]+\mathfrak{T}[F,ZG,H]+\mathfrak{T}[F,G,ZH],
\end{split}
\end{equation}
for $Z\in\{x,\partial_x,\partial_{y}\}$ and let $\Lambda$ be a set of $4$-tuples of dyadic integers. With the notation introduced above, assume also that for all admissible realizations of $\mathfrak{T}$ at $\Lambda$,
\begin{equation}\label{LB}
\Vert \mathfrak{T}_\Lambda[F^a,F^b,F^c]\Vert_{L^2}\le K\min_{\{\alpha,\beta,\gamma\}=\{a,b,c\}}\Vert F^\alpha\Vert_{L^2}\Vert F^\beta\Vert_{\mathcal{B}}\Vert F^\gamma\Vert_{\mathcal{B}}
\end{equation}
for some admissible norm $\mathcal{B}$ such that the Littlewood-Paley projectors $P_{\le M}$ (both in $x$ and in $y$) are uniformly bounded on $\mathcal{B}$. Then, for all admissible realizations of $\mathfrak{T}$ at $\Lambda$,
\begin{equation}\label{SSS}
\Vert \mathfrak{T}_\Lambda[F^a,F^b,F^c]\Vert_{S}\lesssim K\max_{\{\alpha,\beta,\gamma\}=\{a,b,c\}}\Vert F^\alpha\Vert_{S}\Vert F^\beta\Vert_{\mathcal{B}}\Vert F^\gamma\Vert_{\mathcal{B}}\ .
\end{equation}
Assume in addition that, for $Y\in\{x,(1-\partial_{xx})^4\}$,
\begin{equation}\label{YB}
\Vert YF\Vert_{\mathcal{B}}\lesssim \theta_1\Vert F\Vert_{S^+}+\theta_2\Vert F\Vert_S\ ,
\end{equation}
then for all admissible realizations of $\mathfrak{T}$ at $\Lambda$,
\begin{equation}\label{SSS+}
\begin{split}
\Vert \mathfrak{T}_\Lambda[F^a,F^b,F^c]\Vert_{S^+}&\lesssim K\max_{\{\alpha,\beta,\gamma\}=\{a,b,c\}}\Vert F^\alpha\Vert_{S^+}\big(\Vert F^\beta\Vert_{\mathcal{B}}+\theta_1\Vert F^\beta\Vert_{S}\big) \Vert F^\gamma\Vert_{\mathcal{B}}\\
&\quad+\theta_2K\max_{\{\alpha,\beta,\gamma\}=\{a,b,c\}}\Vert F^\alpha\Vert_{S}\Vert F^\beta\Vert_{S}\Vert F^\gamma\Vert_{\mathcal{B}}\ .
\end{split}
\end{equation}
\end{lemma}

\begin{proof}
Let us start with \eqref{SSS}. \\
{\bf 1. The weighted component of $S$ norm.}
By rewriting $xT_A=[x,T_A]+T_Ax$ and using \eqref{LeibnitzRule}, we have
\begin{equation}\begin{split}
&x\mathfrak{T}_\Lambda[F^a,F^b,F^c]=\sum_{(A,B,C,D)\in\Lambda}xT_D\mathfrak{T}[T^\prime_AF^a,T^{\prime\prime}_BF^b,T^{\prime\prime\prime}_CF^c]\\
&\qquad=\sum_{(A,B,C,D)\in\Lambda}[x,T_D]\mathfrak{T}[T^\prime_AF^a,T^{\prime\prime}_BF^b,T^{\prime\prime\prime}_CF^c]+\sum_{(A,B,C,D)\in\Lambda}T_D\mathfrak{T}[[x,T^\prime_A]F^a,T^{\prime\prime}_BF^b,T^{\prime\prime\prime}_CF^c]\\
&\qquad+\sum_{(A,B,C,D)\in\Lambda}T_D\mathfrak{T}[T^\prime_AF^a,[x,T^{\prime\prime}_B]F^b,T^{\prime\prime\prime}_CF^c]+\sum_{(A,B,C,D)\in\Lambda}T_D\mathfrak{T}[T^\prime_AF^a,T^{\prime\prime}_BF^b,[x,T^{\prime\prime\prime}_C]F^c]\\
&\qquad+\mathfrak{T}_\Lambda[xF^a,F^b,F^c]+\mathfrak{T}_\Lambda[F^a,xF^b,F^c]+\mathfrak{T}_\Lambda[F^a,F^b,xF^c]\ .
\end{split}\end{equation}
By simple calculation, we have
$$[x,Q_A]=A^{-1}Q'_A\ .$$
We notice that if $Q_A$ is an LP-family, $Q'_A$ is also an LP-family, then $[x,T_A]$ is also an admissible transformation. Thus, we may consider $x\mathfrak{T}_\Lambda[F^a,F^b,F^c]$ as the following summation
\begin{equation}\label{h0}
\mathfrak{T}_\Lambda[F^a,F^b,F^c]+\mathfrak{T}_\Lambda[xF^a,F^b,F^c]+\mathfrak{T}_\Lambda[F^a,xF^b,F^c]+\mathfrak{T}_\Lambda[F^a,F^b,xF^c]\ ,
\end{equation}
then $\Vert x\mathfrak{T}_\Lambda[F^a,F^b,F^c]\Vert_{L^2}$ follows from \eqref{LB}.

\medskip
\noindent {\bf 2. The $H^N$ component of $S$ norm.} We will use the equivalent definition of $H^N$ norm, 
$$\Vert F\Vert_{H^N}^2:=\sum_{M \text{ dyadic}}M^{2N}\Vert P_M F\Vert_{L^2}^2\ ,$$
with $P_M$ as the Littlewood-Paley projections on $\mathbb{R}\times\mathbb{T}$ defined in Section 2. Then, we may decompose
\begin{equation*}
P_M\mathfrak{T}_\Lambda[F^a,F^b,F^c]=P_M\mathfrak{T}_{\Lambda,low}[F^a,F^b,F^c]+P_M\mathfrak{T}_{\Lambda,high}[F^a,F^b,F^c]\ ,
\end{equation*}
with$\mathfrak{T}_{\Lambda,low}[F^a,F^b,F^c]=\mathfrak{T}_\Lambda[P_{\leq M}F^a,P_{\leq M}F^b,P_{\leq M}F^c]$.

 We have firstly
\begin{equation}\label{h1}
\sum_{M \text{ dyadic}}M^{2N}\Vert P_M\mathfrak{T}_{\Lambda,high}[F^a,F^b,F^c]\Vert_{L^2}^2\lesssim K^2\max_{\{\alpha,\beta,\gamma\}=\{a,b,c\}}\Vert F^\alpha\Vert_{H^N}^2 \Vert F^\beta\Vert_{\mathcal{B}}^2\Vert F^\gamma\Vert_{\mathcal{B}}^2\ ,
\end{equation} 
since
\begin{equation}
\begin{split}
&\sum_M\vert M \vert^{2N}\Vert P_M\mathfrak{T}_\Lambda[P_{\geq 2M}F^a,F^b,F^c]\Vert_{L^2}^2\leq K^2\sum_M \vert M\vert^{2N} \Vert P_{\geq 2M}F^a\Vert_{L^2}^2\Vert F^b\Vert_{\mathcal{B}}^2\Vert F^c\Vert_{\mathcal{B}}^2\\
&\qquad\lesssim K^2\Vert F^a\Vert_{H^N}^2 \Vert F^b\Vert_{\mathcal{B}}^2\Vert F^c\Vert_{\mathcal{B}}^2\ .
\end{split}
\end{equation}

Let $Z\in\{\partial_x,\partial_y\}$, we can bound the contribution of $\mathfrak{T}_{\Lambda,low}$ as below
\begin{equation}
\begin{split}
&M^{N}\Vert P_M\mathfrak{T}_{\Lambda,low}\Vert_{L^2}\\
&\qquad\lesssim M^{-N}\Vert Z^{2N}P_M\mathfrak{T}_{\Lambda,low}[P_{\leq M}F^a,P_{\leq M}F^b,P_{\leq M}F^c]\Vert_{L^2}\\
&\qquad=M^{-N}\Vert \sum_{\alpha+\beta+\gamma\leq 2N}\sum_{M_1,M_2,M_3\leq M}P_M\mathfrak{T}_{\Lambda,low}[Z^{\alpha}P_{ M_1}F^a,Z^{\beta}P_{M_2}F^b,Z^{\gamma}P_{M_3}F^c]\Vert_{L^2}\ .
\end{split}
\end{equation}
Without loss of generality, we assume $M_1=\max{(M_1,M_2,M_3)}\leq M$, then 
\begin{equation}
\begin{split}
M^{N}\Vert P_M\mathfrak{T}_{\Lambda,low}\Vert_{L^2}&\lesssim \sum_{M_1\leq M}M^{-N}M_1^{2N}\sum_{M_2,M_3\leq M_1}\Vert \mathfrak{T}_{\Lambda,low}[P_{ M_1}F^a,P_{M_2}F^b,P_{M_3}F^c]\Vert_{L^2}\\
&\lesssim K \sum_{M_1\leq M}(\frac{M_1}{M})^{-N}M_1^{N}\Vert P_{ M_1}F^a\Vert_{L^2}\Vert F^b\Vert_{\mathcal{B}}\Vert F^c\Vert_{\mathcal{B}}\ ,
\end{split}
\end{equation}
the above sum is in $\ell_M^2$ by Schur test, then
\begin{equation}\label{h2}
\sum_{M \text{ dyadic}}M^{2N}\Vert P_M\mathfrak{T}_{\Lambda,low}[F^a,F^b,F^c]\Vert_{L^2}^2\lesssim K^2\max_{\{\alpha,\beta,\gamma\}=\{a,b,c\}}\Vert F^\alpha\Vert_{H^N}^2 \Vert F^\beta\Vert_{\mathcal{B}}^2\Vert F^\gamma\Vert_{\mathcal{B}}^2\ .
\end{equation}
Therefore we bound the $H^N$ component of $S$ norm, which completes the estimate \eqref{SSS}.

\medskip
Now, we turn to prove the estimate \eqref{SSS+}, due to the definition of $S^+$ norm, we only need to bound $\Vert x\mathfrak{T}_{\Lambda}\Vert_{S}$ and $\Vert (1-\partial_{xx})^4\mathfrak{T}_{\Lambda}\Vert_S$. From \eqref{h0} and \eqref{SSS}, we gain directly
\begin{equation}\begin{split}
&\Vert x\mathfrak{T}_{\Lambda}[F^a,F^b,F^c]\Vert_{S}\\
&\qquad\lesssim \max_{\{\alpha,\beta,\gamma\}=\{a,b,c\}}\Vert F^{\alpha}\Vert_{S^+}\Vert F^\beta\Vert_{\mathcal{B}}\Vert F^\gamma\Vert_{\mathcal{B}}+\max_{\{\alpha,\beta,\gamma\}=\{a,b,c\}}\Vert F^{\alpha}\Vert_{S}\Vert xF^\beta\Vert_{\mathcal{B}}\Vert F^\gamma\Vert_{\mathcal{B}}\ ,
\end{split}\end{equation}
we then using \eqref{YB} to control the norm $\Vert xF\Vert_{\mathcal{B}}$. The estimate on $\Vert (1-\partial_{xx})^4\mathfrak{T}_{\Lambda}\Vert_S$ can be calculated similarly by replacing $x$ with $(1-\partial_{xx})^4$. The proof is completed.
\end{proof}

\begin{remark}\label{leibniz}\mbox{}\\
We have a Leibniz rule for $\mathcal{I}^t[f,g,h]$ and $\mathcal{N}^t[F,G,H]$, namely
\begin{equation}\label{leib}
\begin{split}
&Z\mathcal{I}^t[f,g,h]=\mathcal{I}^t[Zf,g,h]+\mathcal{I}^t[f,Zg,h]+\mathcal{I}^t[f,g,Zh],\quad Z\in\{x,\partial_x\}\ ,\\
&Z\mathcal{N}^t[F,G,H]=\mathcal{N}^t[ZF,G,H]+\mathcal{N}^t[F,ZG,H]+\mathcal{N}^t[F,G,ZH]\ ,\quad Z\in\{x,\partial_x,\partial_y\}\ .
\end{split}
\end{equation}
Property \eqref{leib} will be of importance in order to ensure the hypothesis of  the transfer principle displayed by Lemma~\ref{SS+}.
\end{remark}

\end{document}